\documentclass[a4paper,11pt,reqno]{amsart}
\usepackage[a4paper,top=3cm,bottom=2cm,left=3cm,right=3cm,marginparwidth=1.75cm]{geometry}
\usepackage[english]{babel}
\usepackage{amsfonts, amsmath, amssymb, amsthm}
\usepackage{hyperref}
\usepackage{cleveref}
\usepackage{bm}
\usepackage{tikz-cd}
\usetikzlibrary{arrows}
\usepackage[OT2,T1]{fontenc}
\DeclareSymbolFont{cyrletters}{OT2}{wncyr}{m}{n}
\DeclareMathSymbol{\Sha}{\mathalpha}{cyrletters}{"58}

\newcommand{\N}{\mathbb{N}}\newcommand{\Z}{\mathbb{Z}}
\newcommand{\Q}{\mathbb{Q}}
\newcommand{\R}{\mathbb{R}}
\newcommand{\C}{\mathbb{C}}
\newcommand{\F}{\mathbb{F}}
\newcommand{\G}{\mathbb{G}}
\newcommand{\A}{\mathbb{A}}
\newcommand{\Proj}{\mathbb{P}}

\DeclareMathOperator{\Hom}{Hom}
\DeclareMathOperator{\coker}{coker}
\DeclareMathOperator{\Pic}{Pic}

\DeclareMathOperator{\NS}{NS}

\DeclareMathOperator{\Eff}{\overline{Eff}}
\DeclareMathOperator{\Gal}{Gal}
\DeclareMathOperator{\Spec}{Spec}
\DeclareMathOperator{\Br}{Br}
\DeclareMathOperator{\HH}{H}
\DeclareMathOperator{\Ext}{Ext}

\DeclareMathOperator{\an}{an}
\DeclareMathOperator{\Disc}{Disc}
\DeclareMathOperator{\Res}{Res}

\DeclareMathOperator{\sep}{sep}
\DeclareMathOperator{\adj}{adj}

\renewcommand{\Re}{\mathrm{Re}}

\newtheorem{theorem}{Theorem}[section]
\newtheorem{proposition}[theorem]{Proposition}
\newtheorem{corollary}[theorem]{Corollary}
\newtheorem{lemma}[theorem]{Lemma}

\theoremstyle{definition}
\newtheorem{definition}[theorem]{Definition}
\newtheorem{conjecture}[theorem]{Conjecture}
\newtheorem{example}[theorem]{Example}
\newtheorem{remark}[theorem]{Remark}
\newtheorem{construction}[theorem]{Construction}

\title{Manin's conjecture for integral points on toric varieties}
\author{Tim Santens}
\address{
	University of Cambridge \\ 
	DPMMS \\
	Centre for Mathematical Sciences\\
	Wilberforce Road \\
	Cambridge \\
	CB3 0WB \\ UK}
\email{ts996	@cam.ac.uk}
\begin{document}
\begin{abstract}
	We formulate a conjecture on the number of integral points of bounded height on log Fano varieties in analogy with Manin's conjecture on the number of rational points of bounded height on Fano varieties. We also give a prediction for the leading constant which is similar to Peyre's interpretation of the leading constant in Manin's conjecture. We give evidence for our conjecture by proving it for toric varieties. The proof is based on harmonic analysis on universal torsors.
\end{abstract}
\thanks{The author was supported during the creation of this article by a PhD fellowship of FWO-Vlaanderen with grant number 11I0621N and by the Herschel-Smith fund.}
\maketitle
\tableofcontents
\section{Introduction}
Let $F$ be a number field with ring of integers $\mathcal{O}_F$. Let $X$ be a geometrically integral smooth proper variety over $F$. It is expected that if $X$ is \emph{Fano}, i.e.~has ample anticanonical divisor $- K_X$, then the set of rational points $X(F)$ is either empty or large, both in a qualitative and a quantitative sense.

In this paper we will focus on the quantitative properties of $X(F)$. To do this one fixes a \emph{height function} $H: X(F) \to \R_{>0}$ and studies the number of rational points of bounded height $\# \{P \in X(F): H(P) \leq T\}$ as $T \to \infty$. If $X$ is a Fano variety then the most natural height is the one corresponding to the anticanonical divisor. Manin \cite{Manin1989Rational} conjectured a precise asymptotic for this counting problem if $H$ is an anticanonical height. Batyrev and Manin \cite{Batyrev1990Nombre} later extended this conjecture to general height functions.

Batyrev and Tschinkel \cite{Batyrev1995Rational, Batyrev1998Manin, Batyrev1996Height} have studied this problem when $X$ is a toric variety. They proved a precise asymptotic formula, which agrees with Manin's conjecture, using harmonic analysis on the group of adelic points of the torus. Chambert-Loir and Tschinkel \cite{Chambert-Loir2010Integral} attempted to extend Batyrev and Tschinkel's methods to count integral points on toric varieties, their argument unfortunately has certain issues. Indeed, their main theorem is false, a counterexample was recently given by Wilsch \cite{Wilsch2022Integral}.

The goal of this paper is to prove a correct asymptotic for the number of integral points on toric varieties. Our main theorem is the following.
\begin{theorem}\label{thm:Main theorem all integral points}
 	Let $X$ be a smooth proper toric variety over $F$ with open subtorus $T \subset X$. Let $D \subset X \setminus T$ be a divisor and let $U := X \setminus D$. Assume that $\overline{F}[U]^{\times} = \overline{F}^{\times}$. Let $H: X(F) \to \R_{> 0}$ be a height corresponding to a big divisor.
 	
 	Let $\mathcal{U}$ be an integral model of $U$ over $\mathcal{O}_F$ such that $T(F) \cap \mathcal{U}(\mathcal{O}_F) \neq \emptyset$. There exists $b \in \N$ and $a, c > 0$ such that as $T \to \infty$ we have
 	\[
 	\#\{P \in T(F) \cap \mathcal{U}(\mathcal{O}_F): H(P) \leq T\} \sim c T^a (\log T)^b.
 	\]
\end{theorem}
Moreover, we give a geometric interpretation of $a,b$ and $c$ in the style of Manin's conjecture. Let us note that this is essentially Chambert-Loir and Tschinkel's result, except that the logarithmic exponent $b$ may be different.

Based on this result we make Conjecture \ref{conj:Integral Manin conjecture}, a conjecture on the asymptotics of the number of points of bounded height on log Fano varieties analogous to Manin's conjecture.
\subsection{Manin's conjecture}
Let $X$ be a smooth projective Fano variety such that $X(F) \neq \emptyset$ and $H: X(F) \to \R_{> 0}$ an anticanonical height. The modern form of Manin's conjecture states that there exists a positive constant $c$ and a thin subset $Z \subset X(F)$, i.e. $Z$ is contained in a finite union of images of $Y(F) \to X(F)$, where $Y$ is integral and $Y \to X$ has no birational section, such that
\begin{equation*}
	\# \{P \in X(F) \setminus Z: H(P) \leq T\} \sim c T (\log T)^{b - 1}
\end{equation*}
as $T \to \infty$. Here $b := \text{rk} \Pic X$ is the rank of the Picard group $\Pic X$.

Peyre has a conjectural interpretation \cite{Peyre1995Hauteurs} of the leading constant $c$. It is the product $c = \frac{\theta}{a (b - 1)!} \cdot \beta \cdot \tau(X(\A_F)^{\Br})$ where $\theta$ is \emph{Peyre's effective cone constant}, a combinatorial invariant depending on the \emph{pseudo-effective cone} $\Eff_X \subset (\Pic X)_{\R}$, the closure of the cone generated by the effective divisors, $\beta = \# \HH^1(F, \Pic X_{\overline{F}}) = \# \Br_1 X/\Br F$ is a cohomological invariant, $\tau$ is the \emph{Tamagawa measure} on the space of adelic points $X(\A_F)$, this measure depends on the chosen height function $H$, and $X(\A_F)^{\Br} \subset X(\A_F)$ is the Brauer-Manin set.

Manin's conjecture is now known in a variety of different cases, but even the case of smooth Fano surfaces, known as del Pezzo surfaces, remains far out of reach. For example, there is not a single smooth cubic surface for which Manin's conjecture has been proven. 
\subsection{Integral points}
In this paper we will study the closely related problem of counting \emph{integral points} on varieties. For this one fixes a boundary divisor $D \subset X$ which we will always assume to have geometrically simple normal crossings, i.e.~the base change to the algebraic closure $D_{\overline{F}} \subset X_{\overline{F}}$ has simple normal crossings. One then expects that for any integral model $\mathcal{U}$ of $U := X \setminus D$ the set of integral points $\mathcal{U}(\mathcal{O}_F)$ is either empty or large as long as the pair $(X, D)$ is \emph{log Fano}, i.e.~that the log anticanonical divisor $-K_{(X,D)} := -K_X - D$ is an ample divisor on $X$. 

The number of integral points of bounded height has been studied in a large number of different situations, some of the more recent results are \cite{Chambert2012Integral, Takloo2013Integral, Chow2019Distribution, Wilsch2022Integral,Derenthal2022Integral, Wilsch2023Integral}, but a general Manin-type conjecture has so far remained elusive. Chambert-Loir and Tschinkel have heuristics for what the logarithmic power should be on their analysis of the analytic behavior of Igusa integrals \cite[\S4.4]{Chambert-Loir2010Igusa}. They also construct the correct Tamagawa measures in loc.~cit. Wilsch recently showed that one should take certain analytic obstructions into account and made a suggestion for what the logarithmic power should be \cite[\S2.5]{Wilsch2022Integral} (but only in the case that $X$ is split). Chambert-Loir and Tschinkel's heuristics, together with Wilsch's obstruction, also suggest a prediction for the leading constant as long as the Brauer group is trivial.

So far all known cases in the literature of Manin's conjecture for integral points have trivial Brauer group. It was thus impossible to know what influence the Brauer group has on the leading constant. Toric varieties form a large class of varieties with interesting geometric and arithmetic structure, for example the Brauer group can be non-trivial. Based on Theorem \ref{thm: Main theorem for general heights} we are able to extend the above heuristics to make a conjecture in large generality. We remark that this mirrors the history of Manin's conjecture for rational points, where the presence of the cohomological factor $\beta$ was only understood after the work of Batyrev and Tschinkel on toric varieties \cite{Batyrev1998Manin}.

Let us now discuss some of the new phenomena which appear when studying integral points instead of rational points.

Integral points are concentrated near the boundary divisor $D$ in the analytic topology at the archimedean places, as can already be seen in the simplest case where $(X,D) = (\Proj^1, \infty)$. Moreover, it turns out that they are concentrated near the strata of maximal codimension of $D$. The appropriate Tamagawa measures thus have to be supported on these strata at the archimedean places. Such Tamagawa measures were defined by Chambert-Loir and Tschinkel in \cite[\S2.1.12]{Chambert-Loir2010Igusa}. 

The collection of strata has a combinatorial structure, following Chambert-Loir and Tschinkel \cite[\S3.1]{Chambert-Loir2010Igusa} we encode this structure via the \emph{analytic Clemens complex} $\mathcal{C}^{\an}_{\Omega_{F}^{\infty}}(D)$. A \emph{face} $\mathbf{A} \in \mathcal{C}^{\an}_{\Omega_{F}^{\infty}}(D)$ of the Clemens complex essentially consist of a choice of stratum $Z_v \subset D_{F_v}$ for each archimedean place $v$. Let $Z_v^{\circ}$ be the complement in $Z_v$ of all the proper substrata. The \emph{dimension} $\dim(\mathbf{A})$ of a face $\mathbf{A}$ is a number which is closely related to the codimension of the strata $Z_v$, see \S3.1 for the precise definitions.

Wilsch's counterexample \cite{Wilsch2022Integral} to Chambert-Loir and Tschinkel's main result \cite{Chambert-Loir2010Integral} is due to a so-called \emph{analytic obstruction} to the Zariski density of integral points near certain strata. Wilsch only discusses such analytic obstructions in the split case, but a minor modification of his argument shows that these obstructions also exist in the non-split case, see Theorem \ref{thm: Non-constant regular functions imply failure of Zariski density}.

Our solution to dealing with these analytic obstructions is as follows. Because integral points tend to be concentrated around faces of the Clemens complex we will study the distribution of integral points of bounded height near each face separately. We will do this as follows: for each archimedean place $v$ let $U_{Z_v} \subset X_{F_v}$ be the minimal open subvariety of $X_{F_v}$ which is a union of open strata and contains $Z_v^{\circ}$. We can then count integral points near $\mathbf{A}$ by weighting by a continuous compactly supported function $f_v: U_{Z_v}(F_v) \to \R_{\geq 0}$. One can always reduce counting all integral points to this case by a partition of unity argument. We will then only count integral points near those faces where there is no analytic obstruction, since the integral points are not Zariski dense otherwise.

For general log Fano varieties these analytic obstructions are still insufficient to explain all pathological behavior, see \S\ref{sec:Counterexample when there are global invertible sections} for an example. In that case the issue is due to the existence of non-zero global invertible sections on $U$. This is one reason why we will assume that $\overline{F}[U]^{\times} = \overline{F}^{\times}$, implicit in our conjecture will be that under these assumptions the analytic obstructions explain all pathological behavior.

This assumption can also be explained using the following geometric perspective. If $\overline{F}[U]^{\times} \neq \overline{F}^{\times}$ then there exists a dominant map $f:U \to T$ to a non-trivial torus $T$. It is then natural to count rational points on each fiber of $f$ and sum the contributions. But a torus is a log Calabi-Yau variety so counting points on it should be harder than on a log Fano variety.
\subsection{Main results}
We will deduce Theorem \ref{thm:Main theorem all integral points} from a more precise result which gives an asymptotic formula for the number of integral points near each face separately. We state this theorem only for the anticanonical height for now, for more general heights the interpretation of the logarithmic power and the leading constant is more complicated.
\begin{theorem}\label{thm:Main theorem log anticanonical height}
	Let $X$ be a smooth proper toric variety over $F$ with open subtorus $T \subset X$. Let $D \subset X \setminus T$ be a toric divisor and let $U := X \setminus D$. Assume that $\overline{F}[U]^{\times} = \overline{F}^{\times}$. Let $H: X(F) \to \R_{> 0}$ be a log anticanonical height.
	
	Let $\mathbf{A} \in \mathcal{C}^{\an}_{\Omega_{F}^{\infty}}(D)$ be a face of the analytic Clemens complex and assume that there is no analytic obstruction to Zariski density of integral points near $\mathbf{A}$. 
	
	Let $\mathcal{U}$ be an integral model of $U$ and assume that $\mathcal{U}(\mathcal{O}_F) \cap T(F) \neq \emptyset$. There exists a non-empty closed subset $(\prod_{v \in \Omega_F^{\emph{fin}}} \mathcal{U}(\mathcal{O}_v) \times \prod_{v \in \Omega_F^{\infty}} Z_v^{\circ}(F_v))^{\Br} \subset \prod_{v \in \Omega_F^{\emph{fin}}} \mathcal{U}(\mathcal{O}_v) \times \prod_{v \in \Omega_F^{\infty}} Z_v^{\circ}(F_v)$ with the following property.
	
	For each archimedean place $v$ let $f_v:U_{Z_v}(F_v) \to \R_{\geq 0}$ be a continuous function with compact support $\emph{supp}(f_v)$ such that
	\[
	\Big(\prod_{v \in \Omega_F^{\emph{fin}}} \mathcal{U}(\mathcal{O}_v) \times \prod_{v \in \Omega_F^{\infty}} Z_v^{\circ}(F_v)\Big)^{\Br} \cap \prod_{v \in \Omega_F^{\emph{fin}}} \mathcal{U}(\mathcal{O}_v) \times \prod_{v \in \Omega_F^{\infty}} \emph{supp}(f_v) \neq \emptyset.
	\]
	There exists a $c > 0$ such that
	\[
	\sum_{\substack{P \in \mathcal{U}(\mathcal{O}_F) \cap T(F) \\ H(P) \leq T}} \prod_{v \in \Omega_{F}^{\infty}} f_v(P) \sim c T (\log T)^{\emph{rk} \Pic U + \dim \mathbf{A}}.
	\]
	
	Moreover, the constant $c$ has a Peyre-type interpretation.
\end{theorem}
One recovers Theorem \ref{thm:Main theorem all integral points} for the anticanonical height by a partition of unity argument and putting $b$ to be equal to maximum of $\text{rk} \Pic U + \dim \mathbf{A}$ as $\mathbf{A}$ ranges over all faces which have no analytic obstruction.

As the notation suggest $\big(\prod_{v \in \Omega_F^{\text{fin}}} \mathcal{U}(\mathcal{O}_v) \times \prod_{v \in \Omega_F^{\infty}} Z_v^{\circ}(F_v)\big)^{\Br}$ is a Brauer-Manin set. The relevant Brauer group $\Br_1 (X ; \mathbf{A})$ is a subgroup of the algebraic Brauer group $\Br_1 U$ whose index is a power of $2$. The group $\Br_1 (X ; \mathbf{A})$ depends on the face $\mathbf{A}$. This is another reason why it is crucial to work locally around each face separately. 

Let us remark that the Peyre-type interpretation of the constant $c$ is rather subtle. In general the Brauer group $\Br_1 (X ; \mathbf{A})/ \Br F$ can be infinite, in which case the Tamagawa measure of $(\prod_{v \in \Omega_F^{\text{fin}}} \mathcal{U}(\mathcal{O}_v) \times \prod_{v \in \Omega_F^{\infty}} Z_v(F_v))^{\Br}$ is $0$. But in this case the analogue of $\beta = \# \HH^1(F, \Pic X_{\overline{F}})$ will be infinite. It follows that the naive generalization of Peyre's constant is $\infty \cdot 0$. See \S\ref{sec:Example with infinite Brauer group} for an example.

We will resolve this as follows. By work of Salberger \cite[Assertion~5.25]{Salberger1998Tamagawa} Peyre's constant for $X$ is equal to a sum of adelic volumes on certain auxiliary varieties, the \emph{universal torsors} of $X$. This definition of Peyre's constant generalizes well even in the case that $\Br_1 (X ; \mathbf{A})/\Br \Q$ is infinite, at least if the universal torsors explain all obstructions to strong approximation.

It turns out that one can still relate our generalization of Salberger's formula for Peyre's constant to the Tamagawa measure on $U$. It involves a certain (possibly infinite) sum over $\Br_1 (X ; \mathbf{A})/\Br \Q$.

Let us note that one can use the method of Batyrev and Tschinkel to show that rational points $X(F)$ are dense in $X(\A_F)^{\Br}$ \cite[Cor.~3.10.4]{Chambert-Loir2010Integral}. Similarly, because the $f_v$ and $\mathcal{U}$ can be freely chosen in Theorem \ref{thm:Main theorem log anticanonical height}, we can deduce the following approximation result.
\begin{theorem}\label{thm:Strong approximation in introduction}
	If the assumptions of Theorem \ref{thm:Main theorem log anticanonical height} are satisfied then one has that the closure of $\mathcal{U}(\mathcal{O}_{F})$ in $\prod_{v \in \Omega_F^{\emph{fin}}} \mathcal{U}(\mathcal{O}_v) \times \prod_{v \in \Omega_F^{\infty}} U_{Z_v}(F_v)$ contains $(\prod_{v \in \Omega_F^{\emph{fin}}} \mathcal{U}(\mathcal{O}_v) \times \prod_{v \in \Omega_F^{\infty}} Z_v^{\circ}(F_v))^{\Br}$.
	
	We say that the algebraic Brauer-Manin obstruction is the only obstruction to strong approximation for $(X ; \mathbf{A})$.
\end{theorem}

We can compare this with a theorem of Cao and Xu \cite{Cao2018Strong} which states that the algebraic Brauer-Manin obstruction is the only obstruction to strong approximation off all infinite places for all smooth toric varieties and also with a theorem of Wei \cite{Wei2021Strong} which essentially states that if $\overline{F}[U]^{\times} = \overline{F}^{\times}$ then it is the only obstruction to strong approximation off a single infinite place. The assumptions in Theorem \ref{thm:Strong approximation in introduction} are stronger, but the conclusion is stronger since we are also able to approximate at the infinite places. For example, if $\HH^0(U, \mathcal{O}_U) = F$ then we may take $\mathbf{A} = \emptyset$ in which case the above theorem implies that $U(F)$ is dense in $U(\A_F)^{\Br_1}$.

Let us also note that $\Br U$ is not necessarily equal to $\Br_1 U$, see the example in \S:\ref{sec:Example transcendental Brauer group}. It is not entirely clear to the author why transcendental elements of the Brauer group do not play a role, both for strong approximation and in the leading constant.
\subsection{Methods}
We will combine two methods which are both often used separately to count points on varieties. The first is the universal torsor method and the second is harmonic analysis. Let us be go into some more detail.

Universal torsors of a variety $X$ are a notion introduced by Colliot-Th\'elène and Sansuc \cite{Colliot1987Descente} to study the Brauer-Manin obstruction. It is a torsor $Y \to X$ under the N\'eron-Severi torus $T_{\NS}$, the torus whose Galois module of characters is $\Pic X_{\overline{F}}$, satisfying a certain property. A single variety $X$ can have multiple universal torsors, but they are all twists of each other. The geometry and arithmetic of $Y$ is often simpler than the geometry and arithmetic of $X$.

In groundbreaking work \cite{Salberger1998Tamagawa} Salberger realized that universal torsors form a natural place to study Manin's conjecture. Rational points $P \in X(F)$ are in bijection with $T_{\NS}(F)$-orbits of rational points $Q \in Y(F)$ for all universal torsors $Y \to X$, but it is generally easier to count points on $Y$. Moreover, Manin's conjecture, including Peyre's rather mysterious constant, is equivalent to the statement that the number of $T_{\NS}(F)$ orbits on the universal torsor in height boxes is asymptotic to the volume of these height boxes (assuming that the algebraic Brauer-Manin obstruction is the only obstruction).

In most applications of the universal torsor method one works with explicit integral models $\mathcal{Y} \to \mathcal{X}$ of the universal torsors, in which case $\mathcal{X}(\mathcal{O}_F)$ is in bijection with $T_{\NS}(\mathcal{O}_F)$-orbits of integral points $\mathcal{Y}(\mathcal{O}_F)$ for all possible universal torsors $\mathcal{Y} \to \mathcal{X}$. This becomes especially nice when $T_{\NS}(\mathcal{O}_F)$ is finite, for example if $F = \Q$ and $T_{\NS}$ is split. In more general situations one has to fix a fundamental domain of the action of the finitely generated group $T_{\NS}(\mathcal{O}_F)$ which is usually rather technical and notationally cumbersome. We will avoid fixing a fundamental domain by counting orbits directly.

We remark that $T_{\NS}$ is not necessarily split for a toric variety. This makes finding a fundamental domain harder, this is why non-split varieties are often avoided or a variant, the split torsor method \cite{Derenthal2020Split}, is used instead. This is the first time that the universal torsor method is successfully used when $T_{\NS}$ is non-split.

Universal torsors of toric varieties are particularly simple, they are just the complement of certain explicit codimension $2$ subvarieties in affine space. Proving Manin's conjecture in this way reduces to a lattice point counting problem. In the case when $F = \Q$ this lattice point counting problem was solved by Salberger \cite[\S 11]{Salberger1998Tamagawa}. Universal torsors have also proven useful when counting integral points on varieties, see \cite{Wilsch2022Integral,Derenthal2022Integral,Wilsch2023Integral}.

A significant part of this paper is spent generalizing Salberger's theory. As a first step one notices that Salberger's theory does not generalize if one just uses the universal torsor of $U$. Indeed, an important ingredient of Salberger's theory is that the order of growth in Manin's conjecture is equal to the order of growth of certain increasing volumes on $T_{\NS}(\A_F)$. This is false when one studies integral points.

A crucial observation is that the above statement becomes true if one enlarges the adelic space $T_{\NS}(\A_F)$ at the archimedean places in a way which depends on the face $\mathbf{A}$. We will encode this by replacing $T_{\NS}$ by the novel notion of an \emph{adelic group of multiplicative type} $\mathbf{T}_{\NS}$, this is essentially a choice of a group of multiplicative type $T_v$ over $F_v$ for each place $v$, such that for all but finitely many places $T_v$ is the base change of a group of multiplicative type $T_F$ over $F$.

To generalize Salberger's theory we then need to replace $Y \to X$ by a $\mathbf{T}_{\NS}$-torsor. To do this we will introduce the notion of an \emph{adelic variety} and replace $X$ by the adelic variety $(X ; \mathbf{A})$ and $Y \to X$ by its \emph{adelic universal torsor} $\mathbf{Y} \to (X ; \mathbf{A})$ respectively. An adelic variety consists of a variety $X_v$ over $F_v$ for each place $v$, which is the base change of a variety $X_F$ over $F$ for all but finitely many $v$. The adelic variety $(X ; \mathbf{A})$ will depend on the choice of a face $\mathbf{A} \in \mathcal{C}^{\an}_{\Omega_{F}^{\infty}}$. The invariants of $X$ which play a role in Manin's conjecture, such as the pseudo-effective cone and the Brauer group, will be replaced by the analogous invariants of the adelic variety $(X ; \mathbf{A})$. The group $\Pic (X ; \mathbf{A})$ has been defined before by Wilsch \cite[Def. 2.2.1]{Wilsch2022Integral} in the split case. The definition of $\Br (X ; \mathbf{A})$ is completely new. These notions are introduced in \S3.

An extra difficulty which appears when counting integral points is that the group of multiplicative type $T_{\NS}$ is not necessarily connected. Tamagawa measures of tori play an important role in Salberger's work. We were unable to find a definition of such Tamagawa measures in the literature for non-connected algebraic groups, so we generalize this ourselves.

Another obstruction to stating a general Manin-type conjecture for integral points is that the relevant Brauer group modulo constants can be infinite. This implies that Peyre's interpretation of the leading constant is $\infty \cdot 0$. We resolve this by only considering finite subgroups of the Brauer group and taking the limit as we consider all subgroups. 

To show that this limit converges we give another interpretation of Peyre's constant due to Salberger as a sum of adelic volumes over each universal torsor. This is a (possibly infinite) sum over non-negative numbers which we show to converge absolutely.

To relate the limit to adelic volumes on the universal torsor we generalize Colliot-Th\'elène and Sansuc's theory of descent \cite{Colliot1987Descente} to adelic varieties to show that universal torsors and the algebraic Brauer-Manin pairing are dual in \S 6. The equality of the two versions of Peyre's constant then follows from an application of the Poisson summation formula to certain adelic cohomology groups.

A theoretical innovation which is not strictly required for our method but which simplifies certain arguments will be the notion of \emph{metrics} and their corresponding \emph{heights} on torsors. This generalizes the usual theory of metrics and heights on line bundles, via the equivalence between line bundles and $\G_m$-torsors. We study these notions in \S5. A metric on a universal torsor is essentially a compatible choice of metrics on all line bundles simultaneously. There exists a related notion in the literature, that of a system of metrics and their lifts to universal torsors \cite[\S4]{Peyre1998Terme}, but the notion of a metric on a torsor seems to us to be more natural.

We can then state our integral version of Manin's conjecture in \S7. We will also make a prediction for the leading constant, if the universal torsor is geometrically simple enough, in terms of volumes on universal torsors and then prove how it relates to volumes of Brauer-Manin sets.

After these theoretical preparations we start counting points on toric varieties in \S7. We first reduce our counting problem to counting $T_{\NS}(F)$-orbits of points on the universal torsors of $U$. These universal torsors naturally have the structure of toric varieties, whose underlying torus is a Weil restriction torus $\Theta$. Moreover, the action of $T_{\NS}$ factors through the action of $\Theta$ so it suffice to count elements of the group $\Theta(F)/T_{\NS}(F)$. Using a Tauberian theorem it suffices to show that the \emph{height zeta function} $Z(s, f) = \sum_{P \in \Theta(F)/T_{\NS}(F)} f(P) H(P)^{-s}$ has an appropriate meromorphic continuation.

Similarly the universal torsors of the adelic variety $(X ; \mathbf{A})$ naturally have the structure of an adelic toric variety whose adelic torus $\Theta_{\mathbf{A}}$ is a Weil restriction torus. The group $\Theta(F)/T_{\NS}(F)$ embeds as a discrete group into $\Theta_{\mathbf{A}}(\A_F)/T_{\NS}(F)$. 

We then compute the Fourier transform of a function $\Theta_{\mathbf{A}}(\A_F)/T_{\NS}(F) \to \C$ extending $f(P) H(P)^{-s}$, for this we use the computational tools developed in \cite{Chambert-Loir2010Igusa}. We remark that the factor in this Fourier transform which will eventually explain the order of growth is computed by an integral over the adelic points of the adelic N\'eron-Severi torus $\mathbf{T}_{\NS}$, see Lemma \ref{lem:Integral over Neron-Severi torus}.

We then apply Poisson summation to show that $Z(s, f)$ is equal to an integral over the Pontryagin dual $(\Theta_{\mathbf{A}}(\A_F)/\Theta(F))^{\vee}$ of this Fourier transform. This Pontryagin dual is an extension of a discrete group by a real vector space. This vector space has a canonical basis since $\Theta_{\mathbf{A}}$ is a Weil restriction torus. The integral over this vector space is thus an iterated contour integral which we can compute with the residue theorem by moving contours.

We remark that Chambert-Loir and Tschinkel similarly use harmonic analysis in \cite{Chambert-Loir2010Integral}, but they do this on $T(\A_F)$ where $T \subset X$ is the dense open torus. The main technical obstacle is then that the Pontryagin dual $(T(\A_F)/T(F))^{\vee}$ is an extension of a discrete group by a vector space without a canonical basis. To deal with the integral over this vector space they have to use a technical theorem \cite[Thm.~6.17]{Batyrev1998Manin} of Batyrev and Tschinkel. It is in the application of this theorem where the error is made. The existence of the canonical basis is the main reason why we prefer to work with the universal torsor.

It remains to determine the leading constant given by the characters contributing to the pole of highest order, this is carried out in \S 8. If $H$ is an anticanonical height then there is only a single character contributing to the main pole and the leading constant is seen to be the volume of an adelic subset of the universal torsor. For general heights, we apply Poisson summation to the sum over the contributing characters to identify the leading constant.

We remark that if one specializes this method to the case where $D = \emptyset$ one gets a new proof of Manin's conjecture for toric varieties. This proof does not use the technical theorem \cite[Thm.~6.17]{Batyrev1998Manin}, replacing it by the residue theorem. The cost for this analytic simplification is that the geometric set-up is more involved.

To conclude the introduction let us mention that the crucial geometric reason why our method works is that universal torsors of toric varieties are themselves (very simple) toric varieties and that the action of $T_{\NS}$ factors through the toric action. The only varieties with this property are toric varieties. But there are other geometric objects for which this holds, the so-called toric stacks. Manin's conjecture is known for rational points on split toric stacks by work of Darda-Yasuda \cite{Darda2023Toric}. The author expects that the methods of this paper can be generalized to deal with the non-split case. 
\subsection{Conventions and notations}
For every field $F$ we fix a separable closure $F^{\sep}$ and an algebraic closure $\overline{F}$. We let $\Gamma_F := \Gal(F^{\sep}/F)$ be its absolute Galois group. 

If $A$ is a $\Gamma_F$-set then we let $F[A]$ be the \'etale algebra corresponding to it via Grothendieck's formulation of Galois theory.

If $B \to A$ is a map and $N$ is an object living over $A$ then we will use the notation $N_B$ for the base change of $N$ to $B$. The precise meaning will depend on the context, e.g.~if $\mathcal{L}$ is a line bundle on a scheme $A$ then $\mathcal{L}_B$ is the pullback along $B \to A$.

All cohomology/fundamental groups are \'etale cohomology/fundamental groups. We will omit the base point since it is never relevant. If $G$ is a smooth group scheme over a scheme $X$ then $H^1(X, G)$ classifies $G$-torsors over $X$. We will freely use this identification.

If $U \subset X$ is a subset then we denote by $\mathbf{1}_U: X \to \C$ the indicator function of $U$.

If $V$ is a finite-dimensional vector space over a field $F$ and $V^* := \Hom(V, F)$ is the dual vector space then we will denote the perfect pairing $V \times V^* \to F$ by $\langle \cdot, \cdot \rangle_V$. 

By a \emph{measure} we always mean a Radon measure, i.e.~a functional on the space of compactly supported functions on a locally compact Hausdorff space $X$ sending non-negative functions to non-negative numbers.

Let $S^1 = \{z \in \C: |z| = 1\}$ be the complex unit circle, considered as a locally compact abelian group. For a locally compact abelian group $G$ we let $G^{\vee} := \Hom(G, S^1)$ be the Pontryagin dual.

Let $X$ be a locally compact Hausdorff space and $\mu$ a measure on $X$. If a discrete group $G$ acts on $X$ and the action is free and proper then the quotient $X/G$ is also locally compact Hausdorff. If the action preserves $\mu$ then we will abuse notation to denote the quotient measure on $X/G$ also by $\mu$. This means that for every compactly supported continuous function $f: X \to \C$ we have $\int_X f(x) d \mu(x) = \int_{X/G} \sum_{y \in xG} f(y) d \mu(xG)$.
\subsection*{Acknowledgments}
Firstly I would like to thank my supervisor Daniel Loughran for his continuous support while working on this paper and at other times. I am grateful to Tim Browning, Boaz Moerman, Marta Pieropan and especially Florian Wilsch for useful discussions. I would like to thank Marta Pieropan for inviting me to Utrecht in 2022 where I worked with Boaz Moerman on something which eventually (after quite a detour) lead to this paper. I am especially grateful to Tim Browning and Marta Pieropan for useful comments on a previous version of this preprint.
\section{Preliminaries}
\subsection{Number fields}
Let $F$ be number field or non-archimedean local field. We will denote its ring of integers by $\mathcal{O}_F$.

If $F$ is a number field then we let $\Omega_F$ be its set of places, $\Omega_F^{\infty}$ its set of archimedean places and $\Omega_F^{\text{fin}} := \Omega_F \setminus \Omega_F^{\infty}$ its set of non-archimedean places. For a place $v \in \Omega_F$ we write $F_v$ for the completion at $v$. For each place $v$ we fix an isomorphism $\overline{F} \cong \overline{F}_v$, inducing the decomposition subgroup $\Gamma_v := \Gamma_{F_v} \subset \Gamma_F$. If $v \in \Omega_F^{\text{fin}}$ is a non-archimedean place then we let $\mathcal{O}_v := \mathcal{O}_{F_v}$ be its ring of integers, $\pi_v$ a uniformizer, $\F_v := \mathcal{O}_v/ \pi_v \mathcal{O}_v$ the residue field and $q_v = \# \F_v$. 

For any place let $dx_v$ be the Haar measure on $F_v$. We normalize this in the standard way such that the induced product Haar measure $dx = \prod_v d x_v$ on the adeles $\A_F$ exists and $dx(\A_F/F) = 1$. To be precise $dx_v(\mathcal{O}_v) = \Disc(F_v)^{-\frac{1}{2}}$ where $\Disc(F_v)$ is the absolute discriminant of $F_v$ if $v$ is non-archimedean, $dx_v$ is the Lebesgue measure if $v$ is real and $dx_v$ is $2$ times the Lebesgue measure if $v$ is complex. We normalize all places $v$ such that $dx_v( a \Omega) = |a|_v dx_v(\Omega)$ for all $a \in F_v$ and measurable $\Omega \subset F_v$. This means that if $v$ is non-archimedean then $|\pi_v|_v = q_v^{-1}$, if $v$ is real then $|\cdot|_v$ is the absolute value on $\R$ and if $v$ is complex then $|\cdot|_v$ is the square of the absolute value on $\C$. 

A \emph{Hecke quasi-character} is a multiplicative continuous map $\chi: \G_m(\A_F) \to \C^{\times}$. It is a product of local quasi-characters $\chi_v: \G_m(F_v) \to \C^{\times}$ for all places $v$. Given such a local, resp. global, quasi-character we let $L_v(s; \chi_v)$, resp. $L(s; \chi)$ be the corresponding local, resp. global $L$-function. If $\chi = 1$ then we let $\zeta_F(s) := L(s ; 1)$ be the global Zeta function.

Given a finitely generated $\Gamma_v$, resp. $\Gamma_F$, module $M$ we denote by $L_v(s, M)$, resp. $L(s, M)$, the local, resp. global, Artin $L$-function of the Galois representation $M \otimes \Q$.

Let $\zeta^*_F(1)$, resp. $L^*(1, M)$, denote the value of of $\zeta_F(s)(s - 1)$, resp. $L(s, M)(s - 1)^{\text{rk} M^{\Gamma_F}}$, at $s = 1$.
\subsection{Twists of torsors}
The following is explained in more detail and generality in \cite[\S 2.2]{Skorogobatov2001Torsors}. Let $S$ be a scheme and $G$ a commutative smooth group scheme over $S$. Every right $G$-torsor $a:T \to S$ is also a left $G$-torsor \cite[p. 13]{Skorogobatov2001Torsors}. The inverse torsor $a^{-1}: T \to S$ has the same underlying scheme as $T$ but it is a left $G$-torsor via the action $g \cdot_{a^{-1}} t = g^{-1} \cdot_{a} t$.

Let $X$ be a scheme over $S$ and $\pi:Y \to X$ a left $G$-torsor. The \emph{twisted torsor} \cite[p.20, Ex. 2]{Skorogobatov2001Torsors} $\pi_{a}: Y_{a} \to X$ is the left $G$-torsor given by the contracted product $T \times^G Y := T \times Y/ G$. The action of $G$ is the diagonal one 
\begin{equation*}
	G \times_S T \times_S Y \to T \times_S Y : (g; t, y) \to (t \cdot_{a^{-1}} g^{-1}, g \cdot y) = (g \cdot_{a} t, g \cdot y)
\end{equation*}
The cohomology class of $\pi_{a}: Y_{a} \to X$ is equal to $[Y] - a \in H^1(X, G)$.

We can use torsors to study points on $X$ using the following lemma \cite[Thm.~8.4.1]{Poonen2017Rational}.
\begin{lemma}
	The images $\pi_{a}(Y_{a}(S))$ of the twists $\pi_{a}: Y_{a} \to X$ are pairwise disjoint and 
	\begin{equation*}
		X(S) = \coprod_{a \in H^1(S, G)}\pi_{a}(Y_{a}(S)).
	\end{equation*}
	\label{Points on the base and on the torsor}
\end{lemma}

Let $f: G \to G'$ be a group homomorphism of commutative group schemes over $S$ and let $Y' := G' \times^G Y$ be the balanced product. Let $a: T \to S$ be a left $G$-torsor. The image of $a \in \HH^1(S, G)$ along the pushforward $f_*: \HH^1(S, G) \to \HH^1(S, G')$ is given by the contracted product $a': T' := G' \times^G T \to S$. There is a canonical map
\begin{equation}
	Y_a = T \times^G Y \to G' \times^G (T \times^G Y) = (G' \times^G T) \times^{G'} (G' \times^G Y)= T' \times^{G'} Y' = Y'_{a'}.
	\label{Functoriality of points on torsors}
\end{equation}

\subsection{Group of multiplicative type} 
We recall the following definitions and constructions in modern language from \cite{Ono191Arithmetic}.
\begin{definition}
	Let $B$ be a scheme and $L$ an \'etale sheaf on $B$ which is \'etale locally a constant finitely generated group. The \emph{Cartier dual} of $N$ is the \'etale sheaf
	\begin{equation*}
		D(N) := \Hom_B(N, \G_m).
	\end{equation*}
	
	A group scheme $T$ over $B$ is a \emph{group of multiplicative type} if it isomorphic as an \'etale sheaf to $D(N)$ for some $N$. Such a group $N$ is canonically isomorphic to the sheaf of characters $X^*(T) := \Hom(T, \G_m)$.
	
	A group of multiplicative type $T$ is a \emph{torus} if $X^*(T)$ is \'etale locally free.
	
	A torus is \emph{split} if it is isomorphic to $\G_m^n$ for some $n$.
\end{definition}
The functors $D(\cdot)$ and $X^*(\cdot)$ form an anti-equivalence between the category of sheaves which are \'etale locally finitely generated and group schemes of multiplicative type.

Let $B$ be an integral normal scheme with fraction field $F$. It follows from \cite[\href{https://varieties.math.columbia.edu/tag/0DV5}{Tag 0DV5}]{stacks-project} that the category of sheaves which are \'etale locally finitely generated is equivalent to the category of finitely generated $\pi_1(B)$-modules. In this case we will write $X^*(T)$ for this $\pi_1(B)$-module and $\hat{T}$ for the group of characters $T \to \G_m$, which is equal to $X^*(T)^{\pi_1(B)}$.

Let $T$ be a group of multiplicative type over $F$. The above shows that there exists a group of multiplicative type $\mathcal{T}$ over $B$ with generic fiber $T$ if and only if the $\Gamma_F$-action on the group of characters $X^*(T_{F})$ factors through the quotient $\pi_1(B)$. The group of multiplicative type $\mathcal{T}$ is unique in this case, so we will also denote it by $T$. If the group of multiplicative type exists then we say that $T$ has \emph{good reduction} over $B$. The action of $\Gamma_F$ on $X^*(T)$ factors through a finite quotient so there always exists an open subscheme $B' \subset B$ such that $T$ has good reduction over $B'$. 

\begin{definition}
	Let $F$ be a local field with norm $|\cdot|$ and $T$ a group of multiplicative type with group of global characters $\hat{T}$.
	\begin{enumerate}
		\item The \emph{norm} is the map
		\begin{equation*}
			|\cdot|_T: T(F) \to \Hom(\hat{T}, \R_{>0}^{\times}): x \to (\chi \to |\chi(x)|).
		\end{equation*}
	
		There is a map $\C \otimes_{\Z} \R_{>0}^{\times} \to \C^{\times}: z \otimes r \to r^{z} = \exp(z\log r )$. We will also call the following composition the norm
		\begin{equation*}
			T(F) \xrightarrow{|\cdot|_T} \Hom(\hat{T}, \R_{>0}^{\times}) \subset \Hom(\hat{T} \otimes \C	, \C^{\times}).
		\end{equation*}
	
		Given $\chi \in \hat{T}_{\C}$ we will use the notation $|\cdot|_T^{\chi} := |\cdot|_T(\chi): T(F) \to \C^{\times}$.
		
		\item Let $\log: \Hom(\hat{T}, \R_{>0}^{\times}) \to \Hom(\hat{T}, \R) = \hat{T}^{*}_{\R}$ be the composition with the logarithm. The map $\log |\cdot|_T: T(F) \to \hat{T}^{*}_{\R}$ will be called the \emph{log norm}.
		\item The set of \emph{integral elements} $T(\mathcal{O}_F)$ is the kernel of the norm map.
	\end{enumerate}
	
	\begin{remark}
		If $T$ has good reduction over $\mathcal{O}_F$ and $\mathcal{T}$ is its integral model then we have $T(\mathcal{O}_F) = \mathcal{T}(\mathcal{O}_F)$. Indeed, after a base change we may assume that $T$ is split, in which case it is obvious.
	
		If $F$ is non-archimedean and its residue field has size $q$ then $\log |\cdot|_T$ factors through the lattice $\Hom(\hat{T}, (\log q )\Z) \subset \Hom(\hat{T}, \R)$. 
	
		If $F$ is archimedean then $|\cdot|_T$ is surjective.
	\end{remark}
\end{definition}
\begin{example}
	The following torus will play a crucial role. Let $A$ be a finite $\Gamma_F$-set and $F[A]$ the corresponding \'etale algebra, i.e.~every orbit $a \in A / \Gamma_F$ defines a field extension $F_a$ and $F[A] = \prod_{a \in A /\Gamma_F} F_a$. Let $\Z[A]$ be the $\Gamma_F$-module freely generated by $A$.
	
	Let $\G_m^A := \Res_{F[A]/F} \G_m$ be the Weil-restriction torus. It is the torus whose Galois module of characters is $\Z[A]$. 
\end{example}
We will frequently make use of the following well-known fact.
\begin{lemma}\label{lem: Shapiro's lemma}
	Let $F$ be any field and $A$ a finite $\Gamma_F$-set. Then $\HH^1(F, \Z[A]) = \HH^1(F, \G_m^A) = 0$.
\end{lemma}
\begin{proof}
	By Shapiro's lemma $\HH^1(F, \Z[A]) = \HH^1(F[A], \Z) = 0$. By Hilbert theorem 90 and Shapiro's lemma we have $\HH^1(F, \G_m^A) = \HH^1(F[A], \G_m) = 0$.
\end{proof}

\subsection{Toric varieties}
We recall some facts about toric varieties over non-separably closed fields, see also \cite[\S1]{Batyrev1995Rational}.
 
Let $X$ be a smooth toric variety over a field $F$ containing a rational point in the torus $T$, i.e.~$X$ is equipped with an action of $T$ and a $T$-equivariant embedding $T \subset X$. Let $M := X^*(T)$ be the $\Gamma_F$-module of characters of $T$ and $N := \Hom(M, \Z) = \Hom(\G_m, T)$ the $\Gamma_F$ module of cocharacters. The toric variety $X_{F^{\sep}}$ corresponds to a regular fan $\Sigma$ on $N$. The action of $\Gamma_{F^{\sep}}$ on $T$ extends to $X_{F^{\sep}}$. This implies that the $\Gamma_F$ action on $N$ preserves the fan $\Sigma$. We call such a fan \emph{$\Gamma_F$-invariant}
\begin{remark}
	The datum of the $T$-equivariant embedding $T \subset X$ is completely determined by the $\Gamma_F$-equivariant fan $\Sigma$. 
	
	There exists $\Gamma_F$-equivariant fans $\Sigma$ which do not have a corresponding toric variety over $F$ \cite[Thm. 1.31]{Hurugen2011Toric}. See \cite[Thm. 1.22]{Hurugen2011Toric} for a general criterion when the toric variety exists.
	
	There does always exist a corresponding toric algebraic space over $F$ since Galois descent is effective for algebraic spaces. Our arguments will also apply to toric algebraic spaces with no changes, but we will restrict ourselves to toric varieties to keep the presentation clearer.
\end{remark}

Let $\Sigma_{\min}$ be the $\Gamma_F$-set of rays of $\Sigma$. For each ray $\rho \in \Sigma_{\min}$ let $n_{\rho}$ be a generator of $\rho \cap N$ and $D_{\rho} \subset X$ the toric divisor corresponding to $\rho$. Consider the map $\Z[\Sigma_{\min}] \to N: \rho \to n_{\rho}$ and the dual map $M \to \Z[\Sigma_{\min}]$. This induces the following fundamental exact sequence of $\Gamma_F$-modules.
\begin{equation}\label{eq:Exact sequence invertible global sections and Picard group}
	0 \to F^{\sep}[X]^{\times}/F^{\sep, \times}\to M \to \Z[\Sigma_{\min}] \to \Pic X_{F^{\sep}} \to 0.
\end{equation}
The map $ \Z[\Sigma_{\min}] \to \Pic X_{F^{\sep}}$ sends $\rho$ to the divisor class $[D_{\rho}]$.

The effective cone $\Eff_{X_{F^{\sep}}} \subset (\Pic X_{F^{\sep}})_{\R}$  is generated by the image of $\Z_{\geq 0}[\Sigma_{\min}]$ and
\begin{equation}\label{eq:anticanonical divisor of a toric variety}
	- K_X = \sum_{\rho \in \Sigma_{\min}} [D_{\rho}].
\end{equation}

Let us note the following.
\begin{lemma}\label{lem:log anticanonical divisor is big}
	Let $\sigma$ be a cone and $D_{\sigma} = \sum_{\rho \in \sigma} D_{\rho}$. Then $- K_X - [D_{\sigma}] \in \Eff_{X_{F^{\sep}}}^{\circ}$.
\end{lemma}
\begin{proof}
	The $n_{\rho}$ for $\rho \in \sigma$ are linearly independent (because $X$ is smooth). There thus exists a $m \in M$ such that $m(n_{\rho}) > 0$ for all $\rho \in \sigma$.
	
	The image of $m$ in $\Z[\Sigma_{\min}]$ is the sum $\sum_{\rho \in \Sigma_{\min}} m(n_{\rho})\rho$. Fix an $\varepsilon > 0$ such that $\varepsilon \cdot m(n_{\rho}) > -1$ for all $\rho \in \Sigma_{\min}$. The exact sequence \eqref{eq:Exact sequence invertible global sections and Picard group} implies that 
	\[
	- K_X = \sum_{\rho \in \Sigma_{\min}, \rho \not \in \sigma} (1 + \varepsilon \cdot m(n_\rho))[D_{\rho}] + \sum_{\rho \in \sigma} \varepsilon \cdot m(n_{\rho}) [D_{\rho}].
	\]
	This implies that $- K_X$ lies in the image of $\R_{> 0}[\Sigma_{\min}]$. The lemma follows. 
\end{proof}

The following construction is based on \cite[Prop.~8.4]{Salberger1998Tamagawa}.
\begin{construction}\label{con:Construction of universal torsor}
	The $\Gamma_F$-set $\Sigma_{\min}$ corresponds to an \'etale algebra $F[\Sigma_{\min}]$ over $F$. Let $\A^{\Sigma_{\min}} := \Res_{F[\Sigma_{\min}]/F} \A^1_{F[\Sigma_{\min}]}$ be the Weil restriction. The Weil restriction torus $\G_m^{\Sigma_{\min}} := \Res_{F[\Sigma_{\min}]/F} \G_{m, F[\Sigma_{\min}]}$ acts on this variety in the obvious way and also equivariantly embeds into it. This makes $\A^{\Sigma_{\min}}$ a toric variety.
	
	The group of characters is $X^*(\Sigma_{\min}) \cong \Z[\Sigma_{\min}]$ as a $\Gamma_F$ module. The $\Gamma_F$-invariant fan corresponding to $\A^{\Sigma_{\min}}$ consists of the cones
	\begin{equation*}
		\Lambda_{\theta} := \prod_{\rho \in \theta} \R_{\geq 0}[\rho] \subset \R[\Sigma_{\min}]
	\end{equation*}
	for all subsets $\theta \subset \Sigma_{\min}$.
	
	Let $Y_{\Sigma} \subset \A^{\Sigma_{\min}}$ be the open toric subvariety corresponding to the subfan $\Sigma'$ consisting of all cones $\Lambda_{\theta}$ such that $\theta \subset \Sigma_{\min}$ is the set of rays of a cone of $\Sigma$. It exists because Galois descent is effective for open immersions. The map $\Z[\Sigma_{\min}] \to N$ induces a map of fans $\Sigma' \to \Sigma$. Let $\pi: Y_{\Sigma} \to X$ be the induced morphism of toric varieties, it exists because Galois descent is effective for morphisms.
	
	Let $T_{\NS} = D(\Pic X_{F^{\sep}})$ be the Cartier dual. The surjective map $\Z[\Sigma_{\min}] \to \Pic X_{F^{\sep}}$ induces an injective morphism $T_{\NS} \to \G_m^{\Sigma_{\min}}$. Equip the toric variety $Y_{\Sigma}$ with the $T_{\NS}$-action which is the inverse action\footnote{One has to take the inverse action as the torsor $\A^2 \setminus \{0\} \to \Proj^1$ corresponding to $\mathcal{O}_{\Proj^1}(1)$ is equipped with the action $\G_m \times \A^2 \setminus \{0\} \to \A^2 \setminus \{0\} : (\lambda ; x,y) \to (\lambda^{-1} x, \lambda^{-1} y)$.} to the one coming from this inclusion.
\end{construction}
\begin{proposition}\label{prop:Universal torsor of toric variety}
	Assume that $F^{\sep}[X]^{\times} = F^{\sep, \times}$. The $T_{\NS}$-action on $Y_{\Sigma}$ makes $\pi: Y_{\Sigma} \to X$ a $T_{\NS}$-torsor. It is a universal torsor.
\end{proposition}
We will discuss universal torsors in more detail in $\S 5$.
\begin{proof}
	The case when $X$ is a split proper toric variety is \cite[Prop. 8.5]{Salberger1998Tamagawa}. 
	
	The case when $X$ is split but non-proper can be proven analogously. It also follows from the proper case by Lemma \ref{lem:Universal torsor of open subvariety}.
	
	The general case follows since both properties can be checked after base change to $F^{\sep}$.
\end{proof}

The twists of this universal torsor are all isomorphic as varieties.
\begin{lemma}\label{lem:Twists of universal torsors of toric varieties}
	Let $a \in \HH^1(F, T_{\NS})$ and let $\pi_{a}: Y_{\Sigma, a} \to X$ be a twist of the $T_{\NS}$-torsor $\pi: Y_{\Sigma} \to X$. Then there exists an isomorphism $Y_{\Sigma, a} \cong Y_{\Sigma}$.
	\end{lemma}
\begin{proof}
	The action of $T_{\NS}$ on $Y_{\Sigma}$ factors through the action of $\G_m^{\Sigma_{\min}}$ so by functoriality of twists $Y_{\Sigma, a}$ is a twist of the variety $Y_{\Sigma, a}$ by a certain $\G_m^{\Sigma_{\min}}$-torsor. The proposition now follows since $\HH^1(F, \G_m^{\Sigma_{\min}}) = 0$ by Lemma \ref{lem: Shapiro's lemma}.
\end{proof}
\begin{remark}
	If one carefully expands the above proof one can find the following description of the map $\pi_{a}: Y_{\Sigma} \cong Y_{\Sigma, a} \to X$.
	
	The exact sequence $1 \to T_{\NS} \to \G_m^{\Sigma_{\min}} \to T \to 1$ of groups of multiplicative type leads to a surjective map $T(F) \to \HH^1(F, T_{\NS})$ since $\HH^1(F, \G_m^{\Sigma_{\min}})= 0$. Choose a lift $t \in T(F)$ of $a \in \HH^1(F, T_{\NS})$. The torus $T$ acts on $X$ and the map $\pi_a$ is given as the composition $Y_{\Sigma} \xrightarrow{\pi} X \xrightarrow{t^{-1} \cdot} X$. This map depends on the chosen lift $t$, different choices corresponding to different choices of an isomorphism $Y_{\Sigma} \cong Y_{\Sigma, a}$
\end{remark}
\subsection{\texorpdfstring{$\mathcal{X}$}{X}-functions of cones}
We will recall the notion of characteristic $\mathcal{X}$-functions of cones from \cite[\S 5]{Batyrev1998Manin} and slightly extend it to allow torsion in $A$.

\begin{definition}
	Let $A$ be a finitely generated abelian group with dual lattice $A^* := \Hom(A, \Z)$ and torsion subgroup $A_{\text{tors}}$. The Haar measure $d_A \mathbf{x}$ on $A_{\R}^{*} := \Hom(A, \R)$ is defined by normalizing the volume of $A^*_{\R}/ A^*$ to be $(\# A_{\text{tors}})^{-1}$.
\end{definition}
We normalize the measure this way because of the following lemma.
\begin{lemma}\label{lem:compatibility of measures on dual vector spaces in exact sequences}
	Let $0 \to A \to B \to C \to 0$ be an exact sequence of abelian groups. Then $0 \to C^*_{\R} \to B^*_{\R} \to A^*_{\R} \to 0$ is an exact sequence of $\R$-vector spaces and $d_{B} \mathbf{x} = d_{A} \mathbf{x} d_{C} \mathbf{x}$
\end{lemma}
\begin{proof}
	The exactness of $\R$-vector spaces is clear. 
	
	By applying $\Hom(\cdot, \Z)$ we get a long exact sequence 
	\[0 \to C^* \to B^* \to A^* \to \Ext^1(C, \Z) \to \Ext^1(B, \Z) \to \Ext^1(A, \Z) \to 0\]
	
	Note that $\# \Ext^{1}(A, \Z) = \# \Hom(A, \Q/\Z) = \# A_{\text{tors}}$ for any finitely generated abelian group $A$. So we find that $\# \coker(B^* \to A^*) = \# A_{\text{tors}} \# C_{\text{tors}}/\# B_{\text{tors}}$.
	
	By the snake lemma we get an exact sequence 
	\[0 \to C^*_{\R}/ C^* \to B^*_{\R}/ B^* \to A^*_{\R}/ \text{im}(B^*) \to 0.\]
	The measures are thus compatible if $d_B \mathbf{x}( B^*_{\R}/ B^*) = d_C \mathbf{x}(C^*_{\R}/ C^*) d_A \mathbf{x}(A^*_{\R}/ \text{im}(B^*))$. This is true because $d_A \mathbf{x}(A^*_{\R}/ \text{im}(B^*)) = \# \coker(B^* \to A^*) d_A \mathbf{x}(A^*_{\R}/A^*)$.
\end{proof}

Let $(A, A_{\R}, \Lambda)$ be a triple consisting of a finitely generated group, $A_{\R} := A \otimes \R$ an $\R$-vector space of dimension $k$ and $\Lambda \subset A_{\R}$ a convex $k$-dimensional cone such that $\Lambda \cap - \Lambda = 0$. Denote by $\Lambda^{\circ}$ the interior of $\Lambda$ and by $\Lambda^{\circ}_{\C} := \Lambda^{\circ} + i A_{\R}$. Let $(A^{*}, A_{\R}^*, \Lambda^*)$ be the triple consisting of the dual lattice $A^* := \Hom(A, \Z)$, the dual vector space $A^*_{\R} := \Hom(A_{\R}, \R)$ and the dual cone $\Lambda^* := \{f \in A^*_{\R} : f(\lambda) \geq 0 \text{ for all } \lambda \in \Lambda\}$.
\begin{definition}
	The \emph{$\mathcal{X}$-function} of the triple $(A, A_{\R}, \Lambda)$ is defined as the integral 
	\begin{equation*}
		\mathcal{X}_{\Lambda}(\bm{s}) := \int_{\Lambda^*} e^{- \langle \bm{s}, \mathbf{y} \rangle} d_{A}\mathbf{y}.
	\end{equation*}
	for $\bm{s} \in \Lambda_{\C}^{\circ}$.
\end{definition}

The function $\mathcal{X}_{\Lambda}(\bm{s})$ is a rational function in $\mathbf{s}$ \cite[Prop. 5.4]{Batyrev1998Manin} and thus has a meromorphic continuation to $A_{\C}$. The following lemma determines the order of its poles.
\begin{lemma}\label{Poles of characteristic functions of cones}
	Let $C \subset \Lambda$ be a face, $C_{\R} \subset A_{\R}$ the vector space generated by $C$, $C^{\circ} \subset C$ the interior of $C$ as a subspace of $C_{\R}$ and $C^{\circ}_{\C} := C^{\circ} + i C_{\R}$. Let $b(\ell)$ be the codimension of $C$.
	
	Let $\ell \in C^{\circ}_{\C}$ and $a \in \Lambda^{\circ}$. For each $a \in \Lambda^{\circ}$ the rational function $s^{b(\ell)} \mathcal{X}_{\Lambda}(\ell + as)$ in the complex variable $s$ is holomorphic for $\Re(s) \geq 0$.
	
	If $\ell = 0$ then $s^{\emph{rk} A} \mathcal{X}_{\Lambda}(as) = \mathcal{X}_{\Lambda}(a)$.
\end{lemma}
\begin{proof}
	If $\Lambda$ and thus $\Lambda^*$ are simplicial then this follows from the computation \cite[Prop. 5.3:(ii)]{Batyrev1998Manin}. In general we can cover $\Lambda^*$ by simplicial cones $\Lambda_{1}^*, \cdots, \Lambda_{n}^* \subset \Lambda^*$ which are dual to cones $\Lambda \subset \Lambda_{1}, \cdots, \Lambda_{n}$. We then by definition have $\mathcal{X}_{\Lambda}(\bm{s}) = \sum_i \mathcal{X}_{\Lambda_i}(\bm{s})$, cf.~the proof of \cite[Prop. 5.4]{Batyrev1998Manin}. 
	
	The face $C$ is contained in a face of $\Lambda_i$ of the same or smaller codimension since $\Lambda \subset \Lambda_i$ so we reduce to the simplicial case.
	
	The second statement follows from \cite[Prop.~5.3(i)]{Batyrev1998Manin}.
\end{proof}
\subsection{Poisson summation}
Let $H \subset G$ be a closed inclusion of locally compact abelian groups. Let $\mu$ and $\nu$ be Haar measures on $G$ and $H$ and let $\mu^{\vee}$, resp.~$\nu^{\vee}$, be the dual Haar measures on $G^{\vee}$, resp.~$H^{\vee}$. Let $\mu / \nu$ be the quotient measure on $G/H$ and $(\mu / \nu)^{\vee}$ the dual measure on $(G/H)^{\vee} = \ker(G^{\vee} \to H^{\vee})$.

Let $f: G \to \C$ be an $L^1$ function and 
\[\hat{f}: \hat{G} \to \C: \chi \to \int_{G} f(g) \chi(g) d \mu(g)\]
the Fourier transform of $f$.

In this paper we will use the following version of the Poisson summation formula. Note the similarity with the condition on $f$ and \cite[Cor.~3.36]{Bourqui2011Fonction}.
\begin{theorem}\label{thm:Poisson summation}
	Let $f: G \to \C$ be a continuous $L^1$ function such that $\hat{f}$ is $L^1$ when restricted to $(G/H)^{\vee}$. Assume one of the following conditions.
	\begin{enumerate}
		\item For all $u \in U$ in an open neighbourhood $U \subset G$ of the identity the integral $\int_{H} f(u h) d \nu(h)$ converges absolutely and is continuous in $u$ at $1$.
		\item There exists a continuous $L^1$ function $F: G \to \R_{\geq 0}$ whose Fourier transform $\hat{F}$ is $L^1$ when restricted to $(G/H)^{\vee}$ and such that there exists open neighborhood $U \subset G$ of the identity such that $|f(g)| \leq F(u g)$ for all $u \in U$ and $g \in G$.
	\end{enumerate}
	
	Then $f$ is $L^1$ when restricted to $H$ and 
	\begin{equation}\label{eq:Poisson summation formula}
		\int_{H} f(h) d\nu(h) = \int_{(G/H)^{\vee}} \hat{f}(\chi) d(\mu/\nu)^{\vee}(\chi).
	\end{equation}
\end{theorem}
\begin{proof}
	The first case is the usual Poisson summation formula \cite[Thm.~3.35]{Bourqui2011Fonction}. In the second case the Poisson summation formula implies that for all $u \in G$ outside of a measure $0$ subset one has that $H \to \R_{\geq 0}: h \to F(uh)$ is an $L^1$ function.
	
	Since $U$ is open such an $u \in U$ exists. Let $V \subset U$ be an open neighbourhood of $1$ such that $u \cdot V \subset U$. For all $v^{-1} \in V$ and $h \in H$ one then has that $|f(v h)| \leq F(v^{-1} u v h) = F(u h)$ since $v^{-1} u \in U$. The sum 
	\[
	\sum_{h \in H} f(v h)
	\]
	thus converges and is continuous by the dominated convergence theorem. By the first case this implies \eqref{eq:Poisson summation formula}.
\end{proof}

\section{Adelic families}
Let $F$ be a number field\footnote{Essentially everything in this section can be generalized to arbitrary global fields.}. When dealing with integral points on non-proper varieties we will consider families of algebro-geometric objects over $F_v$ for every place $v$, almost all of which come from an object over $F$. In this section we will introduce and study such families.
\begin{definition} \hfill
	\begin{enumerate}
		\item An \emph{adelic scheme over $F$} is a tuple $\mathbf{X} = (X_F, (X_v)_{v \in \Omega_F}, (\xi_{v})_{v \in \Omega_F})$ where $X_F$ is a scheme over $F$, $X_v$ is a scheme over $F_v$ and $\xi_{v}: (X_F)_{F_v} \to X_v$ is a morphism and an isomorphism for all but finitely many places $v$. 
		\item A \emph{morphism} $\mathbf{f}: \mathbf{Y} \to \mathbf{X}$ of adelic schemes is a tuple $(f_F, (f_v)_{v \in \Omega_v})$ where $f_F: X_F \to Y_F$ and $f_v: X_v \to Y_v$ are maps such that the following diagram commutes for all $v \in \Omega_F$
		\begin{equation*}
			\begin{tikzcd}
				Y_F & X_F \\
				Y_v & X_v
				\arrow["f_F", from=1-1, to=1-2]
				\arrow["\xi_{Y, v}"', from=1-1, to=2-1]
				\arrow["f_v", from=2-1, to=2-2]
				\arrow["\xi_{X, v}"', from=1-2, to=2-2]
			\end{tikzcd}
		\end{equation*}
		\item Let $\mathcal{P}$ be a property of schemes over fields stable under base change. An adelic scheme $\mathbf{X}$ is $\mathcal{P}$ if $X_F$ and all the $X_v$ are $\mathcal{P}$. 
		\item Let $\mathcal{P}$ be a property of morphisms of schemes over fields stable under base change. A morphism of adelic schemes $\mathbf{f}$ is $\mathcal{P}$ if $f_F$ and all the $f_v$ are $\mathcal{P}$. 
	\end{enumerate}
\end{definition}
\begin{remark} 
	We will use bold letters to denote adelic schemes and the same small letter with a subscript to denote its parts, e.g.~$\mathbf{Z} = (Z_F, (Z_v)_{v \in \Omega_F} (\xi_{Z, v})_{v \in \Omega_F})$. We will also not mention the $\xi_v$ when defining an adelic scheme if they are clear from the context.
	
	An adelic scheme $\mathbf{X}$ such that $X_F$ and all the $X_v$ are varieties will be called an \emph{adelic variety}.
\end{remark}
\begin{example}
	Let $X$ be a scheme over $F$. The corresponding \emph{constant} adelic scheme is given by $(X, (X_{F_v})_{v \in \Omega_F}, (\text{id}_{X_{F_v}})_{v \in \Omega_F})$. We will freely identify a scheme with the corresponding constant adelic scheme. 
	
	For any adelic scheme $\mathbf{X}$ there is a canonical map $X_F \to \mathbf{X}$ of adelic schemes whose local components are given by the $\xi_v$.
\end{example}
Non-trivial examples of adelic varieties will be given in the following subsections.

We have a good notion of adelic points on adelic varieties. Note that the following construction is the main reason for our definition of adelic varieties.
\begin{definition}
	Let $\mathbf{X}$ be an adelic variety over $F$, $S$ a finite set of places and $\mathcal{X}$ an $\mathcal{O}_{F, S}$-integral model of $X_F$. The \emph{space of adelic points of} $\mathbf{T}$ is the restricted product of topological spaces
	\begin{equation*}
		\mathbf{X}(\A_F) := \prod_{v \in S} X_v(F_v) \times \sideset{}{'}\prod_{v \not \in S} (X_v(F_v), \mathcal{X}(\mathcal{O}_v)) .
	\end{equation*}
	This is independent of the choice of integral model and $S$ for the usual reason.
	
	There is a diagonal map $X_F(F) \to \mathbf{X}(\A_F)$. This map is an injection since $X_F(F) \subset X_F(F_v) = X_v(F_v)$ for all but finitely many places $v$.
\end{definition}

The following cohomological invariants of adelic varieties will play a crucial role in what follows. Recall \cite{Colliot2021Brauer} that the Brauer group of a scheme $X$ is $\Br X := \HH^2(X, \G_m)$.
\begin{definition}
	Let $\mathbf{X}$ be an adelic variety over $F$.
	\begin{enumerate}
		\item A \emph{regular function} of $\mathbf{X}$ is a map $\mathbf{X} \to \A^1_f$. Equivalently it is a tuple $(f_v, (f_v)_{v \in \Omega_{F}})$ where $f_F \in \HH^0(X_F, \mathcal{O}_{X_F})$ and $f_v \in \HH^0(X_v, \mathcal{O}_{X_v})$ such that $\xi_v^{*}(f_v) = f_F$. Let $\HH^0(\mathbf{X}, \mathcal{O}_{\mathbf{X}})$ be the $F$-vector space of regular functions, where addition and scalar multiplication is done component-wise.
		
		\item The \emph{Picard group} $\Pic \mathbf{X}$ of $\mathbf{X}$ is the group of tuples $(\mathcal{L}_F, ( (\mathcal{L}_v)_{v \in \Omega_F})$ where $\mathcal{L}_F \in \Pic X_F$ and $\mathcal{L}_v \in \Pic X_v$ such that $\xi_v^*\mathcal{L}_v = (\mathcal{L}_{F})_{F_v} \in \Pic (X_F)_{F_v}$ for all $v \in \Omega_F$. Addition is component-wise.

		\item The \emph{Brauer group} $\Br \mathbf{X}$ of $\mathbf{X}$ is the group consisting of tuples $(A_F, (A_v)_{v \in \Omega_F})$ where $A_F \in \Br X_F$ and $A_v \in \Br X_{v}$ such that $\xi_v^* A_v = (A)_{F_v} \in \Br (X_{F})_{F_v}$ for all $v \in \Omega_F$. Addition is component-wise.
		
		The \emph{constant Brauer group} $\Br_0 \mathbf{X}$ is defined as the image of the natural map $\Br F \to \Br \mathbf{X}$. 
		
		The \emph{algebraic Brauer group} $\Br_1 \mathbf{X}$ is the subgroup of $\Br \mathbf{X}$ consisting of all tuples $(A_F, (A_v)_{v \in \Omega_F})$ such that $(A_v)_{\overline{F}_v} = 0 \in \Br (X_v)_{\overline{F}_v}$ for all $v \in \Omega_F$. This also implies that $(A_F)_{\overline{F}} = 0 \in \Br (X_{F})_{\overline{F}}$.
	\end{enumerate}
\end{definition}

We note that these definitions agree with the usual ones if we identify a variety with the corresponding adelic variety.
\begin{definition}
	An \emph{effective (Cartier) divisor} $\mathbf{D} \subset \mathbf{X}$ is an adelic subvariety such that $D_F \subset X_F$ and $D_v \subset X_v$ for all places $v \in \Omega_{F}$ is an effective (Cartier) divisor and $\xi_v^{-1}(D_v) = (D_F)_{F_v}$ for all $v \in \Omega_F$.
	
	Given an effective Cartier divisor $\mathbf{D} \subset \mathbf{X}$ we define the corresponding adelic line bundle by $\mathcal{O}(\mathbf{D}) = (\mathcal{O}(D), (\mathcal{O}(D_v))_{v \in \Omega_F})$.
	
	Assuming that $\Pic \mathbf{X}$ is finitely generated we let the \emph{pseudo-effective cone} $\Eff_{\mathbf{X}} \subset (\Pic \mathbf{X})_{\R}$ be the closure of the cone generated by the effective Cartier divisors. 
\end{definition}

\subsection{Clemens complexes}
The examples of adelic varieties we will consider in this paper will all be associated to faces of Clemens complexes. Let us first recall the definition from \cite[\S3.1]{Chambert-Loir2010Igusa}
\begin{definition}
	Let $X$ be a variety over a field $F$ and $D \subset X$ a normal crossing divisor. Let $\mathcal{A}_D$ be the set of irreducible components of $D$, if $\alpha \in \mathcal{A}_D$ then we write $D_{\alpha}$ for the corresponding irreducible component.
	\begin{enumerate}
		\item The \emph{Clemens complex} $\mathcal{C}(D)$ is the poset which has as elements pairs $(A, Z)$ with $A \subset \mathcal{A}_D$ and $Z$ an irreducible component of $\cap_{\alpha \in \mathcal{A}} D_{\alpha}$. We call such a pair a \emph{face}. The order is defined by $(A, Z) \preceq (A', Z')$ if $A \subset A'$ and $Z \supset Z'$.
		
		\item The \emph{geometric Clemens complex} $\mathcal{C}^{\text{geom}}(D)$ is the Clemens complex $\mathcal{C}(D_{\overline{F}})$ of $D_{\overline{F}} \subset X_{\overline{F}}$.
		
		\item If $F$ is a number field and $v$ a place then we define the \emph{$v$-analytic Clemens complex} $\mathcal{C}_{v}^{\text{an}}(D)$ as the subset of those $(A,Z) \in \mathcal{C}^{\text{geom}}(D)$ such that $Z$ is defined over $F_v$ and $Z(F_v) \neq \emptyset$. The subset of maximal faces is denoted $\mathcal{C}_{v}^{\text{an,max}}(D)$.
		
		\item If $S$ is a set of places of $F$ then the \emph{$S$-analytic Clemens complex} is $\mathcal{C}^{\text{an}}_{S}(D) := \prod_{v \in S} \mathcal{C}_{v}^{\text{an}}(D)$. The subset of maximal faces is denoted $\mathcal{C}_{S}^{\text{an,max}}(D)$.
	\end{enumerate}
	
	Recall that the \emph{dimension} $\dim(p)$ of an element $p \in P$ of a finite poset $(P, \preceq)$ is the maximal $n \in \N$ such that there exists a chain $p_0 \prec \cdots \prec p_n \prec p$. We will use the convention that the dimension of a minimal element is $-1$. As a Clemens complex is a poset its elements have a dimension.
\end{definition}
The Galois group $\Gamma_F$ acts on $\mathcal{C}^{\text{geom}}_D$ in the obvious way. If $F$ is a number field and $v$ a place of $F$ then $(A,Z) \in \mathcal{C}^{\text{geom}}(D)$ lies in $\mathcal{C}^{\text{an}}_v(D)$ if and only if it is fixed by $\Gamma_v$ and $Z(F_v) \neq \emptyset$.

We will write elements of $\mathcal{C}_{S}^{\text{an,max}}(D)$ as $\mathbf{A} = ((A_v, Z_v))_{v}$ and use the convention that $(A_v, Z_v) := (\emptyset, X)$ if $v \not \in S$.
We will associate the following spaces to faces of Clemens complexes.
\begin{definition}
	Let $F$ be a number field, $X$ a variety over $F$ and $D \subset X$ a normal crossing divisor. Write $U := X \setminus D$.
	\begin{enumerate}
		\item Let $v$ be a place of $v$ and $(A_v, Z_v) \in \mathcal{C}_v^{\text{an}}(D)$ a face. The \emph{interior} of $Z_v$ is defined as 
		\begin{equation*}
			Z_v^{\circ} := Z_v \setminus \bigcup_{(A_v, Z_v) \prec (A, Z)} Z.
		\end{equation*}
		\item Let $v$ be a place of $v$ and $(A_v, Z_v) \in \mathcal{C}_v^{\text{an}}(D)$ a face. We define \emph{the variety of $v$-adic points near $(A_v, Z_v)$} as 
		\begin{equation*}
			U_{Z_v} := X_{F_v} \setminus \bigcup_{(A, Z) \nprec (A_v, Z_v)} Z = U_{F_v} \cup \bigcup_{(A, Z) \preceq (A_v, Z_v)} Z^{\circ}.
		\end{equation*}
		\item Let $S$ be a finite set of places and $\mathbf{A} \in \mathcal{C}^{\text{an}}_{S}(D)$ a face. We define \emph{the adelic variety of points near $\mathbf{A}$} as the adelic variety
		\begin{equation*}
			(X; {\mathbf{A}}) := (U, (U_{Z_v})_{v \in \Omega_F}, (\xi_v)_{v \in \Omega_F})
		\end{equation*}
		where $\xi_v: U_{F_v} \to U_{Z_v}$ is the inclusion and we used the convention that $(A_v, Z_v) = (\emptyset, X)$ if $v \not \in S$.
	\end{enumerate} 
\end{definition}

The adelic variety $(X ; {\mathbf{A}})$ has the following space of adelic points
\begin{equation*}
	(X; {\mathbf{A}})(\A_F) = U(\A_{F, S}) \times \prod_{v \in S} U_{Z_v}(F_v). 
\end{equation*}

Note that for all $v \in S$ the inclusion $Z_v^{\circ} \subset U_{Z_v}$ is a closed immersion. So $Z_v^{\circ}(F_v) \subset U_{Z_v}(F_v)$ is a closed subset and taking the restricted product over all places we get a closed subset
\begin{equation*}
	U(\A_F)_{\mathbf{A}} := U(\A_{F, S}) \times \prod_{v \in S} Z_v^{\circ}(F_v) = \sideset{}{'}\prod_{v \in \Omega_{F}} Z_v^{\circ}(F_v) \subset (X; {\mathbf{A}})(\A_F).
\end{equation*}

We can compute the cohomological invariants of $(X ; \mathbf{A})$ in terms of those of $X$.
\begin{lemma}\label{lem:Computation of Pic(X ; A)}
	Let $X, D$ and $\mathbf{A}$ be as above and assume that $X$ is smooth and $\overline{F}[U]^{\times} = \overline{F}^{\times}$.
	\begin{enumerate}
		\item The following sequence of finitely generated groups is exact
		\[
		0 \to \Z[\mathcal{A}_D /\Gamma_F] \to \Pic X \times \prod_{v \in S} \Z[A_v/\Gamma_v] \to \Pic (X ; \mathbf{A}) \to 0
		\]
		The first map sends $a \in \mathcal{A}_D/\Gamma$ to $(\sum_{\alpha \in a} [D_{\alpha}], (\sum_{\alpha\Gamma_v \subset a\Gamma_F} \alpha\Gamma_v)_{v \in S})$. The second map sends 
		$L \in \Pic X$ to $(L, (L_v)_v) \in \Pic (X ; \mathbf{A})$ and for $v \in S$ it sends $a\Gamma_v \in A_v /\Gamma_v$ to $(0, (0)_{w \neq v}, [\sum_{\alpha \in a \Gamma_v} D_{\alpha}])$. 
		Moreover, if $\Pic X$ is finitely generated then $\Eff_{(X ; \mathbf{A})}$ is the cone generated by the images of $\Eff_X$ and $\Z_{\geq 0}[A_v / \Gamma_v]$.
		\item The Brauer group satisfies $\Br (X ; \mathbf{A}) \subset \Br U$ and 
		\[
		\Br (X ; \mathbf{A}) = \{b \in \Br U: b_{F_v} \in \emph{im}(\Br U_{Z_v} \to \Br U) \emph{ for all } v \in S\}.
		\]
	\end{enumerate}
\end{lemma}
\begin{proof}
	The first statement follows by combining that irreducible elements of $D$, resp. $D_{F_v}$, correspond to $\Gamma_F$-orbits, resp. $\Gamma_v$-orbits, of $\mathcal{A}_D$, the standard exact sequences
	\begin{align*}
		&0 \to \Z[\mathcal{A}_D/ \Gamma_F] \to \Pic X \to \Pic U \to 0 \\
		&0 \to \Z[A_v/\Gamma_v] \to \Pic U_{Z_v} \to \Pic U \to 0
	\end{align*}
	the definition of $\Pic (X ; \mathbf{A})$ and a bit of diagram chasing.
	
	The second statement follows from the definition of $\Br (X ; \mathbf{A})$ and that $\Br U_{Z_v} \subset \Br U$ for all $v$ by Grothendieck purity \cite[Thm.~3.7.1]{Colliot2021Brauer}.
\end{proof}
\begin{remark}
	It follows from the above lemma and \cite[Lem.~2.2.3]{Wilsch2022Integral} that in the split case and if $\overline{F}[U]^{\times} = \overline{F}^{\times}$ then $\Pic(X ; \mathbf{A})$ agrees with what is denoted $\Pic(U; \mathbf{A})$ by Wilsch.
\end{remark}
\begin{corollary}\label{cor:Rank of Picard group Clemens complex}
	We have $\emph{rk}( \Pic(X ; \mathbf{A}))= \emph{rk}( \Pic U) + \dim \mathbf{A} + 1$ under the assumptions of Lemma \ref{lem:Computation of Pic(X ; A)} we have
\end{corollary}
\begin{proof}
	Note that $\text{rk}( \Pic X) = \text{rk}( \Pic U) + \# \mathcal{A}_D/\Gamma_F$ and $\dim \mathbf{A} = \sum_{v \in S} \# A_v / \Gamma_v - 1$ by \cite[p.~380]{Chambert-Loir2010Igusa}. This equality then follows from the lemma.
\end{proof}

It is known that integral points tend to concentrate around maximal faces of the analytic Clemens complex. It was observed by Wilsch \cite[\S 2.3]{Wilsch2022Integral} that for split varieties there can be analytic obstructions to integral points near certain faces of the Clemens complex. The argument becomes more natural when generalized to arbitrary adelic varieties.
\begin{theorem}	\label{thm: Non-constant regular functions imply failure of Zariski density}
	Let $\mathbf{X}$ be an adelic variety with such that $\HH^0(\mathbf{X}, \mathcal{O}_{\mathbf{X}}) \neq F$. Then for any compact subset $C \subset \mathbf{X}(\A_F)$ the set $X_F(F) \cap C$ is not Zariski-dense in $\mathbf{X}$ in the sense that there exists a proper closed adelic subvariety $\mathbf{D} \subsetneq \mathbf{X}$ such that $ X_F(F) \cap C \subset D_F(F)$.
\end{theorem}
\begin{proof}
	Let $\mathbf{f}:\mathbf{X} \to \A^1$ be a non-constant regular function. The map $\mathbf{f}: \mathbf{X}(\A_F) \to \A^1(\A_{F}) = \A_F$ is continuous so the image $\mathbf{f}(C)$ is compact. The intersection of $\mathbf{f}(C)$ with the discrete subset $\A^1(F) = F$ is thus finite.
	
	Let $\mathbf{D} := \mathbf{f}^{-1}(\mathbf{f}(C) \cap F)$. We clearly have $ X_F(F) \cap C \subset D_F(F)$ and $\mathbf{D} \neq \mathbf{X}$ because $\mathbf{f}$ is non-constant.
\end{proof}
\begin{remark}\label{rem:Analytic obstructions to Zariski density}
	Our application of this theorem as follows. Let $X, D, \mathbf{A}$ be as above. Let $\mathcal{X}$ be a proper $\mathcal{O}_{F,S}$-integral model of $X$, $\mathcal{D}$ the closure of $D$ in $\mathcal{X}$ and $\mathcal{U} := \mathcal{X} \setminus \mathcal{D}$. 
	
	For each $v \in S$ let $C_v \subset U_{Z_v}(F_v)$ be a compact subspace. We can then consider the compact subspace $\prod_{v \not \in S} \mathcal{U}(\mathcal{O}_v) \times \prod_{v \in S} C_v \subset (X ; \mathbf{A})(\A_F)$. If $\HH^0((X; \mathbf{A}), \mathcal{O}_{(X; \mathbf{A})}) \neq F$ then Theorem \ref{thm: Non-constant regular functions imply failure of Zariski density} implies that the set of $\mathcal{O}_{F,S}$-integral points of $\mathcal{U}$ which lie in $C_v$ for all $v \in S$ is not Zariski dense. This generalizes \cite[Prop. 3.2.1]{Wilsch2022Integral}.
\end{remark}
This leads us to the following definition.
\begin{definition}\label{def:Analytic obstruction}
	The face $\mathbf{A}$ has an \emph{analytic obstruction (to Zariski density of integral points)} if $\HH^0((X ; \mathbf{A}), \mathcal{O}_{(X ; \mathbf{A})}) \neq F$.
\end{definition}

We will apply Theorem \ref{thm: Non-constant regular functions imply failure of Zariski density} to deal with the faces of the Clemens complex where we cannot apply our method. The reason why our method requires that there is no analytic obstruction is that we require the conclusion of the following proposition.
\begin{proposition}\label{prop:pseuo-effective cone has non-zero cone constant}
	Let $X$ be a smooth proper variety over $F$ such that $\Pic X_{\overline{F}}$ is finitely generated. Let $D \subset X$ be a geometrically simple normal crossing divisor, $S$ a finite set of places and $\mathbf{A} \in \mathcal{C}^{\an}_S(D)$ a face of the analytic Clemens complex. 
	
	Assume that the pseudo-effective cone $\Eff_X$ is generated by the effective divisors.
	
	If $\mathbf{A}$ has no analytic obstruction then $\Eff_{(X ; \mathbf{A})} \cap - \Eff_{(X ; \mathbf{A})} = 0$.
\end{proposition}
\begin{proof}
	This is a minor generalization of \cite[Thm.~2.4.1]{Wilsch2022Integral} and the same argument works.
\end{proof}

\subsection{Adelic groups of multiplicative type}
An important class of adelic varieties are adelic groups of multiplicative type. We will also consider the corresponding adelic family of their character groups.
\begin{definition} \hfill
	\begin{enumerate}
		\item An \emph{adelic group of multiplicative type $\mathbf{T} = (T_F, (T_v)_{v \in \Omega_F}, (\tau_v)_{v \in \Omega_F})$ over $F$} is an adelic variety where $T_F$, resp. $T_v$, is a group of multiplicative type over $F$, resp. $F_v$, and $\tau_v$ is a group homomorphism $(T_F)_{F_v} \to T_v$.
		\item A homomorphism $\mathbf{f}: \mathbf{T} \to \mathbf{T'}$ of adelic groups of multiplicative type is a morphism of adelic varieties such that all the $f_v$ are homomorphisms of groups.
		\item An \emph{adelic module over $F$} is a triple $\mathbf{L} = (L_F, (L_v)_{v \in \Omega_F}, (\lambda_v)_{v \in \Omega_F})$ where $L_F$, resp. $L_v$, is abelian group with a continuous $\Gamma_F$, resp. $\Gamma_v$-action, and $\lambda_v$ is a map $L_v \to L_F$ of $\Gamma_v$-modules.
		
		If $L_F$ and all the $L_v$ are finitely generated and $\lambda_v$ is an isomorphism for all but finitely many $v$ then we say that $\mathbf{L}$ is finitely generated.
		
		\item Let $\mathbf{L} = (L_F, (L_v)_{v \in \Omega_F}, (\lambda_v)_{v \in \Omega_F})$ and $\mathbf{L}' = (L'_F, (L'_v)_{v \in \Omega_F}, (\lambda'_v)_{v \in \Omega_F})$ be adelic modules over $F$. A map $\bm{f}: \mathbf{L} \to \mathbf{L}'$ is a tuple $(f_F, (f_v)_{v \in \Omega_F})$ where $f_F: L_F \to L'_F$ is a morphism of $\Gamma_F$-modules and $f_v: L_v \to L'_v$ is a morphism of $\Gamma_v$ modules such that $f_F \circ \lambda_v = \lambda'_v \circ f_v$ for all $v \in \Omega_F$.
		\item If $\mathbf{T} = (T_F, (T_v)_{v \in \Omega_F}, (\tau_v)_{v \in \Omega_F})$ is an adelic group of multiplicative type then its \emph{group of characters} is the finitely generated adelic module
		\begin{equation*}
			X^*(\mathbf{T}) := (X^*(T_F), (X^*(T_v))_{v \in \Omega_F}, (X^*(\tau_v))_{v \in \Omega_F}).
		\end{equation*}
		
		Its group of \emph{global characters} is $\hat{\mathbf{T}} := \Hom(\mathbf{T}, \G_m) \cong \Hom(\Z, X^*(\mathbf{T}))$.
		\item If $\mathbf{L} = (L_F, (L_v)_{v \in \Omega_F}, (\lambda_v)_{v \in \Omega_F})$ is an finitely generated adelic module then its \emph{Cartier dual} is the adelic group of multiplicative type 
		\begin{equation*}
			D(\mathbf{L}) := (D(L_F), (D(L_v))_{v \in \Omega_F}, (D(\lambda_v))_{v \in \Omega_F}).
		\end{equation*}
	\end{enumerate}
\end{definition}
The functors $X^*(\cdot)$ and $D(\cdot)$ clearly define a contravariant equivalence between the categories of adelic groups of multiplicative type and the finitely generated adelic modules.

We can define the Galois cohomology of adelic modules and adelic groups of multiplicative type.
\begin{definition}
	Let $\mathbf{L}$ be an adelic module over $F$ and $i \in \N$. We define the \emph{$i$-th Galois cohomology group} as the following group, where addition is component-wise.
	\begin{equation*}
		\HH^i(F, \mathbf{L}) := \{(\alpha_F, (\alpha_v)_{v \in \Omega_F}) \in \HH^i(F, L_F) \times \prod_{v \in \Omega_F} \HH^i(F_v, L_v) : \lambda_{v*} \alpha_v = (\alpha_F)_{F_v}\}.
	\end{equation*}

	Let $S \subset \Omega_F$ be a finite set of places such that $L_v$ is unramified for $v \not \in S$. The \emph{adelic Galois cohomology} of $\mathbf{L}$ is the restricted product
	\begin{equation*}
		\HH^1(\A_F, \mathbf{L}) := \prod_{v \in S} \HH^i(F_v, L_v) \sideset{}{'}\prod_{v \in \Omega_F \setminus S} (\HH^i(F_v, L_v), \HH^1(\mathcal{O}_v, L_v)).
	\end{equation*}

	The \emph{Tate-Shafarevich group} of $\mathbf{L}$ is 
	\begin{equation*}
		\Sha^i(F, \mathbf{L}) := \ker(\HH^i(F, \mathbf{L}) \to 	\HH^i(\A_F, \mathbf{L})).
	\end{equation*}
\end{definition}
\begin{definition}
	Let $\mathbf{T}$ be an adelic group of multiplicative type and $i \in \N$. Its \emph{$i$-th Galois cohomology group} is defined as $\HH^i(F, \mathbf{T}) := \HH^i(F, T_F)$.
	
	Let $S \subset \Omega_F$ be a finite set of places such that $T_v$ has good reduction when $v \not \in S$. We define the \emph{adelic Galois cohomology} of $\mathbf{T}$ as the restricted product
	\begin{equation*}
		\HH^i(\A_F, \mathbf{T}) := \prod_{v \in S} \HH^i(F_v, T_v) \sideset{}{'}\prod_{v \in \Omega_F \setminus S} (\HH^i(F_v, T_v) , \HH^i(\mathcal{O}_v, T_v)).
	\end{equation*}
	This is obviously independent of $S$.
	
	The \emph{Tate-Shafarevich group} of $\mathbf{T}$ is
	\begin{equation*}
		\Sha^i(F, \mathbf{T}) := \ker(\HH^i(F, \mathbf{T}) \to 	\HH^i(\A_F, \mathbf{T}))
	\end{equation*} 
\end{definition}
Let us note the following simple fact.
\begin{lemma}
	Let $\mathbf{L}$ be an adelic module. We have $\Hom(\Z, \mathbf{L}) = \HH^0(F, \mathbf{L})$.
	\label{Computation of the global characters of an adelic multiplicative group}
\end{lemma}
\begin{proof}
	This follows after unfolding definitions.
\end{proof}
\begin{remark}
	The category of adelic modules is abelian so one might expect the natural cohomology theory on it to be given by the Ext groups $\text{Ext}^i(\Z, \cdot)$. It turns out that this is not isomorphic to the cohomology we defined. For example, if $\mathbf{L} := (\Z, (0)_{v \in \Omega_F})$ then $\HH^1(F, \mathbf{L}) = 0$ but $\text{Ext}^1(\Z, \mathbf{L}) \cong (\prod_{v \in \Omega_F} \Z) /\Z$. 
	\label{Comparison between adelic H and adelic Ext}
\end{remark}

Tate-Shafarevich groups are finite for finitely generated adelic modules and groups of multiplicative type.
\begin{lemma}
	\hfill
	\begin{enumerate} 
		\item Let $\mathbf{L}$ be an adelic module over $F$ and $i \in \N$. Then $\Sha^i(F, \mathbf{L}) = \Sha^i(F, L_F)$. In particular, $\Sha^i(F, \mathbf{L})$ is finite if $L_F$ is finitely generated.
		
		\item Let $\mathbf{T}$ be an adelic group of multiplicative type. Then $\Sha^1(F, \mathbf{T})$ is finite.
	\end{enumerate}
\end{lemma}
\begin{proof}
	For the first part let $(\alpha, (\alpha_v)_{v \in \Omega_F}) \in \Sha^i(F, \mathbf{L}))$. Then $\alpha_v = 0$ for all $v \in \Omega_F$ by the definition of $\Sha^i(F, \mathbf{L})$. By the definition of $\HH^i(F, \mathbf{L})$ this implies that $(\alpha)_{F_v} = 0$ for all $v \in \Omega_F$, i.e.~$\alpha \in \Sha^i(F, L_F)$. Conversely, if $\alpha \in \Sha^i(F, L_F)$ then $(\alpha, (0)_{v \in \Omega_F}) \in \Sha^i(F, \mathbf{L}))$. Finiteness follows from \cite[Thm.~4.20]{Milne2006Duality}.
	
	For the second part let $S$ is a finite set of places such that for all $v \not \in S$ $T_v = (T_F)_{F_v}$. By definition we know that there exists an exact sequence
	\begin{equation*}
		0 \to \Sha^1(F, T_F) \to \Sha^1(F, \mathbf{T}) \to \oplus_{v \in S} \HH^1(F_v, T_F)
		\end{equation*}
	The lemma follows since $\Sha^1(F, T_F)$ and $\HH^1(F_v, T_F)$ are finite for all $v \in \Omega_F$ by \cite[Thm.~8.6.7]{Neukirch2008Cohomology} and \cite[Cor.~2.3]{Milne2006Duality}.
\end{proof}

The main examples of adelic modules appearing in this paper are adelic Picard groups.
\begin{definition}
	The \emph{adelic Picard group} of an adelic variety $\mathbf{X}$ is the adelic module
	\begin{equation*}
		\bm{\Pic} \mathbf{X} := (\Pic X_{F,\overline{F}}, (\Pic X_{v, \overline{F}_v})_{v \in \Omega_F}, (\xi_v^*)_{v \in \Omega_F}).
	\end{equation*}
\end{definition}

There is an obvious map of adelic modules $\Pic \mathbf{X} \to \bm{\Pic} \mathbf{X}$.
\begin{lemma}\label{lem:global sections of adelic Pic are Pic}
	If $\mathbf{X}$ is an adelic variety such that $X_F(F) \neq \emptyset$ then $\Pic \mathbf{X} \cong \Hom(\Z, \bm{\Pic} \mathbf{X})$.
\end{lemma}
\begin{proof}
	If $X_F(F) \neq \emptyset$ then also $X_v(F_v) \neq 0$ for all $v \in \Omega_F$. These facts imply that $\Pic X_F = (\Pic X_{F, \overline{F}})^{\Gamma_F}$ and $\Pic X_v \cong (\Pic X_{v, \overline{F}_v})^{\Gamma_v}$. The lemma then follows from the definition.
\end{proof}

Moreover, $\bm{\Pic} \mathbf{X}$ is also related to the algebraic Brauer group $\Br_1 \mathbf{X}$. Indeed, by the Hochschild-Serre spectral sequence there are maps $r: \Br_1 X_F \to \HH^1(F, \Pic X_F)$ and $r_v: \Br_1 X_v \to \HH^1(F_v, \Pic X_v)$. By the functoriality of the Hochschild-Serre spectral sequence these fit together into a map $r: \Br_1 \mathbf{X} \to \HH^1(F, \bm{\Pic} \mathbf{X})$.
\begin{lemma}\label{lem:Br_1 is isomorphic to H^1(Pic) adelically}
	Let $\mathbf{X}$ be an adelic variety such that $\overline{F}[X_F]^{\times} = \overline{F}^{\times}$ and $\overline{F}_v[X_v]^{\times} = \overline{F}_v^{\times}$ for all $v \in \Omega_F$. The map $r: \Br_1 \mathbf{X}/\Br_ 0 \mathbf{X} \to \HH^1(F, \bm{\Pic} \mathbf{X})$ is an isomorphism.
\end{lemma}
\begin{proof}
	The assumptions and \cite[Cor. 2.3.9]{Skorogobatov2001Torsors} imply that the maps $r_v: \Br_1 X_v/ \Br_0 X_v\to \HH^1(F_v, \Pic X_v)$ for $v \in \Omega_{F}$ and $r: \Br_1 X_F /\Br_0 X_F \to \HH^1(F, \Pic X_F)$ are all isomorphisms. 
\end{proof}

Note that the topological space $\mathbf{T}(\A_F)$ is a topological group because it is the restricted product of topological groups. We can generalize the norm and the notion of norm $1$ elements to adelic groups of multiplicative type.
\begin{definition}
	Let $\mathbf{T}$ be an adelic group of multiplicative type over $F$.
	We define the \emph{norm map} as the map
	\begin{equation*}
		| \cdot |_{\mathbf{T}}: \mathbf{T}(\A_F) \to \Hom(\hat{\mathbf{T}}, \R_{>0}^{\times}): (t_v)_{v} \to \left((\chi, (\chi_v)_v) \to \prod_v |\chi_v(t_v)|_v \right).
	\end{equation*}
	We will call its kernel the \emph{group of norm $1$ elements} and denote it by $\mathbf{T}(\A_F)^1$.
	
	We note that $T_F(F) \subset \mathbf{T}(\A_F)^1$ by the product formula. 
\end{definition}
\begin{lemma}\label{lem:Norm map is surjective}
	The norm map is surjective if $T_F \to T_v$ is an isomorphism for all non-archimedean $v$.
	\label{Norm map is surjective}
\end{lemma}
\begin{proof}
We will show that the log norm is surjective.

Consider the following commutative diagram
\begin{equation}\label{eq:exact sequence of log norm of adelic groups of multiplicative type}
	\begin{tikzcd}
		{T_F(\A_F)} & {\mathbf{T}(\A_F)} & {\prod_v\coker(T_F(F_v) \to T_v(F_v))} & 0 \\
		{(L_{F, \R}^{\Gamma_F})^*} & {\hat{\mathbf{T}}_{\R}^*} & {\prod_v \coker((L_{F, \R}^{\Gamma_v})^* \to (L_{v, \R}^{\Gamma_v})^*)} & 0
		\arrow[from=1-2, to=1-3]
		\arrow[from=1-1, to=1-2]
		\arrow[from=1-3, to=1-4]
		\arrow[from=2-3, to=2-4]
		\arrow[from=2-2, to=2-3]
		\arrow[from=2-1, to=2-2]
		\arrow["{\log |\cdot|_{\mathbf{T}}}", from=1-2, to=2-2]
		\arrow["{\log |\cdot|_{T_F}}", from=1-1, to=2-1]
		\arrow[from=1-3, to=2-3]
	\end{tikzcd}
\end{equation}
The first row is exact by definition and the second row is exact since it is the dual of the following sequence tensored by $\R$. 
\begin{equation*}
	0 \to \bigoplus_v \ker(L_v^{\Gamma_v} \to L_F^{\Gamma_v}) \to \hat{\mathbf{T}} \to \hat{T}_F 
\end{equation*}
This sequence is exact by the definition of $\hat{\mathbf{T}}$.

The map $\log |\cdot|_{T_F}$ is surjective so by the four lemma it suffices to show that 
\[
\prod_v\coker(T_F(F_v) \to T_v(F_v)) \to \prod_v \coker((L_{F, \R}^{\Gamma_v})^* \to (L_{v, \R}^{\Gamma_v})^*)
\] is surjective. This follows from the fact that for each $v \in \Omega_{F}^{\infty}$ the map $\log |\cdot|_{T_v}: T_{v}(F_v) \to (L_{v, \R}^{\Gamma_v})^*$ is surjective.
\end{proof}
\begin{lemma}\label{lem: Quotient is compact}
	If $(T_F)_{F_v} \to T_v$ is a closed immersion for all $v$ then $T(F) \subset \mathbf{T}(\A_F)$ is a discrete inclusion and $\mathbf{T}(\A_F)^1/ T_F(F)$ is compact.
\end{lemma}
\begin{proof}
	In this case $T_F(\A_F) \subset \mathbf{T}(\A_F)$ is a closed inclusion and $T_F(F) \subset T_F(\A_F)$ is a discrete inclusion.
	
	It follows from the condition that $L_{v, \R} \to L_{F, \R}$ is surjective. This implies that $L_{v, \R}^{\Gamma_v} \to L_{F, \R}^{\Gamma_v}$ is surjective because $\Gamma_v$ is compact. It follows from the definition of $\hat{\mathbf{T}}$ that $\hat{\mathbf{T}}_{\R} \to L_{F, \R}^{\Gamma_F}$ is surjective.
	
	Applying this and the diagram \eqref{eq:exact sequence of log norm of adelic groups of multiplicative type} we deduce that the following sequence is exact.
	\[
	T_F(\A_F)^1 \to \mathbf{T}(\A_F)^1 \to \prod_v\ker\left(T_v(F_v)/T_F(F_v) \to (L_{v, \R}^{\Gamma_v})^*/(L_{F, \R}^{\Gamma_v})^*\right).
	\]
	
	We have $X^*(T_v/T_F) = \ker(L_v \to L_F)$. The group $T_F(F_v)/T_v(F_v)$ injects into $T_F/T_v(F_v)$ and the map $T_v(F_v)/T_F(F_v) \to (L_{v, \R}^{\Gamma_v})^*/(L_{F, \R}^{\Gamma_v})^*$ is the composition of this injection and $\log |\cdot|_{T_v/T_F}: T_v/T_F(F_v) \to X^*(T_v/T_F)_{\R}^* = (L_{v, \R}^{\Gamma_v})^*/(L_{F, \R}^{\Gamma_v})^*$ by functoriality. The kernel of this map is thus compact for all $v$.
	
	Since $T_F(\A_F)^1/T_F(F)$ is also compact this implies that $\mathbf{T}(\A_F)^1/T_F(F)$ is compact.
\end{proof}

To relate the Brauer-Manin obstruction to the descent obstruction for adelic varieties we will need a form of Poitou-Tate duality for adelic groups of multiplicative type, generalizing that of groups of multiplicative type \cite{Milne2006Duality}.
\subsubsection{Local Poitou-Tate duality}
Let $F$ be a local field of characteristic $p$ (possibly $0$), $T$ a smooth\footnote{Equivalently, $L$ has no $p$-torsion.} group of multiplicative type with $\Gamma_F$-module of characters $L$. Recall from class field theory that there is an injective invariant map $\text{inv}_{F}: \HH^2(F, \G_m) = \Br(F) \to \Q/\Z$. The bilinear pairing $T \times L \to \G_m$ then induces a perfect pairing \cite[Cor. 2.3,Thm. 2.13]{Milne2006Duality}
\begin{equation}
	\langle \cdot, \cdot \rangle_{\text{PT}, F}: \HH^1(F, T) \times \HH^1(F, L) \to H^2(F, \G_m) \to \Q/\Z: (a_v, b_v) \to \text{inv}_{F}(a_v \cup b_v).
	\label{Local Poitou-Tate duality}
\end{equation}

If $T$ has good reduction over $\mathcal{O}_F$ then the orthogonal complement of $\HH^1(\mathcal{O}_F, T) \subset \HH^1(F, T)$ with respect to this pairing is $\HH^1(\mathcal{O}_F, L) \subset \HH^1(F, T)$ \cite[Thm. 2.6]{Milne2006Duality}.

The Poitou-Tate pairing is functorial since the cup product is. To be precise, let $\tau: T \to T'$ be a morphism of groups of multiplicative type and $\lambda: L' \to L$ the dual morphism of their $\Gamma_F$-modules of characters. Then for all $a \in \HH^1(F, T)$ and $b \in \HH^1(F, L)$ we have the equality
\begin{equation}
	\langle a, \lambda_* b \rangle_{\text{PT}, F} = \langle \tau_* a, b \rangle_{\text{PT}, F}.
	\label{Functoriality of local Poitou-Tate pairing}
\end{equation}

\subsubsection{Global Poitou-Tate duality}
Let $F$ be a number field, $\mathbf{T}$ an adelic group of multiplicative type with adelic module of characters $\mathbf{L}$.

For each place $v$ of $F$ we get a local Poitou-Tate pairing \eqref{Local Poitou-Tate duality} $\langle \cdot, \cdot \rangle_{\text{PT}, F_v}$. The group of multiplicative type $T_v$ has good reduction at all but finitely many places. By the orthogonality of $\HH^1(\mathcal{O}_v, T_v)$ and $\HH^1(\mathcal{O}_v, L_v)$ the restricted product of these pairings exists. It is explicitly given by the formula
\begin{equation}
	\langle \cdot, \cdot \rangle_{\text{PT}}: \HH^1(\A_F, \mathbf{T}) \times \HH^1(\A_F, \mathbf{L}) \to \Q/\Z: ((a_v)_{v}, (b_v)_{v}) \to \sum_{v \in \Omega_F} \text{inv}_v(a_v \cup b_v).
	\label{Adelic Poitou-Tate pairing}
\end{equation}
This pairing is perfect since all the local pairings are. We will need the following version of Poitou-Tate duality.
\begin{proposition}
	The subgroup orthogonal to the image of $\HH^1(F, T_F) \to \HH^1(\A_F, \mathbf{T})$ with respect to the pairing $\langle \cdot, \cdot \rangle_{\emph{PT}}$ is the image of $\HH^1(F, \mathbf{L}) \to \HH^1(\A_F, \mathbf{L})$.
	\label{Global Poitou Tate for adelic groups of multiplicative type}
\end{proposition}
\begin{proof}
	Let $(b_v)_{v} \in \HH^1(\A_F, \mathbf{L})$ and $a \in \HH^1(F, T_F)$. The maps $\tau_v: T \to T_v$ and $\lambda_v: L_v \to L_F$ are dual so by functoriality of the local Poitou-Tate pairing \eqref{Functoriality of local Poitou-Tate pairing} we have 
	\begin{equation*}
		\langle a, (b_v)_v \rangle_{\text{PT}} = \sum_{v \in \Omega_F} \text{inv}_v(\tau_v(a) \cup b_v) = \sum_{v \in \Omega_F} \text{inv}_v(a \cup \lambda_v(b_v)) = \langle a,(\lambda_v(b_v))_v \rangle_{\text{PT}}.
	\end{equation*}
	
	Global Poitou-Tate duality \cite[Thm. 4.20]{Milne2006Duality} shows that $\langle a, (\lambda_v(b_v))_v \rangle_{\text{PT}} = 0$ for all $a \in \HH^1(F, T_F)$ if and only if there exists a $b \in \HH^1(F, L_F)$ such that $\lambda_v(b_v) = b$ for all $v \in \Omega_F$. This is exactly what we had to prove.
\end{proof}
\begin{remark}
		Poitou-Tate duality includes a duality between $\Sha^1(F, T_F)$ and $\Sha^2(F, L_F)$. There is no such duality between $\Sha^1(F, \mathbf{T})$ and $\Sha^2(F, \mathbf{L})$ in general. For example: if $F = \Q$, $L_F = L_v = \Z/8\Z$ for all places $v \neq 2$ and $L_2 = 0$ then $\Sha^1(F, \mathbf{T}) = \Z/2\Z$, the non-trivial class is given by $16 \in \Q^{\times}/ \Q^{\times 8} = \HH^1(\Q, \mu_8)$, but $\Sha^2(\Q, \mathbf{L}) = \Sha^2(\Q, L_F) \cong \Sha^2(\Q, \Z/8 \Z) =  0$. 
\end{remark}

\subsubsection{Haar measures}
In this section we will consider certain Haar measures related to adelic groups of multiplicative type.

We first discuss the local case. Let $F$ be a local field and $T$ be a smooth group of multiplicative type over $F$. Let $L := X^*(L_{F^{\sep}})$ be its $\Gamma_F$-module of characters. 
\begin{lemma}\label{Number of integral torsors of group of multiplicative type with good reduction}
	Let $\pi_0(T)$ be the group scheme of connected components of $T$, equivalently the Cartier dual of the submodule $L_{\emph{tors}} \subset L$ of torsion elements. If $T$ has good reduction then we have $\# \HH^1(\mathcal{O}, T) = \# \pi_0(T)(F)$.
\end{lemma}
\begin{proof}
	Let $\F$ be the residue field of $F$ and $I := \ker(\Gamma_F \to \Gamma_{\F})$ the inertia group. Because $L$ is finitely generated and $\mathcal{O}$ is henselian we have $\HH^1(\mathcal{O}, L) = \HH^1(\F, L)$. The exact sequence $1 \to I \to \Gamma_F \to \Gamma_{\F} \to 1$ splits since $\Gamma_{\F} \cong \hat{\Z}$ is free. The Hochschild-Serre spectral sequence $\HH^p(I, \HH^q(\F, L)) \implies \HH^{p + q}(F, L)$ thus induces a split exact sequence
	\begin{equation*}
		0 \to \HH^1(\F, L) \to \HH^1(F, L) \to \HH^1(I, L)^{\Gamma_{\F}} \to 0.
	\end{equation*}
	
	We find that $\# \HH^1(\mathcal{O}, T) = \# \HH^1(I, L)^{\Gamma_{\F}}$ since $\HH^1(\F, L)$ is the orthogonal complement of $\HH^1(\mathcal{O}, T)$. The action of $I$ on $L$ is trivial so $\HH^1(I, L) = \Hom(I, L)$. The $\Gamma_{\F}$-module $I$ is isomorphic to the Tate twist $\hat{\Z}(1)$ modulo $p$-torsion so since $L$ has no $p$-torsion because $T$ is smooth we deduce that $\Hom(I, L) \cong \Hom(\hat{\Z}(1), L) \cong L_{\text{tors}}(-1)$.
	
	Consider the dual group $((L_{\text{tors}}(-1))^{\Gamma_{\F}})^{\vee} \cong (L_{\text{tors}}^{\vee}(1))_{\Gamma_{\F}} = \pi_{0}(T)_{\Gamma_{\F}}$ where $(\cdot)_{\Gamma_F}$ denotes the $\Gamma_{\F}$-coinvariants. A finite group has the same size as its dual so it remains to prove that $\# \pi_{0}(T)_{\Gamma_{\F}} = \# \pi_{0}(T)(F)$. If $\sigma$ denotes a topological generator of $\Gamma_{\F} \cong \hat{\Z}$ then this follows from the following exact sequence
	\begin{equation*}
		1 \to \pi_{0}(T)(F) = \pi_{0}(T)^{\Gamma_{\F}} \to \pi_0(T) \xrightarrow{t \to \frac{\sigma(t)}{t}} \pi_0(T) \to \pi_0(T)_{\Gamma_{\F}} \to 1
	\end{equation*}
\end{proof}

This suggest the definition of the following Haar measure, if $T$ is finite then it agrees with \cite[Def.~3.1.2]{Darda2022Torsors}.
\begin{construction}
	Let $d\mu_{T}$ be $\frac{1}{\# \pi_0(T)(F)}$ times the counting measure on $\HH^1(F, T)$ and $d\mu_{L}$ the dual measure via the pairing $\langle \cdot, \cdot \rangle_{\text{PT}, F}$. In other words $d\mu_{L}$ is $\frac{\# \pi_0(T_v)(F_v)}{\# \HH^1(F, T)}$ times the counting measure.
	
	If $T$ has good reduction then $d\mu_{L}(\HH^1(\mathcal{O}, T)) = 1$ by Lemma \ref{Number of integral torsors of group of multiplicative type with good reduction}. 
	\label{Local measures on H^1}
\end{construction}

We can now extend this adelically. Let $F$ be a number field, $\mathbf{T}$ an adelic group of multiplicative type and $\mathbf{L}$ its adelic module of characters.
\begin{construction}\label{con:Tamagawa measures on adelic cohomology}
	We define $d\mu_{\mathbf{T}} := \prod_{v} d\mu_{T_v}$ to be the product measure on the restricted product $\HH^1(\A_F, \mathbf{T})$. This product converges due to Lemma \ref{Number of integral torsors of group of multiplicative type with good reduction}. This is a Haar measure as it is a product of Haar measures.
		
	Let $d\mu_{\mathbf{L}}$ be the dual Haar measure on $\HH^1(\A_F, \mathbf{L})$ via the Poitou-Tate pairing $\langle \cdot, \cdot \rangle_{\text{PT}}$. In other words $d \mu_{\mathbf{L}}$ is the product measure $\prod_v d \mu_{L_v}$.
\end{construction}

The perfect pairing \eqref{Adelic Poitou-Tate pairing} induces a perfect pairing 
\begin{equation*}
	e^{2 i \pi \langle \cdot, \cdot \rangle_{\text{PT}}}: \HH^1(\A_F, \mathbf{T}) \times \HH^1(\A_F, \mathbf{L}) \to S^1: (x,y) \to e^{2 i \pi \langle x, y \rangle_{\text{PT}}}
\end{equation*}
which identifies the Pontryagin dual of $\HH^1(\A_F, \mathbf{T})$ with $\HH^1(\A_F, \mathbf{L})$.

Note that the group $\HH^1(\A_F, \mathbf{T})/(\HH^1(F, T_F)/\Sha^1(F, \mathbf{T}))$ is compact since it is Pontryagin dual to the discrete group $\HH^1(F, \mathbf{L})/\Sha^1(F, \mathbf{L})$ by Proposition \ref{Global Poitou Tate for adelic groups of multiplicative type}.
\begin{definition}
	The \emph{Tamagawa number} of $\mathbf{T}$ is defined as 
	\begin{equation*}
		\tau(\mathbf{T}) := \frac{\# \Sha^1(F, \mathbf{L})}{\# \Sha^1(F, \mathbf{T})} d\mu_{\mathbf{T}}(\HH^1(\A_F, \mathbf{T})/(\HH^1(F, T_F)/\Sha^1(F, \mathbf{T})))
	\end{equation*}
	where $\HH^1(F, T_F)/\Sha^1(F, \mathbf{T})$ is given the counting measure.
\end{definition}
\begin{remark}
	For tori $T$ the Tamagawa number is usually defined in terms of the volume of $T(\A_F)^1/T(F)$ with respect to a certain natural measure, the Tamagawa measure. It was shown by Ono \cite{Ono1963Tamagawa} that with this definition it is equal to $\# \HH^1(F, L)/\# \Sha^1(F, T)$. We will show in Corollary \ref{Tamagawa number of tori} that our definition recovers the same number.
\end{remark}
\begin{example}
	If $T$ is a finite then $\tau(T) = \frac{\# L^{\Gamma_F}}{\# T(F)}$ by \cite[Lem.~3.1.5]{Darda2022Torsors}.
\end{example}
The justification for this seemingly strange definition is the following corollary of the Poisson summation formula.
\begin{corollary}\label{cor:Poisson summation for cohomology}
	Let $f: \HH^1(\A_F, \mathbf{T}) \to \C$ be an $L^1$-function and assume that the Poisson summation formula \eqref{eq:Poisson summation formula} holds for $f$. Then
	\begin{equation}\label{eq:Poisson summation for cohomology}
		\tau(\mathbf{T}) \sum_{x \in \HH^1(F, \mathbf{T})} f(x) = \sum_{y \in \HH^1(F, \mathbf{L})} \hat{f}(y).
	\end{equation}
\end{corollary}
\begin{proof}
	The isomorphism $(\HH^1(\A_F, \mathbf{T})/(\HH^1(F, T_F)/\Sha^1(F, \mathbf{T}))^{\vee} \cong \HH^1(F, \mathbf{L})/\Sha^1(F, \mathbf{L})$ coming from Proposition \ref{Global Poitou Tate for adelic groups of multiplicative type} induces a measure on $\HH^1(F, \mathbf{L})/\Sha^1(F, \mathbf{L})$. Since $\HH^1(F, \mathbf{L})$ is discrete this measure is $\mu_T(\HH^1(\A_F, \mathbf{T})/(\HH^1(F, T_F)/\Sha^1(F, \mathbf{T}))^{-1}$ times the counting measure.
	
	The Poisson summation formula thus takes the following form
	\begin{equation*}
		\sum_{x \in \HH^1(F, \mathbf{T})/\Sha^1(F, \mathbf{T})} f(x) = \frac{\sum_{y \in \HH^1(F, \mathbf{L})/\Sha^1(F, \mathbf{L})} \hat{f}(y)}{\mu_{\mathbf{T}}(\HH^1(\A_F, \mathbf{T})/(\HH^1(F, T_F)/\Sha^1(F, \mathbf{T}))}.
	\end{equation*}
	We can rewrite this equation as
	\begin{equation*}
		\tau(\mathbf{T}) \#\Sha^1(F, \mathbf{T}) \sum_{x \in \HH^1(F, \mathbf{T})/\Sha^1(F, \mathbf{T})} f(x) = \# \Sha^1(F, \mathbf{L}) \sum_{y \in \HH^1(F, \mathbf{L})/\Sha^1(F, \mathbf{L})}\hat{f}(y)
	\end{equation*}
	which is the same as \eqref{eq:Poisson summation for cohomology}.
\end{proof}
\begin{corollary}
	Let $T$ be a torus over $F$. Then $\tau(T) = \frac{\# \HH^1(F, L)}{\# \Sha^1(F, T)}$.
	\label{Tamagawa number of tori}
\end{corollary}
\begin{proof}
	It follows from Lemma \ref{Number of integral torsors of group of multiplicative type with good reduction} that the topology on $\HH^1(\A_F, T)$ is the discrete topology. Let $f: \HH^1(\A_F, T) \to \C$ be the indicator function of $0$. Then $\hat{f}: \HH^1(\A_F, \mathbf{L}) \to \C$ is the constant function $1$. So \eqref{eq:Poisson summation for cohomology} becomes $\tau(T) \# \Sha^1(F, T) = \# \HH^1(F, L)$.
\end{proof}

We will need end up needing the following estimate.
\begin{definition}\label{def:conductor}
	Let $\mathbf{A}$ be either an adelic group of multiplicative type or a finitely generated $\Gamma_F$-module. Let $S$ be a finite set of places containing the archimedean places such that $A_F$ has good reduction outside of $S$. We define the \emph{conductor} of $a \in \HH^1(F, \mathbf{A})$ as 
	\[
	\text{cond}_S(a) := \prod_{\substack{v \not \in S \\ a_v \not \in \HH^1(\mathcal{O}_v, A)}} q_v.
	\]
\end{definition}
Let $\omega(n)$ be the number of divisors of $n$.
\begin{lemma}\label{lem:Bounds for conductor sums}
	In the situation of Definition \ref{def:conductor} we have that the following sum converges absolutely for all $\delta, k > 0$ 
	\[
	\sum_{a \in \HH^1(F, \mathbf{A})}
	\frac{k^{\omega(\emph{cond}_S(a))}}{\emph{cond}_S(a)^{1 + \delta}}.
	\]
\end{lemma}
\begin{proof}
	By the divisor bound $k^{\omega(\text{cond}_S(a))} \ll_{\varepsilon} \text{cond}_S(a)^{\varepsilon}$ for all $\varepsilon > 0$ we may assume that $k = 1$.
	
	For $v \in S$ let $f_v = \mathbf{1}_{\HH^1(F_v, A_v)}$ and for $v \not \in S$ define $f_v(a_v) = 1$ for $a_v \in \HH^1(\mathcal{O}_v, A_v)$ and $f_v(a_v) = q_v^{-1 - \delta}$ for the other $a_v \in \HH^1(F_v, A_v)$. So $\text{cond}_S(a)^{-1 - \delta} = \prod_v f_v(a)$ for $a \in \HH^1(F, \mathbf{A})$.
	
	The Fourier transform $\hat{f}_v$ for $v \not \in S$ is equal to $1 + q_v^{-1 - \delta} \# \HH^1(F_v, A_v)/\HH^1(\mathcal{O}_v, A_v)$ times the indicator function $\mathbf{1}_{\HH^1(\mathcal{O}_v, \Hom(A_v, \G_m))}$ by character orthogonality and the construction of the measure. It follows that $\hat{f} = \prod_v \hat{f}_v$ converges absolutely and is compactly supported. The function $f$ is invariant under translation by the open subgroup $\prod_{v \not \in S} \HH^1(\mathcal{O}_v, A_v) \subset \HH^1(\A_F, \mathbf{A})$ so the Poisson summation formula is valid by Theorem \ref{thm:Poisson summation} which implies that $\sum_{a \in \HH^1(F, \mathbf{A})} f(a) < \infty$.
\end{proof}
\section{Metrics on torsors}
We will develop the theory of metrics on torsors in this section. Our definition of a metric on a line bundle is slightly different to \cite[\S2.1.3]{Chambert-Loir2010Igusa} since we will only allow them to take values in $q^{\Z} \cup \{0\}$. But this has essentially no effect on the heights, which is what we are eventually interested in.
\subsection{Local metrics on torsors}
For this section let $F$ be a local field and let $\mathcal{O}^{\times} \subset F^{\times}$ be the subgroup of norm $1$ elements. If $F$ is non-archimedean let $q$ be the size of its residue field. Let $X$ be a variety over $F$ and $L$ a line bundle on $X$. Let $\pi: L \to X$ be the total space of $L$ and let $L^{\neq 0} := L - X$ where we identify $X$ with the zero section $X \subset L$, this is the $\G_m$-torsor corresponding to $L$. 
\begin{definition}
	A \emph{metric} on $L$ is a continuous function $|| \cdot ||: L(F) \to \R_{\geq 0}$ whose image lies in $q^{\Z} \cup \{ 0\}$ if $F$ is non-archimedean and such that for all $x \in X(F)$ the restriction of $|| \cdot ||$ to the vector space $L(x)$ is a norm. 
	
	If $F$ is archimedean and $X$ is smooth then we say that the metric $|| \cdot ||$ is \emph{smooth} if is smooth as a map $L(F) \to \R_{\geq 0}$.
		
	The above data is clearly equivalent to a continuous map $L^{\neq 0} (F) \to \R_{>0}^{\times}$ whose image lies in $q^{\Z}$ if $F$ is non-archimedean and such that for all $y \in L^{\neq 0} (F)$ and $f \in F^{\times}$ we have $||f \cdot y|| = |f| \cdot || y||$.
\end{definition}
\begin{remark}
	We can write this definition in a more uniform way. Note that $|\cdot|$ identifies $F^{\times}/\mathcal{O}^{\times}$ with $q^{\Z}$ if $F$ is non-archimedean and with $\R_{>0}^{\times}$ if $F$ is archimedean. We can thus redefine a metric to be a $F^{\times}$-equivariant map $L^{\neq 0} (F) \to F^{\times}/\mathcal{O}^{\times}$.
\end{remark}
\begin{example}
	Consider the $\G_m$-torsor $\A^{n + 1}_F \setminus \{0\} \to \Proj^n_F$, where $\G_m$ acts on $\A^{n + 1}_F$ by coordinate-wise multiplication. This corresponds to the line bundle $\mathcal{O}(-1)$. A simple example of a metric on this torsor is given by the following formula
	\begin{equation*}
		(\A^{n + 1}_F \setminus \{0\})(F) =F^{n + 1} \setminus \{0\} \to \R_{>0}^{\times}: (x_0, \cdots, x_n) \to \max(|x_0|, \cdots, |x_n|).
	\end{equation*}
\end{example}
This definition via torsors easily generalizes to torsors under groups of multiplicative type
\begin{definition}
	Let $T$ be a group of multiplicative type. A \emph{metric $||\cdot||$} on a $T$-torsor $\pi: Y \to X$ is a $T(F)$-equivariant map 
	\begin{equation*}
		||\cdot||: \coprod_{a \in H^1(F, T)} Y_{a}(F) \to T(F)/T(\mathcal{O}).
	\end{equation*} 
\end{definition}

Torsors in the \'etale topology are torsors in the analytic topology in the following sense, we will use this fact freely in what follows.
\begin{lemma}
	The map $\pi:\coprod_{a \in H^1(F, T)} Y_{a}(F) \to X(F)$ is a $T(F)$-principal bundle and the induced map $\pi:\coprod_{a \in H^1(F, T)} Y_{a}(F)/T(\mathcal{O}) \to X(F)$ is a $T(F)/T(\mathcal{O})$-principal bundle in the analytic topology.
	\label{Torsors are principal bundles in the analytic topology}
\end{lemma}
\begin{proof}
	The first statement is proven in \cite[p. 127]{Salberger1998Tamagawa}. The second follows immediately from the first.
\end{proof}
Such torsors always split.
\begin{lemma}
 	Let $\pi: Y \to X$ be any $T(F)/T(\mathcal{O})$-principal bundle of $F$-analytic manifolds. Then $\pi$ has a continuous section.
	\label{Metrics always exist}
\end{lemma}
\begin{proof}
	Let $\bigcup_{i \in I} U_i = X$ be an open cover such that we have local sections $\psi_i: U_i \to Y$. Choose a partition of unity $(\varphi_i: X \to \R)_{i \in I}$ subordinate to this cover, i.e.~the support of $\varphi_i$ is contained in $U_i$ and $\sum_{i \in I} \varphi_i = 1$. If $F$ is non-archimedean then we may also assume that the image of $\varphi_i$ is contained in $\Z$.
	
	The group $T(F)/T(\mathcal{O})$ is an $\R$-vector space if $F$ is archimedean so the sum $\sum_{i \in I} \varphi_i \psi_i$ is a well-defined section of $\pi$ in all cases.
\end{proof}

Salberger \cite[\S 9]{Salberger1998Tamagawa} observed that one can associate analytic sections to metrics and made this explicit for a certain special metric for toric varieties. This correspondence holds in a much greater generality.
\begin{proposition}
	There is an explicit bijection between metrics $||\cdot||: \coprod_{a \in H^1(F, T)} Y_{a}(F) \to T(F)/T(\mathcal{O})$ and sections $\varphi$ of the $T(F)/T(\mathcal{O})$-torsor $\pi: \coprod_{a \in H^1(F, T)} Y_{a}(F)/T(\mathcal{O}) \to X(F)$.
	
	If the section $\varphi$ and metric $||\cdot||$ correspond to each other then $|| \varphi (x)|| = 1$ for all $x \in X(F)$ and $y = ||y|| \cdot \varphi(\pi(y))$ for all $y \in \coprod_{a \in H^1(F, T)} Y_{a}(F)/T(\mathcal{O})$.
\end{proposition}
\begin{proof}
	A metric is the same as a trivialization of the $T(F)/T(\mathcal{O})$-principal bundle 
	\[
	\coprod_{a \in H^1(F, T)} Y_{a}(F)/T(\mathcal{O}) \to X(F)
	\]
	so this is just the standard correspondence between trivializations and sections.
\end{proof}
\begin{example}
	Let $F$ be a non-archimedean field and $T$ be a group of multiplicative type over $F$ with good reduction. Let $\pi: Y \to X$ be a $T$-torsor over $F$ and $\pi: \mathcal{Y} \to \mathcal{X}$ an integral model of this torsor. Assume also that $\mathcal{X}$ is proper so that $\mathcal{X}(\mathcal{O}) = X(F)$ by the valuative criterion of properness.
	
	The map $\coprod_{a \in H^1(\mathcal{O}, T)} \mathcal{Y}_{a}(\mathcal{O}) \to \mathcal{X}(\mathcal{O})$ is a $T(\mathcal{O})$-torsor. We can thus define a section via the following composition
	\begin{equation*}
		X(F) = \mathcal{X}(\mathcal{O}) \cong \coprod_{a \in H^1(\mathcal{O}, T)} \mathcal{Y}_{a}(\mathcal{O})/ T(\mathcal{O}) \subset \coprod_{a \in H^1(F, T)} Y_{a}(F)/T(\mathcal{O}).
	\end{equation*}

	This section induces a metric $||\cdot||$ which is uniquely determined by the property that $||y|| = 1$ if and only if $y \in \mathcal{Y}_{a}(\mathcal{O})$ for some $a \in H^1(\mathcal{O}, T)$.
\end{example}
If $T = \G_m$ then this is the metric corresponding to an integral model of a line bundle \cite[\S2.1.5]{Chambert-Loir2010Igusa}. This suggest the following definition.
\begin{definition}
	Let $F$ be a non-archimedean field and $T$ be a group of multiplicative type over $F$ with good reduction. Let $\pi: Y \to X$ be a $T$-torsor over $F$ and $\pi: \mathcal{Y} \to \mathcal{X}$ an integral model, which is a $T$-torsor over $\mathcal{O}$.
	
	 A metric $||\cdot||: \coprod_{a \in H^1(F, T)} Y_{a}(F) \to T(F)/T(\mathcal{O})$ is \emph{$\mathcal{Y}$-integral} if $||y|| = 1$ for all $a \in H^1(\mathcal{O}, T)$ and $y \in \mathcal{Y}_{a}(\mathcal{O})$.
	 
	 A section $\psi: X(F) \to \coprod_{a \in H^1(F, T)} Y_{a}(F)/T(\mathcal{O})$ is \emph{$\mathcal{Y}$-integral} if the image of $\mathcal{X}(\mathcal{O})$ under $\psi$ is equal to $\coprod_{a \in H^1(\mathcal{O}, T)} \mathcal{Y}_{a}(\mathcal{O})/ T(\mathcal{O})$.
\end{definition}
It is clear that a metric is $\mathcal{Y}$-integral if and only if the corresponding section is $\mathcal{Y}$-integral. Note that there can be multiple $\mathcal{Y}$-integral metrics if $\mathcal{X}$ is not proper.

Sections are functorial in $T$, which implies that metrics are functorial in $T$ too.
\begin{construction}
	Let $T, T'$ be groups of multiplicative type, $f: T \to T'$ a group homomorphism and $\pi: Y \to X$ a $T$-torsor. Consider the induced $T'$-torsor $Y' = T' \times^T Y$. The $T$-equivariant map $Y \to Y'$ induces after twisting a $T(F)$-equivariant map coming component-wise from \eqref{Functoriality of points on torsors}
	\begin{equation*}
		\coprod_{a \in H^1(F, T)} Y_{a}(F) \to \coprod_{a' \in H^1(F, T')} Y'_{a'}(F).
	\end{equation*}
	
	Any section $\psi: X(F) \to \coprod_{a \in H^1(F, T)} Y_{a}(F)/T(\mathcal{O})$ can be composed with this map to get a section $\psi': X(F) \to \coprod_{a' \in H^1(F, T')} Y'_{a'}(F)/T'(\mathcal{O})$.
	
	Let $||\cdot||: \coprod_{a \in H^1(F, T)} Y_{a}(F) \to T(F)/T(\mathcal{O})$ and $||\cdot||': \coprod_{a' \in H^1(F, T')} Y'_{a'}(F) \to T'(F)/T'(\mathcal{O})$ be the metrics corresponding respectively to the sections $\psi$ and $\psi'$. It follows directly from the construction that the following diagram commutes.
	\begin{equation*}
		\begin{tikzcd}
			{\coprod_{a \in H^1(F, T)} Y_{a}(F)} & {\coprod_{a' \in H^1(F, T')} Y'_{a'}(F)} \\
			{T(F)/T(\mathcal{O})} & {T'(F)/T'(\mathcal{O})}
			\arrow[from=1-1, to=1-2]
			\arrow["{||\cdot||'}", from=1-2, to=2-2]
			\arrow["{||\cdot||}"', from=1-1, to=2-1]
			\arrow["f", from=2-1, to=2-2]
		\end{tikzcd}
	\end{equation*}
	\label{Functoriality of metrics and sections}
\end{construction}
An important special case is when $T' = \G_m$ and we have a character $\chi: T \to \G_m$. The above construction then produces a section and a metric on a $\G_m$-torsor. This allows us to think of a metric on a $T$-torsor as a collection of metrics on line bundles, one for each character of $T$. This is similar to the notion of a system of metrics \cite[\S4]{Peyre1998Terme}, but metrics on torsors are a more natural notion. 

The above construction also behaves well with respect to integrality.
\begin{lemma}
	Let $F$ be non-archimedean and $T, T'$ groups of multiplicative type with good reduction. Let $\pi: Y \to X$ a $T$-torsor with integral model $\pi: \mathcal{Y} \to \mathcal{X}$. The induced $T'$-torsor $Y' = T' \times_T Y$ has an induced integral model $\mathcal{Y} := T' \times_T \mathcal{Y}$.
	
	If a section $\psi: X(F) \to \coprod_{a \in H^1(F, T)} Y_{a}(F)/T(\mathcal{O})$ is $\mathcal{Y}$-integral then the section $\psi'$ of Construction \ref{Functoriality of metrics and sections} is $\mathcal{Y}'$-integral.
	\label{Integrality of metrics behaves well with respect to functoriality}
\end{lemma}
\begin{proof}
	Follows immediately from the definitions and the fact that \eqref{Functoriality of points on torsors} is functorial in $S$.
\end{proof}

Let us also note that if $T \to T'$ is surjective then any metric on a $T'$-torsor comes from a $T$-torsor.
\begin{lemma}\label{lem:Metrics can be extended}
	In the situation of Construction \ref{Functoriality of metrics and sections} assume that $f: T \to T'$ is surjective. Let $||\cdot ||': \coprod_{a' \in H^1(F, T')} Y'_{a'}(F) \to T'(F)/T'(\mathcal{O})$ be a metric. Then there exists a metric $||\cdot||: \coprod_{a \in H^1(F, T)} Y_{a}(F) \to T(F)/T(\mathcal{O})$ such that $||\cdot||'$ is induced by $||\cdot||$ via Construction \ref{Functoriality of metrics and sections}.
	
	Moreover, if $T$ and $T'$ have good reduction and $\mathcal{Y} \to \mathcal{X}$ is an integral model with induced integral model $\mathcal{Y'} \to \mathcal{X}$ such that $||\cdot||'$ is $\mathcal{Y}'$-integral then we may take $||\cdot||$ to be $\mathcal{Y}$-integral. 
\end{lemma}
\begin{proof}
	Let $G := \ker(T \to T')$. We claim that for each $a \in \HH^1(F, T)$ the map $Y_a \to Y_{a'}$ is a $G$-torsor. This can be checked \'etale locally on $X$ so we may assume that $Y_a = T \times X$ and $Y'_{a'} = T' \times Y$. The map $T \to T'$ is a $G$-torsor since it is surjective so the claim follows.
	
	The map $\coprod_{a \in H^1(F, T)} Y_{a}(F) \to \coprod_{a' \in H^1(F, T')} Y'_{a'}(F)$ is thus a $G(F)$-principal bundle so 
	\begin{equation*}
		\coprod_{a \in H^1(F, T)} Y_{a}(F)/T(\mathcal{O}) \to \coprod_{a' \in H^1(F, T')} Y'_{a'}(F)/T'(\mathcal{O})
	\end{equation*}
	is a $G(F)/\ker(T(\mathcal{O}) \to T'(\mathcal{O})) = G(F)/G(\mathcal{O})$-principal bundle. By Lemma \ref{Metrics always exist} this map has a section $\sigma$.
	
	By composing $\sigma$ with the section $\psi': X(F) \to \coprod_{a' \in H^1(F, T')} Y'_{a'}(F)/T'(\mathcal{O})$ corresponding to $||\cdot||'$ we construct a section $\psi: X(F) \to \coprod_{a \in H^1(F, T)} Y_{a}(F)/T(\mathcal{O})$ which corresponds to a metric with the desired properties.
	
	If the metric $||\cdot||'$ is integral then we may take $\psi$ such that when restricted to the clopen subset $\mathcal{X}(\mathcal{O})$ it is given by $\mathcal{X}(\mathcal{O}) \cong \coprod_{a \in H^1(\mathcal{O}, T)} \mathcal{Y}_{a}(\mathcal{O})/T(\mathcal{O}) \subset \coprod_{a \in H^1(F, T)} Y_{a}(F)/T(\mathcal{O})$.
\end{proof}
\begin{remark}
	Applying this to $T' = 1$ we deduce that any torsor has a metric.
\end{remark}

\subsection{Adelic metrics and heights}
The story in the previous section can also be done adelically. For this section we let $F$ be a number field, $\mathbf{T}$ an adelic group of multiplicative type over $F$, $\mathbf{X}$ an adelic variety over $F$ and $\bm{\pi}: \mathbf{Y} \to \mathbf{X}$ a $\mathbf{T}$-torsor.
\begin{definition}
	An \emph{adelic metric} on $\bm{\pi}$ is a collection $(||\cdot||_v)_{v \in \Omega_F}$ where $||\cdot||_v$ is a metric on $\pi_v: Y_v \to X_v$ such that there exists an integral model $\pi: \mathcal{Y} \to \mathcal{X}$ of $\pi_F: Y_F \to X_F$ such that $||\cdot||_v$ is $\mathcal{Y}_{\mathcal{O}_v}$-integral for all but finitely many places $v$.
\end{definition}
Such metrics are functorial in $T$.
\begin{construction}
	Let $\mathbf{f}: \mathbf{T} \to \mathbf{T}'$ be a map of adelic groups of multiplicative type. We can consider the \emph{adelic contracted product}
	\begin{equation*}
		\mathbf{Y}' := \mathbf{Y} \times^{\mathbf{T}} \mathbf{T'} := (Y_F \times_{T_F} T'_F, (Y_v \times_{T_v} T'_v)_{v \in \Omega_F}, (\xi_{Y, v} \times_{T_v} T'_v)_{v \in \Omega_F}).
	\end{equation*}

	Applying Construction \ref{Functoriality of metrics and sections} to every local metric $||\cdot||_v$ for each $v$ provides a local metric $||\cdot||_v$ on the torsor $Y_{L, v}$ for all $v \in \Omega_F$. It follows immediately from Lemma \ref{Integrality of metrics behaves well with respect to functoriality} and the definitions that this collection forms an adelic metric on $\mathbf{Y}'$.
	\label{Functoriality of adelic metrics}
\end{construction}

Given an adelic metric we can define a total metric and a height function. In what follows we will slightly abuse notation and use $\coprod_{\mathbf{a} \in H^1(\A_F, \mathbf{T})} \mathbf{Y}_{\mathbf{a}}(\A_F)$ to denote the restricted product of $\coprod_{a_v \in H^1(F_v, T_v)} Y_{v, a_v}(F_v)$ for all $v \in \Omega_{F}$ with respect to the compact subspaces $\coprod_{a_v \in H^1(\mathcal{O}_v, T_v)} \mathcal{Y}_{a_v}(\mathcal{O}_v)$ for $v \not \in S$ and any $\mathcal{O}_{F, S}$-integral model $\mathcal{Y} \to \mathcal{X}$.
\begin{construction}
	Let $(||\cdot||_v)_{v \in\Omega_F}$ be an adelic metric. Taking their restricted product we get the \emph{total metric}
	\begin{equation*}
		|| \cdot|| := \prod_v ||\cdot||_v: \coprod_{\mathbf{a} \in H^1(\A_F, \mathbf{T})} \mathbf{Y}_{\mathbf{a}}(\A_F) \to \mathbf{T}(\A_F)/\mathbf{T}(\A_{\mathcal{O}}).
	\end{equation*}

	We define the \emph{adelic height} corresponding to $(||\cdot||_v)_{v \in\Omega_F}$ as the composition 
	\begin{equation*}
		\mathbf{H} := |\cdot|_{\mathbf{T}}^{-1} \circ ||\cdot|| :\coprod_{\mathbf{a} \in H^1(\A_F, \mathbf{T})} \mathbf{Y}_{\mathbf{a}}(\A_F) \to \Hom(\hat{\mathbf{T}}, \R_{>0}^{\times}).
	\end{equation*}
	Given $L \in \hat{\mathbf{T}}_{\C}$ we denote by $\mathbf{H}(\cdot)^{L}$ the evaluation of $\mathbf{H}$ at $L$, i.e.~$|\cdot|_{\mathbf{T}}^{-L} \circ ||\cdot||$.
	
	The map $||\cdot||$ is well-defined due to the existence of an integral model in the definition of adelic metric.
	
	Both $||\cdot||$ and $\mathbf{H}$ are $\mathbf{T}(\A_F)$-equivariant, with $t \in \mathbf{T}(\A_F)$ acting by $|t|_{\mathbf{T}}^{-1}$ on $\Hom(\hat{\mathbf{T}}, \R_{>0}^{\times})$, by construction.
\end{construction}

If $x \in X_F(F)$ then there exists a unique $a \in H^1(F, T_F)$ and a point $y \in Y_a(F)$ such that $\pi_{a}(y) = x$. The point $y$ is moreover unique up to multiplication by $t \in T(F)$. But $|t|_T = 1$ by the product formula so $\mathbf{H}(y)$ is independent of the choice of $y$. The following is thus well-defined.
\begin{definition}
	Given $x \in X_F(F)$ we define the \emph{height} of $x$ with respect to the adelic metric $(||\cdot||_v)_{v \in\Omega_F}$ as $\mathbf{H}(x) := \mathbf{H}(y)$ for any choice of $y \in Y_a(F)$ such that $\pi_{a}(y) = x$.
\end{definition}

Unfolding the definitions we see that if $T = \G_m$ then this is equivalent to the usual definition of the height \cite[\S2.2.6]{Chambert-Loir2010Igusa} in terms of adelically metrized line bundles. 

\begin{remark}	
	 Instead of defining the height function by an adelic metric one could define an adelic height function on the torsor $\mathbf{Y} \to \mathbf{X}$ to be any $\mathbf{T}(\A_F)$-equivariant map $\mathbf{H}: \coprod_{\mathbf{a} \in H^1(\A_F, \mathbf{T})} \mathbf{Y}_{\mathbf{a}}(\A_F) \to \Hom(\hat{\mathbf{T}}, \R_{>0}^{\times})$.
\end{remark}

Given any character $\chi \in \hat{\mathbf{T}}$ we can apply Construction \ref{Functoriality of adelic metrics} to get an adelic metric on a $\G_m$-torsor, which thus gives a height function coming from an adelically metrized line bundle. It follows directly from the constructions that this height is equal to specializing $\mathbf{H}$ at $\chi$.

The height has the following positivity property.
\begin{proposition}
	Let $\mathbf{D} \subset \mathbf{X}$ be an effective divisor on an adelic variety $\mathbf{X}$. Let $\bm{\pi}: \mathbf{Y} \to \mathbf{X}$ be the $\G_m$-torsor corresponding to the line bundle $\mathcal{O}(\mathbf{D})$.
	
	Let $(||\cdot||_v)_{v \in \Omega_F}$ be an adelic metric on the torsor $\bm{\pi}$ and $C \subset \mathbf{X}(\A_F)$ a compact subset. There exists a constant $\delta > 0$ such that for all $x \in (X_F \setminus D_F)(F) \cap C$ we have $\mathbf{H}(x) > \delta \in \Hom(\Z, \R_{>0}^{\times}) = \R_{>0}^{\times}$.
	\label{Height is bounded below}
\end{proposition}
\begin{proof}
	For each place $v$ let $1_{D_v} \in \Gamma(X_v,\mathcal{O}(D_v))$ be the canonical section defining $D_v$. This defines a $\G_m$-equivariant map $f_v: Y_v \to \A_{F_v}^1$, where $\lambda \in \G_m$ act by multiplication by $\lambda^{-1}$ on $\A^1$. These maps fit together into a $\G_m$-equivariant map of adelic varieties $\mathbf{f}: \mathbf{Y} \to \A_F^1$. It follows from the construction that $\mathbf{f}^{-1}(0) = \bm{\pi}^{-1}(\mathbf{D})$.
	
	Let $x \in (X_F \setminus D_F)(F) \cap C$ and $y \in Y(F)$ such that $\pi(y) = x$. We have the equality $y = \mathbf{H}(y)^{-1} \psi(x) \mod \G_m(\A_F)^1$ and since $\mathbf{f}$ is $\G_m$-equivariant this implies the equality $\mathbf{f}(y) = \mathbf{H}(y) \mathbf{f}(\psi(x))$ in $\A_F/\G_m(\A_F)^1$. Since $x \not \in D_F(F)$ we know that $\mathbf{f}(y) \neq 0$. Taking the norm we deduce the following equality from the product formula
	\begin{equation*}
		1 = |\mathbf{f}(y)| = \mathbf{H}(x) |\mathbf{f}(\psi(x))|.
	\end{equation*}
	
	The composition 
	\begin{equation*}
		\mathbf{X}(\A_F) \xrightarrow{\psi} \mathbf{Y}(\A_F)/\G_m(\A_{F})^1 \xrightarrow{\mathbf{f}} \A_F/\G_m(\A_{F})^1 \xrightarrow{|\cdot|} \R_{\geq 0}
	\end{equation*}
	is continuous so the image of $C$ is bounded from above, this implies that $|\mathbf{f}(\psi(x))|$ is bounded from above and thus that $\mathbf{H}(x)$ is bounded from below.
\end{proof}
\begin{corollary}
	Let $(||\cdot||_v)_{v \in \Omega_F}$ be an adelic metric on the $\mathbf{T}$-torsor $\mathbf{Y} \to \mathbf{X}$ and let $\chi_1, \cdots, \chi_n \in \hat{\mathbf{T}}$ be characters whose corresponding line bundles correspond to effective divisors $D_1, \cdots, D_n$. Let $\Lambda = \{\lambda \in \hat{T}^{*}_{\R}: c(\chi_i) \geq 0 \emph{ for all } i\}$ be the dual cone to the $\chi_1, \cdots, \chi_n$ and assume that it is full-dimensional.
	
	For any compact subset $C \subset \mathbf{X}(\A_F)$ a compact subset There exists an $\ell \in \hat{T}^{*}_{\R}$ such that for all $x \in (X_F \setminus \cup_{i = 1}^n D_i)(F) \cap C$ we have $\log \mathbf{H}(x) \in \ell + \Lambda$.
	\label{Height lies in shift in effective cone}
\end{corollary}
\begin{proof}
	Apply Proposition \ref{Height is bounded below} to the adelic metric induced on the $\G_m$ torsor corresponding to $\chi_i$ for all $i$ to find a constant $\delta_i > 0$ such that $\mathbf{H}(x)(\chi_i) > \delta_i$ for all $i$.
	
	 It thus suffices to choose $\ell \in \hat{T}^{*}_{\R}$ such that $\ell(\chi_i) \leq \log \delta_i$ for all $i$. Such an $\ell$ exists because $\Lambda$ is full-dimensional.
\end{proof}

\subsection{Tamagawa measures}
In \cite[\S 2.1, \S 2.4]{Chambert-Loir2010Igusa} Tamagawa measures are associated to metrics on line bundles. For the sake of clarity we rewrite all the relevant definitions in our language.

\subsubsection{Local residue measures}
\label{Local residue measures}
Let $F$ be a local field, $X$ a smooth geometrically integral variety and $D \subset X$ a divisor with geometrically simple normal crossings with geometric components $D_1, D_2, \cdots, D_n \subset X_{F^{\sep}}$. Let $Z := \bigcap_{i = 1}^n D_i$, it is $\Gamma_F$-invariant and thus defines a closed subvariety of $X$.

Let $\omega_X(D)$ be the sheaf of top degree differential forms with at most a single pole along $D$ and $\omega_X(D)^{\neq 0} \to X$ the corresponding $\G_m$-torsor. Let $||\cdot||: \omega_X(D)^{\neq 0}(F) \to \R^{\times}_{> 0}$ be a metric. 

Repeated applications of the adjunction formula, see \cite[\S 2.1.12]{Chambert-Loir2010Igusa}, defines an isomorphism $\omega_{X_{F^{\sep}}(D)}|_{Z^{F^{\sep}}} \cong \omega_{Z_{F^{\sep}}}$ which is canonical up to a sign. If $D_i$ is defined by the local equation $f_i$ then the inverse is explicitly given by lifting an $m := \dim Z$-form 
$\omega_Z$ on $Z$ to an $m$-form $\Tilde{\omega}$ on $X$ and sending
\[\omega_Z \to \omega_X = \Tilde{\omega} \wedge \frac{d f_1}{f_1} \wedge \cdots \wedge \frac{d f_n}{f_n}. \] 

Assume that $\omega_Z$ is defined over $F$. Then $\Gamma_F$ acts via a character $\Gamma_F \to \mu_2$ on $\omega_X$. There exists an $a \in F^{\sep, \times}$ such that $a \omega_X$ is defined over $F$ by Hilbert theorem 90.

On each open $V \subset Z(F)$ on which there exists a non-vanishing differential form $\omega_Z$, i.e.~a section $V \to \omega_Z^{\neq 0}(F)$, we define a local residue measure on $U$ by the formula $\tau_{||\cdot||} := |a|\frac{|\omega_Z|}{||a \omega_X||}$, where $|\omega_Z|$ is defined in the usual way, e.g. \cite[\S 2.1.7]{Chambert-Loir2010Igusa}. This measure is independent of the choice of $\omega_Z$ and $a$ since if $g: V \to F^{\times}$ and $\lambda \in F^{\times}$ then $|\lambda a|\frac{|g \omega_Z|}{||\lambda a g \omega_X||} = |a|\frac{|\omega_Z|}{||a \omega_X||}$.

These measures can thus be glued to a measure $\tau_{||\cdot||}$ on $Z(F)$. We will also denote the pushforward measure along the closed immersion $Z(F) \subset X(F)$ by $\tau_{||\cdot||}$.

\subsubsection{Global measures}
Let $F$ be a number field, $X$ a smooth proper geometrically integral variety and $D \subset X$ a divisor with geometrically simple normal crossings and $U := X \setminus D$. Let $\mathbf{A} = ((A_v, Z_v))_{v} \in \mathcal{C}_{\Omega_{F}^{\infty}}^{\an}(D)$ be a face.

For each place $v$ the divisor $D_{A_v} = \sum_{\alpha_v \in A_v} D_{\alpha_v}$ is defined over $F_v$ and $U_{Z_v} \setminus D_{A_v} = U_{F_v}$. We have an equality $\omega_{U_{Z_v}}(D_{A_v})|_{U_{F_v}} = \omega_{U_{F_v}}$ so the tuple 
\[
\omega_{(X; \mathbf{A})}^{\neq 0} := (\omega_U^{\neq 0}, (\omega_{U_{Z_v}}(D_{A_v})^{\neq 0})_{v \in \Omega_F})
\] forms an adelic variety. The map $\omega_{(X; \mathbf{A})}^{\neq 0} \to (X ; \mathbf{A})$ is a $\G_m$-torsor.

Let $(||\cdot||_v)_{v}$ be an adelic metric on the $\G_m$-torsor $\omega_{(X; \mathbf{A})}^{\neq 0} \to (X ; \mathbf{A})$. Applying the construction of \S\ref{Local residue measures} for each place $v$ defines a local Tamagawa measure $\tau_{||\cdot||_v}$ on $Z_v^{\circ}(F_v)$. 
\begin{definition}
	If $\HH^1(X, \mathcal{O}_X) = \HH^2(X, \mathcal{O}_X) = 0$ then the \emph{Tamagawa measure} $\tau_{\mathbf{H}}$ on $U(\A_F)_{\mathbf{A}} = \prod_{v \in \Omega_F} Z_v^{\circ}(F_v)$ is the product measure
	\begin{equation*}
		c_{\mathbf{A}}L^*(1, \Pic U_{\overline{F}})^{-1} L^*(1, \overline{F}[U]^{\times}/\overline{F}^{\times}) \prod_v L_v(1, \Pic U_{\overline{F}}) L^*(1, \overline{F}[U]^{\times}/\overline{F}^{\times})^{-1} \tau_{||\cdot||_v}.
	\end{equation*}
	where
	\begin{equation}\label{eq:Definition c_A}
		c_{\mathbf{A}} := \prod_{v \in \Omega_{F}^{\infty}} 2^{\#\{\alpha_v \in A_v/ \Gamma_v: F_{\alpha_v} = \R \}} (2 \pi)^{\#\{\alpha_v \in A_v/ \Gamma_v: F_{\alpha_v} = \C \}}.
	\end{equation}
	
	The pushforward of this measure along the closed inclusion $U(\A_F)_{\mathbf{A}} \subset (X ; \mathbf{A})(\A_F)$ will also be denoted by $\tau_{\mathbf{H}}$.
\end{definition}
This product measure is well-defined by \cite[Thm.~2.5]{Chambert-Loir2010Igusa}.
\begin{remark}
	The powers of $2$ and $2 \pi$ are introduced in \cite[p.~10]{Chambert-Loir2010Integral}. The authors also provide a definition in the case that $\mathbf{A} \in \mathcal{C}_{S}^{\an}(D)$ for $S \neq \Omega_{F}^{\infty}$, but it is unclear to the author if the $\log q_v$ should appear here or in the cone constant for Manin's conjecture.
\end{remark}

The Tamagawa measure only depends on the height $\mathbf{H}$ associated to the adelic metric $(||\cdot||_v)_{v}$. This explains our choice of notation $\tau_{\mathbf{H}}$.
\begin{lemma}\label{lem:Tamagawa measure only depends on heigth function}
	Let $(||\cdot||_v)_v$ and $(||\cdot||'_v)_v$ be two adelic metrics on $\omega_{(X; \mathbf{A})}^{\neq 0} \to (X ; \mathbf{A})$ with corresponding height functions $\mathbf{H}$ and $\mathbf{H'}$ respectively. The quotient $\mathbf{H}/\mathbf{H}'$ descends a continuous function $X(\A_F) \to \Hom(\Z, \R^{\times}_{> 0}) = \R^{\times}_{> 0}$ and $\tau_{\mathbf{H}} = \mathbf{H}/\mathbf{H}' \tau_{\mathbf{H}'}$.
\end{lemma}
\begin{proof}
	The statement about descent follows since $\omega_{(X; \mathbf{A})}^{\neq 0}(\A_F) \to (X ; \mathbf{A})(\A_F)$ is a principal $\G_m(\A_F)$-bundle and $\mathbf{H}/\mathbf{H}'$ is $\G_m(\A_F)$-invariant.
	
	There exist continuous functions $g_v: U_{Z_v}(F_v) \to \R_{> 0}$ for each $v$ such that $|| \cdot ||_v = (g_v \circ \pi_v) \cdot ||\cdot||_v'$ for all places $v$. By definition we then have $(g_v \circ \pi_v) \cdot \tau_{||\cdot||_v} = (g_v \circ \pi_v) \tau_{||\cdot||'_v}$.
	
	Taking products we get $\prod_v (g_v \circ \pi_v) \cdot \mathbf{H} = \mathbf{H}'$ and $\prod_v (g_v \circ \pi_v) \tau_{\mathbf{H}} = \tau_{\mathbf{H}'}$.
\end{proof}
\begin{remark}
	This lemma allows one to define Tamagawa measures for height functions which do not come from adelic metrics. If $(||\cdot||_v)_v$ is an adelic metric with corresponding height function $\mathbf{H}$ and $\mathbf{H'}$ is another height function then we put
	$\tau_{\mathbf{H}'} := \mathbf{H}/\mathbf{H}' \tau_{\mathbf{H}}$.
\end{remark}

\subsubsection{Tamagawa measures on adelic groups of multiplicative type}
A torus $T$ has a canonical metric on $\omega_T$ \cite[Def. 3.29]{Salberger1998Tamagawa}. The definition easily extends to groups of multiplicative type.
\begin{construction}\label{cons: Construction of local measures on groups of multiplicative type}
	Let $T$ be a group of multiplicative type over any field $F$ of rank $n$. There is a map $X^*(T) \to \HH^0(T_{F^{\sep}}, \Omega_{T_{F^{\sep}}/F^{\sep}}^1)$ sending a character $\chi: T_{F^{\sep}} \to \G_m$ to the $T$-invariant differential $d\chi /\chi$. Note that torsion characters are locally constant and are thus mapped to $0$.
	
	This map is clearly $\Gamma_F$-equivariant and thus descents to a map $X^{*}(T) \to \Omega_{T_{F}/F}^1$ of \'etale sheaves on $T$. Taking wedge products we get a map $\wedge^n X^*(T) \to \wedge^n \Omega_T^1 = \omega_T$. If $n = 0$ then this is the map $\wedge^0 X^*(T) = \Z \to \mathcal{O}_T = \wedge^0 \Omega_T^1$ sending $1$ to $1$. 
	
	The group $X^*(T)$ has rank $n$ so $\wedge^n X^*(T)$ modulo torsion is $\Z$ is with a potentially non-trivial $\Gamma_F$-action. One checks that the image of a generator of $\Z$ is a nowhere vanishing differential form, i.e.~is an element of $\HH^0(T,\omega_{T_{F^{\sep}}}^{\neq 0})$. We denote it by $d \log T$, it is unique up to multiplication by $-1$.
	
	If $F$ is a local field then extend the absolute value $|\cdot|_F$ to $F^{\sep}$. For any $T$-invariant nowhere vanishing differential form $\omega \in \HH^0(T,\omega_{T_{F}}^{\neq 0})$ let $a \in F^{\sep, \times}$ be such that $a \cdot \omega = d\log T$. Define the measure $\tau_T = |a|_F |\omega|$ on $T(F_v)$. The fact that $\omega$ is $T$-invariant implies that $\tau_T$ is a Haar measure.
	
	If $T$ has good reduction then the above constructions can all be done over $\mathcal{O}$. So if $\omega$ is defined over $\mathcal{O}$ then we may take $a \in F^{\sep}$ with $|a| = 1$.
\end{construction}

This construction has the following compatibility with exact sequences.
\begin{lemma}\label{lem:compatibility of tamagawa measures of groups of multiplicative type}
	Let $1 \to T' \to T \to T'' \to 1$ be an exact sequence of groups of multiplicative type over a local field. The measure $\tau_T''$ restricted to the image of $T(F)$ is equal to the quotient measure $\tau_T/\tau_{T'}$.
\end{lemma}
\begin{proof}
	Because $T$ is an algebraic group we can identify $\omega_T$ with its fiber at the identity. Under this identification we have an isomorphism $\omega_T \cong \omega_{T'} \otimes \omega_{T''}$ and it is clear by construction that $d \log_T = d \log_{T'} d \log_{T''}$ under this isomorphism. The lemma follows.
\end{proof}

To get an adelic Tamagawa measure on $T(F)$ we will still have to multiply by convergence factors. These will be determined by the following lemma
\begin{lemma}\label{lem:Computation of local volume of groups of multiplicative type}
	Let $T$ be a group of multiplicative type of dimension $n$ over a non-archimedean local field $F$ with ring of integers $\mathcal{O}$. Assume that $T$ has good reduction. Then we have $\tau_{T}(T(\mathcal{O})) = \# \pi_0(T)(F) L_F(1, X^*(T))^{-1}\Disc_{F}^{-\frac{n}{2}}$.
\end{lemma}
\begin{proof}
	Let $T' \subset T$ be the connected component of the identity. This is a torus which has good reduction and $X^*(T')_{\Q} \cong X^*(T)_{\Q}$.
	
	We know by Lemma \ref{Number of integral torsors of group of multiplicative type with good reduction} that $\HH^1(\mathcal{O}, T') = 0$. The exact sequence $1 \to T' \to T \to \pi_0(T) \to 1$ thus induces an exact sequence $1 \to T'(\mathcal{O}) \to T(\mathcal{O}) \to \pi_0(T)(\mathcal{O}) = \pi_0(T)(F) \to 1$. We have $\tau_{T}(T'(\mathcal{O})) = \tau_{T'}(T'(\mathcal{O})) = \Disc_{F}^{-\frac{n}{2}} L_F(1, X^*(T))^{-1}$ by \cite[\S 3.3]{Ono191Arithmetic} \footnote{The presence of $\Disc_{F}^{-\frac{n}{2}}$ is due to our choice of Haar measure on $F$, which is different to Ono's.} The lemma follows.
\end{proof}

This lemma implies that the following product measure exists.
\begin{definition}
	Let $\mathbf{T}$ be an adelic group of multiplicative type over a global field. The \emph{global Tamagawa measure} is defined as the following product measure on $\mathbf{T}(\A_F)$
	\[
	\tau_{\mathbf{T}} := L^*(1, X^*(T_F))^{-1} \prod_v (\# \pi_0(T)(F_v))^{-1} L_v(1, X^*(T_F)) \tau_{T_v}.
	\]
\end{definition}

\begin{definition}
	Assume that $T_F \to T_v$ is an isomorphism for all non-archimedean $v$. Then by Lemma \ref{Norm map is surjective} there exists an exact sequence 
	\[0 \to \mathbf{T}(\A_F)^1 \to \mathbf{T}(\A_F) \xrightarrow{\log |\cdot|_{\mathbf{T}}} \hat{\mathbf{T}}_{\R}^* \to 0\]
	
	Let $dx_{\hat{\mathbf{T}}}$ be the Haar measure on $\hat{\mathbf{T}}_{\R}^*$. We let $\tau_{\mathbf{T}^1}$ be the Haar measure on $\mathbf{T}(\A_F)^1$ such that $\tau_{\mathbf{T}} = \tau_{\mathbf{T}^1} dx_{\hat{\mathbf{T}}}$.
\end{definition}

The adelic groups of multiplicative type we will encounter will have a particular form which will allow us to compute the volume $\tau_{\mathbf{T}}^1(\mathbf{T}(\A_F)^1/T_F(F))$. This computation is the main result of this section.
\begin{proposition}\label{prop:Tamagawa volume of adelic groups of multiplicative type}
	Let $\mathbf{T}$ be an adelic group of multiplicative type. Assume that there exists a torus $T'$ over $F$ and a finite $\Gamma_F$-set $E$ such that there exists an exact sequence 
	\begin{equation}\label{eq:exact sequence of groups of multiplicative type with last one a Weil restriction}
		 1 \to T_F \to T' \to \G_m^E \to 1
	\end{equation}
	
	Assume moreover that for each place $v$ there exists a finite $\Gamma_v$-set $A_v$ such that $A_v = \emptyset$ when $v$ is non-archimedean and for each place $v$ there exist a short exact sequence of groups of multiplicative type
	\[1 \to T_v \to (T')_{F_v} \times \G_m^{A_v} \to \G_m^{E} \to 1\]
	which is compatible with the sequence \eqref{eq:exact sequence of groups of multiplicative type with last one a Weil restriction}.
	
	We then have 
	\[\tau_{\mathbf{T}^1}\left(\mathbf{T}(\A_F)^1/T(F)\right) = \tau(\mathbf{T}) \prod_{v \in \Omega_F^{\infty}} 2^{\# \{a_v \in A_v / \Gamma_v : F_{a_v} = \R\}} (2 \pi)^{\# \{a_v \in A_v / \Gamma_v : F_{a_v} = \C\}}.\]
\end{proposition}
\begin{proof}
	By Lemma \ref{lem: Shapiro's lemma} one has $\HH^1(F, \G_m^E) = \HH^1(F_v, \G_m^E)= 0$ and $\HH^1(F_v, \G_m^{A_v}) = 0$ for all places $v$. 
	
	We thus have a long exact sequence
	\begin{equation}\label{eq:long exact sequence for global points on adelic tori}
		1 \to T_F(F) \to T'(F) \to \G_m^E(F) \to \HH^1(F, T_F) \to \HH^1(F, T') \to 0
	\end{equation}
	and for each place $v$ we get a long exact sequence
	\begin{equation}\label{eq:long exact sequence for local points on adelic tori}
		1 \to T_v(F_v) \to T'(F_v) \times \G_m^{A_v}(F_v) \to \G_m^{E}(F_v) \to \HH^1(F, T_v) \to \HH^1(F, T') \to 0
	\end{equation}
	Moreover, if $v$ is a place of good reduction for all the relevant groups of multiplicative type we have an exact sequence
	\begin{equation}\label{eq:long exact sequence for local integral points on adelic tori}
		1 \to T_v(\mathcal{O}_v) \to T'(\mathcal{O}_v) \times \G_m^{A_v}(\mathcal{O}_v) \to \G_m^{E}(\mathcal{O}_v) \to \HH^1(\mathcal{O}, T_v) \to \HH^1(\mathcal{O}_v, T') \to 0
	\end{equation}
	
	The exact sequences \eqref{eq:long exact sequence for local points on adelic tori} and \eqref{eq:long exact sequence for local integral points on adelic tori} are compatible with the local Tamagawa measures by Lemma \ref{lem:compatibility of tamagawa measures of groups of multiplicative type}, where we consider the cohomology groups to be equipped with the counting measure. 
	
	Taking restricted products of \eqref{eq:long exact sequence for local points on adelic tori} we deduce that the following sequence is exact
	\begin{equation}\label{eq:long exact sequence for adelic points on adelic tori}
		1 \to \mathbf{T}(\A_F) \to T'(\A_F) \times \prod_{v \in \Omega_F^{\infty}} \G_m^{E_v}(F_v) \to \G_m^{A}(\A_v) \to \HH^1(\A_F, \mathbf{T}) \to \HH^1(\A_F, T') \to 0
	\end{equation}
	
	This sequence is compatible with the global Tamagawa measures, where we equip the cohomology groups with the measures of Construction \ref{con:Tamagawa measures on adelic cohomology} and $\G_m^{E_v}(F_v)$ with the local Tamagawa measure. Indeed, the convergence factors are compatible since $\# \pi_0(T_v)$ appears both in the measure on $\mathbf{T}(\A_F)$ and the measure on $\HH^1(\A_F, \mathbf{T})$ and $L_v(1, X^*(T_F)) \cdot L_v(1, \Z[E]) = L_v(1, X^*(T'))$ for all places $v$ and $L^*(1, X^*(T_F)) \cdot L^*(1, \Z[E]) = L^*(1, X^*(T'))$.
	
	Dually, Lemma \ref{lem: Shapiro's lemma} implies that $\HH^1(F, \Z[A]) = \HH^1(F_v, \Z[A_v]) = \HH^1(F_v, \Z[E_v]) = 0$ for all places $v$ so we have a short exact sequence 
	\[
	0 \to \Z[E / \Gamma_F] \to \hat{T'} \to \hat{T_F} \to 0
	\]
	and for each place $v$ a short exact sequence 
	\[
	0 \to \Z[E / \Gamma_v] \to \hat{T}'_{F_v} \times \Z[A_v / \Gamma_v] \to \hat{T_v} \to 0
	\]
	
	It follows from the definition of $\hat{\mathbf{T}}$ that there is a short exact sequence
	\[
	0 \to \Z[E / \Gamma_v] \to \hat{T}'_{F_v} \times \prod_{v \in \Omega_{F}^{\infty}} \Z[A_v / \Gamma_v] \to \hat{\mathbf{T}} \to 0
	\] 
	
	By Lemma \ref{lem:compatibility of measures on dual vector spaces in exact sequences} we get the following short exact sequence of $\R$-vector spaces equipped with compatible Haar measures 
	\begin{equation}\label{eq:exact sequence for images of norm map}
	0 \to \hat{\mathbf{T}}_{\R}^* \to \hat{T}^{'*}_{F_v, \R} \times \prod_{v \in \Omega_{F}^{\infty}} \R[A_v / \Gamma_v] \to \R[E / \Gamma_v] \to 0
	\end{equation}
	
	The logarithmic norm maps the exact sequence \eqref{eq:long exact sequence for adelic points on adelic tori} to \eqref{eq:exact sequence for images of norm map} and these maps are surjective by Lemma \ref{lem:Norm map is surjective}. Taking kernels we thus get an exact sequence
	\[
	\begin{split}
		1 &\to \mathbf{T}(\A_F)^1 \to T'(\A_F)^1 \times \prod_{v \in \Omega_F^{\infty}} \G_m^{A_v}(\mathcal{O}_v) \to \G_m^{E}(\A_v)^1 \\ &\to \HH^1(\A_F, \mathbf{T}) \to \HH^1(\A_F, T') \to 0
	\end{split}
	\]
	The relevant Tamagawa measures are again compatible.
	
	Taking the quotient of this by the exact sequence \eqref{eq:long exact sequence for global points on adelic tori} we get an exact sequence of compact groups with compatible measures. This implies the equality
	\begin{equation*}
		\begin{split}
			&\tau_{\mathbf{T}^1}\left(\mathbf{T}(\A_F)^1/T(F)\right) \tau_{\G_m^{A,1}}\left(\G_m^E(\A_F)^1/\G_m^E(F)\right) \frac{\tau(T')}{\# \Sha^1(F, X^*(T'))} = \\
			&\tau_{T'^1}\left(T'(\A_F)^1/T'(F)\right) \prod_{v \in \Omega_F^{\infty}} \tau_v(\G_m^{A_v}(\mathcal{O}_v)) \frac{\tau(\mathbf{T})}{\# \Sha^1(F, X^*(\mathbf{T}))}.
		\end{split}
	\end{equation*}

Now $\tau_{\G_m^{E,1}}\left(\G_m^E(\A_F)^1/\G_m^E(F)\right) = 1$, resp. $\tau_{T'^1}\left(T'(\A_F)^1/T'(F)\right) = \tau(T')$, by \cite[Thm. ~3.5.1]{Ono191Arithmetic}, resp.~\cite[Main Thm.]{Ono1963Tamagawa}, and Corollary \ref{Tamagawa number of tori}. 

Since $\G_m^{A_v}(\mathcal{O}_v)$ is a product of complex unit circles $S^1$ and the group $\{1,-1\}$, respectively equipped with the Lebesgue measure and the counting measure it only remains to show that $\# \Sha^1(F, X^*(T')) = \# \Sha^1(F, X^*(\mathbf{T}))$.

To do this we use that $\HH^1(F, \Z[E]) = \HH^1(F_v, \Z[E]) = \HH^1(F_v, \Z[A_v]) = 0$ for all $v$ to deduce the exactness of the following sequences 
\[
0 \to \HH^1(F, X^*(T')) \to \HH^1(F, X^*(T_F)) \to \HH^2(F, \Z[E])
\]
\[
0 \to \HH^1(F_v, X^*(T')) \to \HH^1(F_v, X^*(T_v)) \to \HH^2(F_v, \Z[E])
\]

The sequence $0 \to \Sha^1(F, X^*(T')) \to \Sha^1(F, X^*(\mathbf{T})) \to \Sha^2(F, \Z[E])$ is thus exact by a diagram chase, but $\Sha^2(F, \Z[E])$ is dual to $\Sha^1(F, \G_m^E) \subset \HH^1(F, \G_m^E) = 0$ by \cite[Thm.~8.6.7]{Neukirch2008Cohomology}. The proposition follows.
\end{proof}
\section{Universal torsors and the Brauer-Manin obstruction}
\subsection{The Brauer-Manin pairing}
We refer to \cite[\S 13]{Colliot2021Brauer} for the usual Brauer-Manin pairing for varieties. We will generalize this to adelic varieties. For this section fix a number field $F$ and a geometrically integral adelic variety $\mathbf{X}$ over $F$. Recall from class field theory that for each place $v \in \Omega_F$ there exists an injection $\text{inv}_v: \Br F_v \to \Q/\Z$ which is surjective if $v$ is finite

\begin{definition}
	The \emph{Brauer-Manin pairing} is defined as the pairing
	\begin{equation*}
		\langle \cdot, \cdot \rangle_{\text{BM}}: \Br \mathbf{X} \times \mathbf{X}(\A_F) \to \Q/\Z: ((A_F, (A_v)_{v \in \Omega_F}, (x_v)_{v \in \Omega_F}) \to \sum_{v \in \Omega_F} \text{inv}_v(A_v(x_v)).
	\end{equation*}
	It is linear in the second variable. It is well-defined and continuous for the usual reason, see \cite[Prop. 13.3.1]{Colliot2021Brauer}.
	
	The constant Brauer elements $\Br_0 \mathbf{X}$ are contained in the right kernel of this pairing due to \cite[Thm. 13.1.8]{Colliot2021Brauer}, we thus get an induced pairing $\mathbf{X}(\A_F) \times \Br \mathbf{X}/\Br_0 \mathbf{X} \to \Q/\Z$, which we will also call the Brauer-Manin pairing.
	
	The \emph{Brauer-Manin set} $\mathbf{X}(\A_F)^{\Br}$ is the left kernel of this pairing. By the usual argument $X_F(F) \subset \mathbf{X}(\A_F)^{\Br}$.
	More generally, for any subset $B \subset \Br \mathbf{X}$ or $B \subset \Br \mathbf{X}/\Br_0 \mathbf{X}$ we denote by $\mathbf{X}(\A_F)^{B}$ the elements of $\mathbf{X}(\A_F)$ orthogonal to $B$ with respect to the Brauer-Manin pairing. If $B = \Br_1 \mathbf{X}$ then we write $\mathbf{X}(\A_F)^{\Br_1} := \mathbf{X}(\A_F)^{\Br_1 \mathbf{X}}$ and call it the algebraic Brauer-Manin set.
\end{definition}

\subsubsection{Strong approximation}
Let $X$ be a geometrically integral smooth proper variety over $F$ and $D \subset X$ a geometrically simple normal crossing divisor. It is known that integral points on $U := X \setminus D$ tend to concentrate near maximal faces of the Clemens complex. The naive form of strong approximation, that $U(F)$ is dense in 
$U(\A_F)$, will thus often fail, the simplest example of this failure is $U = \A^1$ in which case the inclusion $\A^1(F) \subset \A^1(\A_F)$ is discrete. We propose the following alternative definition
\begin{definition}
	Let $\mathbf{A} \in \mathcal{C}^{\an}_{\Omega_{F}^{\infty}}(D)$ be a face of the Clemens complex. We say that $(X ; \mathbf{A})$ satisfies \emph{strong approximation} if the closure of $U(F)$ in $(X; \mathbf{A})(\A_F)$ contains $U(\A_F)_{\mathbf{A}}$.
\end{definition}

The Brauer-Manin set $(X ; \mathbf{A})(\A_F)^{\Br}$ is closed since the Brauer-Manin pairing is continuous and $U(F) \subset (X ; \mathbf{A})(\A_F)^{\Br}$. So if $U(\A_F)_{\mathbf{A}}^{\Br} := U(\A_F)_{\mathbf{A}} \cap (X ; \mathbf{A})(\A_F)^{\Br} \neq U(\A_F)_{\mathbf{A}}$ then $(X ; \mathbf{A})$ cannot satisfy strong approximation, in this case we will say that there is a \emph{Brauer-Manin obstruction to strong approximation}. If the closure of $U(F)$ in $(X; \mathbf{A})(\A_F)$ contains $U(\A_F)_{\mathbf{A}}^{\Br}$ then we say that the Brauer-Manin obstruction is the \emph{only obstruction to strong approximation}. If the closure of $U(F)$ in $(X; \mathbf{A})(\A_F)$ contains $U(\A_F)_{\mathbf{A}}^{\Br_1}$ then we say that the \emph{algebraic} Brauer-Manin obstruction is the only obstruction to strong approximation.
\subsection{Universal torsors}
Let us recall one of the definitions \cite[Lem.~2.3.1]{Skorogobatov2001Torsors} of the type of a torsor.
\begin{definition}
	Let $T$ be a group of multiplicative type over a field $F$ and $\pi: Y \to X$ a $T$-torsor over $F$. Its \emph{type} is the $\Gamma_F$-equivariant morphism $\texttt{type}(\pi): X^*(T) \to \Pic X_{F^{\sep}}$ which sends a character $\chi: T_{F^{\sep}} \to \G_m$ to the image of the class $[Y] \in H^1(X_{F^{\sep}}, T)$ under the map $\chi_*: H^1(X_{F^{\sep}}, T) \to H^1(X_{F^{\sep}}, \G_m) = \Pic X_{F^{\sep}}$.
\end{definition}

We can extend this definition to adelic schemes. Let $F$ be a number field, $\mathbf{X}$ an adelic scheme over $F$ and $\mathbf{T}$ an adelic group of multiplicative type over $F$. The following is well-defined since the type is functorial and invariant under base change of field.
\begin{definition}
	Let $\bm{\pi}: \mathbf{Y} \to \mathbf{X}$ be a $\mathbf{T}$-torsor. Its \emph{type} is the morphism of adelic modules 
	$\texttt{type}(\bm{\pi}): X^*(\mathbf{T}) \to \bm{\Pic} \mathbf{X}$ given by the tuple $(\texttt{type}(\pi_F), (\texttt{type}(\pi_v))_{v \in \Omega_F})$.
\end{definition}
If a variety has a rational point then torsors of every possible type exist. This also holds for adelic varieties.
\begin{lemma}
	Let $\mathbf{X}$ be an adelic scheme, $\mathbf{L}$ a finitely generated adelic module $\mathbf{T} := D(\mathbf{L})$ the Cartier dual and $\bm{\lambda}: \mathbf{L} \to \bm{\Pic} \mathbf{X}$ a morphism of adelic modules.
	
	If $X_F(F) \neq \emptyset$ then there exists a torsor of type $\lambda$.	
	\label{Torsors of abitrary type exist if there is a rational point}
\end{lemma}
\begin{proof}	
	Let $x \in X_F(F)$. By \cite[p. 426]{Colliot1987Descente} there exist a $T_F$ torsor $\pi_F: Y_F \to X_F$ of type $\lambda_F$ such that $\pi_F^{-1}(x) \cong T_F$ which is unique up to (non-unique) isomorphism.
	
	Similarly, for every place $v$ there exists a $T_v$-torsor $\pi_v: Y_v \to X_v$ of type $\lambda_v$ such that $\pi_v^{-1}(\xi_v(x)) \cong T_v$ which is unique up to isomorphism.
	
	It follows from the uniqueness up to isomorphism that $(Y_F)_{F_v} \times^T T_v \cong Y_v \times_{X_v} (X_F)_{F_v}$ for all $v$. We can thus form an adelic variety $\mathbf{Y} := (Y_F, (Y_v)_v)$ and $\bm{\pi} := (\pi_F, (\pi_v)_v)$ is a $\mathbf{T}$-torsor of type $\bm{\lambda}$
\end{proof}
We can now define the notion of \emph{universal torsors}, which will play a crucial role in our argument.
\begin{definition}
	Assume that the adelic module $\bm{\Pic} \mathbf{X}$ is finitely generated as an abelian group\footnote{One could remove this assumption by replacing the geometric Picard group $\Pic X_{F^{\sep}}$ by the geometric N\'eron–Severi group.}. The \emph{N\'eron–Severi torus}\footnote{We warn that this is not necessarily a torus.} of $X$ is the adelic group of multiplicative type $\mathbf{T}_{\NS} := D(\bm{\Pic} \mathbf{X})$.

	A \emph{universal torsor} of $X$ is a $\mathbf{T}_{\NS}$-torsor $\bm{\pi}: \mathbf{Y} \to \mathbf{X}$ whose type $\texttt{type}(\bm{\pi})$ is the identity map $\bm{\Pic} \mathbf{X} \to \bm{\Pic} \mathbf{X}$. In other words, $Y_F \to X_F$ and $Y_v \to X_v$ are universal torsors for all places $v$.
\end{definition}
\begin{remark}
	Universal torsors are not unique, but they are unique up to twists, see \cite[(2.22)]{Skorogobatov2001Torsors}.
\end{remark}
One can recover the universal torsor on an open subvariety from the universal torsors of the original variety. The following lemma and construction are probably known to experts, but we were unable to find a reference.
\begin{construction}\label{cons:Divisor map}
	Let $X$ be a smooth variety over a field $K$ such that $\Pic X_{K^{\sep}}$ is finitely generated and torsion-free and $K^{\sep}[X]^{\times} = K^{\sep, \times}$. Let $T_{\NS X}$ be the N\'eron-Severi torus of $X$ and $\pi: Y \to X$ a universal torsor. Let $D \subset X$ be a divisor.
	
	Let $A$ be the $\Gamma_F$-set of irreducible components of the divisor $D_{K^{\sep}} \subset X_{K^{\sep}}$, for $a \in A$ let $D_a$ be the corresponding irreducible component. The pullback of $D_a$ along $Y_{K^{\sep}}$ is a principal divisor by \cite[Prop.~2.1.1]{Colliot1987Descente} so defines a map $Y_{K^{\sep}} \to \A_{K^{\sep}}^1$ such that the fiber at $0$ is $\pi^{-1}(D_a)$. This map is unique up to multiplication by a constant since by loc. cit. $K^{\sep}[Y]^{\times} = K^{\sep}$.
	
	Multiplying these maps for all $a \in A$ defines a map $Y_{F^{\sep}} \to \A_{K^{\sep}}^A$, unique up to multiplication by an element of $\G_m^A(K^{\sep})$. The obstruction to this map being $\Gamma_K$-invariant is thus given by a cocycle $\Gamma_K \to \G_m^A(K^{\sep})$. But $\HH^1(K, \G_m^A(K^{\sep})) = 0$ by Lemma \ref{lem: Shapiro's lemma} so after multiplication by an element of $\G_m^A(K^{\sep})$ the map becomes $\Gamma_K$-invariant and thus descends to a map $f_D:Y \to \A^A$, unique up to multiplication by $\G_m^{A}(K)$. 
	
	 Let $U := X \setminus D$. Assume that $K^{\sep}[U]^{\times} = K^{\sep, \times}$ and let $T_{\NS U}$ be the N\'eron-Severi torus of $U$. This implies that the map of $\Gamma_K$-modules $\Z[A] \to \Pic X_{K^{\sep}}: a \to [D_a]$ is injective. Its cokernel is $\Pic U_{K^{\sep}}$. Dually we get an exact sequence of groups of multiplicative type $1 \to T_{\NS U} \to T_{\NS X} \to \G_m^A \to 1$.
	 
	 The map $f:Y \to \A^A$ is $T_{\NS X}$-equivariant, where $T_{\NS X}$ acts on $\G_m^A$ through the inverse of the map $T_{\NS X} \to \G_m^A$. Indeed this can be checked after base change to $K^{\sep}$ in which case it can be reduced to the case when $D$ is irreducible. The statement is then an easy computation.
	 
	It follows from the construction that $\pi^{-1}(U) \cong f^{-1}(\G_m^A)$. Each fiber of the map $f: \pi^{-1}(U) \to \G_m^{A}$ has a $T_{\NS U}$-action. Note also that if $t \in \G_m(A)(F)$ and $g \in \text{im}(T_{\NS X}(K) \to \G_m^A(K))$ then multiplication by $g$ induces an isomorphism between $f^{-1}(t)$ and $f^{-1}(g \cdot t)$.
\end{construction}
\begin{lemma}\label{lem:Universal torsor of open subvariety}
	Let $t \in \G_m^A(K)$. Then $f^{-1}(t) \to U$ is a universal torsor.
	
	Moreover, if we fix for each $a \in \HH^1(K, T_{\NS X})$ a map $f_a: Y_a \to \G_m^A$ then every universal torsor of $U$ is of the form $f_a^{-1}(t)$ for some $t \in \G_m^A(K)$ and $a \in \HH^1(K, T_{\NS X})$. Moreover, the $t$ is unique up to multiplication by an element of $\emph{im}(T_{\NS X}(K) \to \G_m^A(K))$.
\end{lemma}
\begin{proof}
	The quotient $\pi^{-1}(U)/ T_{\NS U} \to U$ is a $\G_m^A$-torsor because $\pi^{-1}(U) \to U$ is a $T_{\NS}$-torsor. The map $\pi^{-1}(U) \to \G_m^A$ factors trough $\pi^{-1}(U)/ T_{\NS U} \to \G_m^A$ and is $\G_m^A$-equivariant. This implies that the torsor $\pi^{-1}(U)/ T_{\NS U} \to U$ is trivial so that the product map $\pi^{-1}(U)/ T_{\NS U} \to U \times \G_m^A$ is an isomorphism. Taking the fiber of $t$ we deduce that $f^{-1}(t)/T_{\NS U} \to U$ is an isomorphism, i.e.~$f^{-1}(t) \to U$ is a $T_{\NS U}$-torsor.
	
	Note that the above implies that $(f^{-1}(t) \times T_{\NS X})/T_{\NS U} \cong \pi^{-1}(U)$. We deduce that $\pi^{U} \to U$ is the image of $f^{-1}(t) \to U$ along the map $\HH^1(U, T_{\NS U}) \to \HH^1(U, T_{\NS X})$.
	
	The $T_{\NS X}$-torsor $\pi^{-1}(U) \to U$ has type $\Pic X_{K^{\sep}} \to \Pic U_{K^{\sep}}$ by functoriality of the type. We deduce from the above paragraph and the definition of the type that $f^{-1}(t) \to U$ is a universal torsor.
	
	To show the last statement it suffices to show that the twists of $f^{-1}(1)$ are always of the form $f_a^{-1}(t)$. Note that the group of automorphisms of the diagram $X \leftarrow Y \to \A^{A}$ which fix $X$, where $Y \to X$ is a $T_{\NS X}$-torsor and $Y \to \A^{A}$ is $T_{\NS X}$-equivariant map, is $\ker(T_{\NS X} \to \G_m^A) = T_{\NS U}$. 
	
	If $b \in \HH^1(K, T_{\NS U})$ then this implies that the twist of $f^{-1}(1) \to U$ by $b$ is received by twisting the entire diagram $X \leftarrow Y \to \A^{A}$. The twist of $Y$ is $Y_a$ where $a$ is the image of $b$ along $\HH^1(K, T_{\NS U}) \to \HH^1(K, T_{\NS X})$. The group $T_{\NS U}$ acts trivially on $\A^A$ so the twist of this diagram is of the form $X \leftarrow Y_a \xrightarrow{f_b} \A^{A}$. 
	
	The map $f_b$ is not necessarily equal to $f_a$, the two maps can differ by an element of $\G_m^A(K)$, . The two diagrams $X \leftarrow Y_a \xrightarrow{f_b} \A^{A}$ and $X \leftarrow Y_a \xrightarrow{f_b} \A^{A}$ are isomorphic if and only they differ by an element of $T_{\NS X}(K)$.
	
	We deduce that the twists of $X \leftarrow Y \to \G_m^{A}$ are all of the form $X \leftarrow Y_a \xrightarrow{t \cdot f_a} \A^{A}$ with $t \in \G_m^A(K)$ unique up to multiplication by $\text{im}(T_{\NS X}(K) \to \G_m^A(K))$. The corresponding twist of $f^{-1}(1)$ is given by $(t\cdot f_a)^{-1}(1) = f_a^{-1}(t^{-1})$.
	\end{proof}
	\begin{lemma}\label{lem:Divisor map smooths}
		If $D$ has geometrically simple normal crossings then the map $f:Y \to \A^A$ is smooth.
	\end{lemma}
	\begin{proof}
		Since smoothness can be checked after base change to $\overline{K}$ we may assume that $K = \overline{K}$. By induction we may assume that $A$ consists of a single element, i.e. $D$ is a smooth irreducible divisor.
		
		We have $f^{-1}(0) = \pi^{-1}(D)$ which is a smooth divisor in $Y$. On the other hand, by Lemma \ref{lem:Universal torsor of open subvariety}, for all $t \in \overline{K}^{\times}$ we have that $f^{-1}(t)$ is a $T_{\NS U}$-torsor over $U$. So each geometric fiber of $K$ is smooth of the same dimension. 
		
		The dimensions of the fibers being equal implies by miracle flatness that $f$ is flat. Since the fibers are also smooth this implies that $f$ is smooth.
	\end{proof}
\subsection{Adelic descent theory}
Universal torsors can be used to study the Brauer-Manin obstruction, as was first observed and exploited by Colliot-Th\'el\`ene and Sansuc \cite{Colliot1987Descente} for proper varieties $X$. This relation extends to adelic varieties	
\begin{theorem}\label{thm:Descent theory for adelic varieties}
	Let $\mathbf{X}$ be an adelic variety over a number field $F$. Assume that $\overline{F}_v[X_v]^{\times} = \overline{F}_v$ and that $\Pic (X_{v})_{\overline{F}_v}$ is finitely generated for all places $v \in \Omega_F$. Let $\bm{\pi}: \mathbf{Y} \to \mathbf{X}$ be a universal torsor. We then have 
	\begin{equation*}
		\mathbf{X}(\A_F)^{\Br_1 \mathbf{X}} = \bigcup_{a \in H^1(F, T_{\NS, F})} \bm{\pi}_a(\mathbf{Y}_a(\A_F)).
	\end{equation*}
\end{theorem}
We will prove this completely analogously to the case of (not necessarily proper) varieties in \cite[\S 6.1]{Skorogobatov2001Torsors}. Let us first set up some notation.

For $x_v \in X(F_v)$ let $Y_v(x_v) := \pi_v^{-1}(x_v)$. This is a $T_v$-torsor over $F_v$, namely the unique $T_v$-torsor such that $x_v \in \pi_{Y_v(x_v)}(F_v)$, see \cite[p. 22]{Skorogobatov2001Torsors}. If $T$ has good reduction at $v$, $\mathcal{Y} \to \mathcal{X}$ is an $\mathcal{O}_v$ integral model of the $T_v$-torsor, and if $x_v \in \mathcal{X}(\mathcal{O}_v)$ then $Y_v(x_v) \in \HH^1(\mathcal{O}_v, T_v)$, so taking restricted products of the $Y_v$ induces a continuous map 
\begin{equation}\label{Map sending a point to the twist it lies on}
	\mathbf{Y}(\cdot): \mathbf{X}(\A_F) \to \HH^1(\A_F, \mathbf{T}).
\end{equation}

Theorem \ref{thm:Descent theory for adelic varieties} will follow from the following lemma.
\begin{lemma}\label{lem:Brauer-Manin pairing agrees with Poitou-Tate pairing}
	For $x \in \mathbf{X}(\A_F)$ and $A \in \Br_1 \mathbf{X}$ we have $\langle A, x\rangle_{\emph{BM}} = \langle r(A),  \mathbf{Y}(x) \rangle_{\emph{PT}}$.
\end{lemma}
\begin{proof}
	The bilinear pairings $T_F \times \Pic (X_F)_{\overline{F}} \to \G_m$ and $T_v \times \Pic (X_v)_{\overline{F}_v} \to \G_m$ for $v \in \Omega_F$ induces compatible cup product pairings 
	\begin{align*}
		&\cdot \cup \; \cdot : \HH^1(F, \Pic (X_F)_{\overline{F}}) \times \HH^1(X, T_F) \to \HH^2(X_F, \G_m) = \Br(X_F)
		\\ & \cdot \cup \; \cdot :\HH^1(F, \Pic (X_v)_{\overline{F}_v}) \times \HH^1(X, T_v) \to \HH^2(X_v, \G_m) = \Br(X_v).
	\end{align*}
	These pairings factor through $\Br_1 X_F$ and $\Br_1 X_v$ respectively since every cocycle in $\HH^1(F, \Pic (X_F)_{\overline{F}})$ and $\HH^1(F, \Pic (X_v)_{\overline{F}_v})$ becomes trivial after base change to the algebraic closure.
	
	Consider the product pairing 
	\begin{equation*}
		\cdot \cup \; \cdot : \HH^1(\A_F, \bm{\Pic} \mathbf{X}) \times \HH^1(\mathbf{X}, \mathbf{T}) \to \Br_1(\mathbf{X}).
	\end{equation*}
	Note that $r(r(A) \cup [\mathbf{Y}]) = r(A)$. Indeed, unfolding the definitions this means that $r(r(A_F) \cup [Y_F]) = r(A_F)$ and $r(r(A_v) \cup [Y_v]) = r(A_v)$ and these follow from \cite[Thm. 4.1.1]{Skorogobatov2001Torsors}.
	
	By Lemma \ref{lem:Br_1 is isomorphic to H^1(Pic) adelically} we deduce that $A = r(A) \cup [\mathbf{Y}]$ in $\Br_1 \mathbf{X}/\Br F$. We thus have 
	\begin{equation*}
		\langle x, A \rangle_{\text{BM}} = \langle x, r(A) \cup [\mathbf{Y}] \rangle_{\text{BM}} = \sum_{v \in \Omega_F} \text{inv}_v(r(A) \cup Y_v(x_v)) = \langle r(A), \mathbf{Y}(x) \rangle_{\text{PT}}.
	\end{equation*} 
\end{proof}
\begin{proof}[Proof of Theorem \ref{thm:Descent theory for adelic varieties}]
	Let $x \in \mathbf{X}(\A_F)$. By Lemma \ref{lem:Brauer-Manin pairing agrees with Poitou-Tate pairing} we have $x \in \mathbf{X}(\A_F)^{\Br_1 \mathbf{X}}$ if and only if $\langle\mathbf{Y}(x), r(A) \rangle_{\text{PT}} = 0$ for all $A \in \Br_1 \mathbf{X}$. But $r$ is surjective by Lemma \ref{lem:Br_1 is isomorphic to H^1(Pic) adelically} so this is equivalent to the equality $\langle \mathbf{Y}(x), b \rangle_{\text{PT}} = 0$ for all $b \in \HH^1(F, \bm{\Pic} \mathbf{X})$. We deduce from Proposition \ref{Global Poitou Tate for adelic groups of multiplicative type} that this is equivalent to the existence of an $a \in \HH^1(F, T_{\NS, F})$ such that $\mathbf{Y}(x) = a$, i.e.~$x \in \pi_a(\mathbf{Y}_a(\A_F))$.
\end{proof}

\subsection{Tamagawa measures on universal torsors}
In this subsection we will construct Tamagawa measures on universal torsors and show that they are compatible with Tamagawa measures on the base. We start with the local case.
\subsubsection{Local Tamagawa measures on universal torsors}\label{section:Tamagawa measures local}
Let $F$ be a local field and $X$ a geometrically integral smooth variety over $F$. Assume that $\Pic X_{F^{\sep}}$ is finitely generated and $F^{\sep}[X]^{\times} = F^{\sep, \times}$. Let $T_{\NS}$ be the N\'eron-Severi torus and $\pi: Y \to X$ be a universal torsor.

Let $D \subset X$ be a geometrically simple normal crossing divisor and let $D_1, \cdots, D_n$ be the irreducible components of $D_{F^{\sep}}$. Let $Z := \cap_{i = 1}^n D_i$. This is $\Gamma_F$-invariant and thus defined over $F$.

Let us first recall the structure of $\omega_{Y}$.
\begin{lemma}\label{lem:Properties of omega_Y}
	Taking wedge products of differential forms locally defines  a canonical isomorphism $\pi^{*}(\omega_X(D)) \otimes \HH^0(T, \omega_T)^T \cong \omega_Y(\pi^{-1}(D))$.
	
	We also have $\omega_Y(\pi^{-1}(D)) \cong \mathcal{O}_Y$ and $\HH^0(Y, \omega_Y(\pi^{-1}(D))^{\neq 0}) \cong F^{\times}$.
\end{lemma}
\begin{proof}
	The first statement follows from the fact that $\omega_{Y/X} \cong \mathcal{O}_Y \otimes \HH^0(T, \omega_T)^T$ by taking determinants in \cite[Prop.~3.8, 3.9]{Salberger1998Tamagawa}.
	
	The second follows from the first, the fact that $F^{\sep}[Y]^{\times} \cong F^{\sep, \times}$ and that the map $\Pic X \to \Pic Y$ is zero, see \cite[Prop. 2.1.1]{Colliot1987Descente}.
\end{proof}
\begin{construction}
	Let $||\cdot||$ be an adelic metric on $\omega_X$.

	Let $\omega \in \omega_X(D)^{\neq 0}(F)$. Then $a^{-1} \pi^*\omega \wedge d\log_T \in \omega_Y(\pi^{-1}(D))^{\neq 0}$ for some $a \in F^{\sep, \times}$ by Lemma \ref{lem:Properties of omega_Y} and Construction \ref{cons: Construction of local measures on groups of multiplicative type}. By repeated applications of the adjunction formula as in \S\ref{Local residue measures} we can define a local Tamagawa measure $\tau_{Y, ||\cdot||} :=|a|\cdot |a^{-1} \pi^*\omega \wedge d\log_T|$ on $\pi^{-1}(Z)(F)$. This is independent of the choice of $\omega$, cf.~\S\ref{Local residue measures}.
	
	We will use the same notation $\tau_{Y, \omega}$ to denote the pushforward of this measure along the closed inclusion $\pi^{-1}(Z)(F) \subset Y(F)$.
	
	Since we can do the above construction for any universal torsor the above construction can be applied to define a measure on $\coprod_{a \in \HH^1(F, T_{\NS})} Y_a(F)$ which we will denote by $\tau_{\coprod_{a} Y_a, \omega}$.
\end{construction}

This measure is related to measures on the base.
\begin{lemma}\label{lem:measure on universal torsor is product of measure on base and measure on torus}
	For every open $V \subset Z(F)$ for which $\pi^{-1}(V) \cong V \times T_{\NS}(F)$ one has the equality of measures
	\[
	\tau_{Y, ||\cdot||}(x) = (\tau_{X, ||\cdot||} \times \tau_{T_{\NS}})(x).
	\]
\end{lemma}
\begin{proof}
	This follows immediately after unfolding the definitions.
\end{proof}
\subsubsection{Global Tamagawa measures on universal torsors}\label{section:Tamagawa measures global}
Let now $X$ be a geometrically integral smooth proper variety over a number field $F$. Let $D \subset X$ be a divisor with geometrically strict normal crossing and $\mathbf{A} \in \mathcal{C}^{\an}_{\Omega_{F}^{\infty}}$ a face of the analytic Clemens complex. Let $U := X \setminus D$ and assume that $\overline{F}[U]^{\times} = \overline{F}^{\times}$

Assume that $\HH^1(X, \mathcal{O}_X) = \HH^2(X, \mathcal{O}_X) = 0$. In particular $\Pic X$ is torsion-free and finitely generated. Let $\mathbf{T}_{\NS}$ be the N\'eron-Severi torus.

Let $\bm{\pi}:\mathbf{Y} \to (X, \mathbf{A})$ be a universal torsor and let $(||\cdot||_v)_{v \in \Omega_F}$ be an adelic metric on this torsor with corresponding height $\mathbf{H}$. The line bundle $\omega_{(X; \mathbf{A})} \in \Pic (X ; \mathbf{A})$ defines a character $\mathbf{T}_{\NS} \to \G_m$. It follows from the definition of type that there exists a $\mathbf{T}_{\NS}$-equivariant map $\mathbf{Y} \to \omega_{(X; \mathbf{A})}^{\neq 0}$, unique up to multiplication by $F^{\times}$.

The adelic metric $(||\cdot||_v)_{v \in \Omega_F}$ thus defines an adelic metric on $ \omega_{(X; \mathbf{A})}^{\neq 0}$, which by abuse of notation we will also denote by $(||\cdot||_v)_{v \in \Omega_F}$. It is only well-defined up to multiplication by $F^{\times}$, but this does not influence the heights or the adelic Tamagawa measures. 

We have the following bound on the local integrals.
\begin{lemma}\label{lem:local measure universal torsor}
	Let $S$ be a finite set of places containing the places of bad reduction of $T_{\NS}$ and such that $X$ has a smooth proper $\mathcal{O}_{D, S}$-model $\mathcal{X}$.
 	Let $\mathcal{D} \subset \mathcal{X}$ be the closure of $D$ and assume that the $T_{\NS}$-torsor $Y \to U$ has an $\mathcal{O}_{F, S}$-integral model $\mathcal{Y} \to \mathcal{U} := \mathcal{X} \setminus \mathcal{D}$ such that the metric $||\cdot||_v$ is $\mathcal{U}$-integral for all $v \not \in S$. For $a_v \in \HH^1(\mathcal{O}_v, T_{\NS})$ let $\pi_v: \mathcal{Y}_{a_v} \to \mathcal{U}$ be a twist of $\mathcal{Y}_{\mathcal{O}_v} \to \mathcal{U}_{\mathcal{O}_v}$. We have uniformly in $v$ and $a_v$ that
	\[
	\tau_{Y_{v, a_v}, ||\cdot||_v}(\mathcal{Y}_{a_v}(\mathcal{O}_v))= 1 + O(q_v^{- \frac{3}{2}}).
	\]
\end{lemma}
\begin{proof}
	This follows from \cite[Thm~2.5, Rem.~2.9(c)]{Chambert-Loir2010Igusa}.
\end{proof}
\begin{construction}
	Given $\mathbf{a} \in \HH^1(\A_F, \mathbf{T}_{\NS})$ we define a measure on $\mathbf{Y}_{\mathbf{a}}(\A_F)$ as the product
	\[
	\tau_{\mathbf{Y}_{\mathbf{a}},\mathbf{H}} = \prod_{v} \tau_{Y_{v, a_v}, ||\cdot||_v}.
	\]
	This product converges by the previous lemma.
	
	We also define a measure on $\coprod_{\mathbf{a} \in \HH^1(\A_F, \mathbf{T}_{\NS})} \mathbf{Y}_{\mathbf{a}}(\A_F)$ as the product measure 
	\[
	\tau_{\coprod_{\mathbf{a}} \mathbf{Y}_{\mathbf{a}}, \mathbf{H}} = \prod_{v} (\# \pi_0(T_{\NS})(\mathcal{O}_v))^{-1}\tau_{\coprod_{a_v} Y_{v, a_v}, ||\cdot||_v}.
	\] 
	This measure converges by the previous Lemma and Lemma \ref{Number of integral torsors of group of multiplicative type with good reduction}.
\end{construction}
\begin{remark}
	We again use $\mathbf{H}$ in the notation since these measures only depend on the height function $\mathbf{H}$ by the same reasoning as in Lemma \ref{lem:Tamagawa measure only depends on heigth function}. See Lemma \ref{lem:Comparison Gauge and Height measure} for another explanation of this fact for $\tau_{\mathbf{Y}_{\mathbf{a}},\mathbf{H}}$.
\end{remark}
We have the following relation between these measures.
\begin{proposition}\label{prop:Relation between the two Tamagawa measures on Universal torsors}
	Let $f:\coprod_{\mathbf{a} \in \HH^1(\A_F, \mathbf{T}_{\NS})} \mathbf{Y}_{\mathbf{a}}(\A_F) \to \C$ be a continuous $L^1$ function with respect to the measure $\tau_{\coprod_{\mathbf{a}} \mathbf{Y}_{\mathbf{a}}, \mathbf{H}}$. Then
	\[
	\int_{\coprod_{\mathbf{a} \in \HH^1(\A_F, \mathbf{T}_{\NS})} \mathbf{Y}_{\mathbf{a}}(\A_F)} f(y)
	d\tau_{\coprod_{\mathbf{a}} \mathbf{Y}_{\mathbf{a}}, \mathbf{H}}(y) =
	\int_{\HH^1(\A_F, \mathbf{T}_{\NS})} 
	\int_{\mathbf{Y}_{\mathbf{a}}(\A_F)} f(y) d\tau_{\mathbf{Y}_{\mathbf{a}}, \mathbf{H}}(y) d\mu_{\mathbf{T}_{\NS}}(\mathbf{a})
	\]
\end{proposition}
\begin{proof}
	Both $d\tau_{\coprod_{\mathbf{a}} \mathbf{Y}_{\mathbf{a}}, \mathbf{H}}(y)$ and $d\tau_{\mathbf{Y}_{\mathbf{a}}, \mathbf{H}}(y) d\mu_{\mathbf{T}_{\NS}}(\mathbf{a})$ are infinite product measures over the same factors, because the convergence factor $\# \pi_0(T_{\NS, v})$ appears both in $\tau_{\coprod_{\mathbf{a}} \mathbf{Y}_{\mathbf{a}}, \mathbf{H}}$ and $\mu_{\mathbf{T}_{\NS}}(\mathbf{a})$. Both sides are thus integrals over the same measure and thus equal.
\end{proof}
We can also relate the measure $\tau_{\coprod_{\mathbf{a}} \mathbf{Y}_{\mathbf{a}}, \mathbf{H}}$ to the Tamagawa measure on the base $(X ; \mathbf{A})$.
\begin{lemma}\label{lem:Tamagawa measure on universal torsor is product of base and torus}
	Let $V \subset (X ; \mathbf{A})(\A_F)$ be an open such that $\bm{\pi}^{-1}(V) \cong \mathbf{T}_{\NS}(\A_F) \times V$. Over $V$ we have the equality of measures
	\[
	c_{\mathbf{A}} \cdot \tau_{\coprod_{\mathbf{a}} \mathbf{Y}_{\mathbf{a}}, \mathbf{H}} = (\tau_{(X ; \mathbf{A}), \mathbf{H}} \times \tau_{\mathbf{T}_{\NS}}).
	\]
\end{lemma}
\begin{proof}
	This follows from Lemma \ref{lem:measure on universal torsor is product of measure on base and measure on torus} and noticing that the convergence factors in the constructions of
	 $\tau_{(X ; \mathbf{A}), \mathbf{H}}$ and $\tau_{\mathbf{N}_{\NS}}$ cancel with each other, except for the factor $c_{\mathbf{A}}$.
\end{proof}
An important consequence of this lemma is the following proposition.
\begin{proposition}\label{prop:Integral over universal torsor and height function}
	Let $f: (X ; \mathbf{A})(\A_F) \to \C$ be an $L^1$ continuous function with respect to the measure $\tau_{(X ; \mathbf{A}), \mathbf{H}}$ and $g: \Pic(X ; \mathbf{A}) \to \C$ a continuous $L^1$ function with respect to the measure $d \mathbf{z}_{\Pic(X ; \mathbf{A})}$. Then $f(\bm{\pi}(y)) g(\log \mathbf{H}(y))$ is $L^1$ for $y \in \coprod_{\mathbf{a} \in \HH^1(\A_F, \mathbf{T}_{\NS})} \mathbf{Y}_{\mathbf{a}}(\A_F)/T_{\NS}(F)$ with respect to the measure $\tau_{\coprod_{\mathbf{a}} \mathbf{Y}_{\mathbf{a}}, \mathbf{H}}$ and
	\[
	\begin{split}
		&\int_{\coprod_{\mathbf{a} \in \HH^1(\A_F, \mathbf{T}_{\NS})} \mathbf{Y}_{\mathbf{a}}(\A_F)/T_{\NS}(F)} f(\bm{\pi}(y)) g(\log \mathbf{H}(y))
		d\tau_{\coprod_{\mathbf{a}} \mathbf{Y}_{\mathbf{a}}, \mathbf{H}}(y) = \\ &\tau(\mathbf{T}_{\NS}) \int_{(X ; \mathbf{A})(\A_F)} f(x) d\tau_{(X ; \mathbf{A}), \mathbf{H}}(x) \int_{\Pic(X ; \mathbf{A})} g(z) d \mathbf{z}_{\Pic(X ; \mathbf{A})}.
	\end{split}
	\]
\end{proposition}
\begin{proof}
	By a partition of unity argument we may reduce to the case that $f$ is supported on an open $V \subset (X ; \mathbf{A})(\A_F)$ such that $\bm{\pi}^{-1}(V) \cong \mathbf{T}_{\NS}(\A_F) \times V$. 
	
	By the previous lemma and Fubini we may then first integrate over the fibers of $\bm{\pi}$. Because $\tau_{\mathbf{T}_{\NS}}$ is a Haar measure the integral over each fiber is $\int_{\mathbf{T}_{\NS}(\A_F)/ T_{\NS}(F)} g(|y|_{\mathbf{T}_{\NS}}) d\tau_{\mathbf{T}_{\NS}}$.
	
	The integrand is $\mathbf{T}_{\NS}(\A_F)^1$ invariant so by Proposition \ref{prop:Tamagawa volume of adelic groups of multiplicative type} we have 
	\[
	\int_{\mathbf{T}_{\NS}(\A_F)/ T_{\NS}(F)} g(|y|_{\mathbf{T}_{\NS}}) d\tau_{\mathbf{T}_{\NS}} = c_{\mathbf{A}} \tau(\mathbf{T}_{\NS})\int_{\Pic(X ; \mathbf{A})} g(z) d \mathbf{z}_{\Pic(X ; \mathbf{A})}. 
	\]
	
	The remaining integral over $V$ is then exactly $\int_{(X ; \mathbf{A})(\A_F)} f(x) d\tau_{(X ; \mathbf{A}), \mathbf{H}}(x)$ as expected.
\end{proof}
\subsubsection{Gauge forms}
Peyre \cite[Lem.~3.1.14]{Peyre2001Torseur} showed that in the case of proper varieties the Tamagawa measure on the universal torsor can be interpreted in terms of gauge forms of the universal torsor. In this part we will show how to extend this relation to the non-proper case. We keep the notation of the previous part.

Let us first note the following.
\begin{construction}\label{cons:measure gauge form}
	Let $\omega \in \HH^0(Y_F, \omega_{Y_F}^{\neq 0})$. We call such an $\omega$ a \emph{gauge form}. It exists and is unique up to multiplication by an element of $F^{\times}$ by Lemma \ref{lem:Properties of omega_Y}. 
	
	For each place $v$ let $f_{A_v}: Y_v \to \A_{F_v}^{A_v}$ be the map from Construction \ref{cons:Divisor map}. By Lemma \ref{lem:Universal torsor of open subvariety} we may assume that $f_{A_v}^{-1}(1) = (Y_F)_{F_v}$. Note also that $f_{A_v}^{-1}(0) = \pi_v^{-1}(Z_v^{\circ})$.
	
	The map $f_{A_v}$ is smooth and $\omega_{f_{A_v}} \cong \omega_{Y_v} \otimes \omega_{\A_{F_v}^{A_v}}^{-1} \cong \mathcal{O}_{Y_v}$ and $\HH^0(Y_v, \omega_{f_{A_v}}^{\neq 0}) \cong F_v^{\times}$ by Lemma \ref{lem:Properties of omega_Y}. Restriction to $f_{A_v}^{-1}(1) = (Y_F)_{F_v}$ defines a map $F_v^{\times} \cong \HH^0(Y_v, \omega_{f_{A_v}}^{\neq 0}) \to \HH^0(Y_F, \omega_{Y_F}^{\neq 0}) \cong F_v^{\times}$ which is $F_v^{\times}$-equivariant and thus a bijection. 
	
	Let $\tilde{\omega}_v \in \HH^0(Y_v, \omega_{f_{A_v}}^{\neq 0})$ be a lift of $\omega$ under this map and $\omega_v := \tilde{\omega}_v|_{f_{A_v}^{-1}(0)}$
	
	Each $\omega_v$ defines a measure $|\omega_v|$ on $f_{A_v}^{-1}(0)(F_v) = \pi_v^{-1}(Z_v^{\circ}(F_v))$, which we push forward to define a measure on $Y_v(F_v)$. The product $\prod_v |\omega_v|$ of these measures is well-defined by \cite[Rem.~2.9(c)]{Chambert-Loir2010Igusa} or Lemma \ref{lem:Comparison Gauge and Height measure} and independent of $\omega$ by the product formula. We will denote it by $\tau_{\mathbf{Y}}^{\text{gauge}}$.
\end{construction} 
\begin{remark}\label{rem:gauge forms}
	The above construction can be made concrete. Let $d X_{A_v} \in \HH^0(\A_{F_v}^{A_v}, \omega_{\A_{F_v}^{A_v}})$ be a top degree differential form. Then $\omega \wedge d X_{A_v}$ defines a top degree differential form on $Y_v$ locally around $f_{A_v}^{-1}(1)$. By the above arguments this extends uniquely to a gauge form on the entirety of $Y_v$ and there exists an $\omega_v \in \HH^0(\pi_v^{-1}(Z_v^{\circ}), \omega_{\pi_v^{-1}(Z_v^{\circ})})$ such that $\omega_v \wedge d X_{A_v} = \omega \wedge d X_{A_v}$ locally around $\pi_v^{-1}(Z_v^{\circ}) = f_{A_v}^{-1}(1)$.
\end{remark}
This is related to $\tau_{\mathbf{Y}, \mathbf{H}}$ by the following formula.
\begin{lemma}\label{lem:Comparison Gauge and Height measure}
	The equality $\tau_{\mathbf{Y}, \mathbf{H}} = \mathbf{H}^{K_{(X; \mathbf{A})}} \cdot \tau_{\mathbf{Y}}^{\emph{gauge}}$ of measures holds.
\end{lemma}
\begin{proof}
	If $X$ is proper and $D = \emptyset$ then this is \cite[Lem.~3.1.14]{Peyre2001Torseur}. The same proof works mutatis mutandis in the general case.
\end{proof}
\subsection{Brauer transform}
We keep the notation of \S\ref{section:Tamagawa measures global}. The following integrals play an important role in the leading constant of Manin's conjecture.
\begin{definition}
	For a place $v$ and $\beta_v \in \Br U_{F_v}$ and an $L^1$ function $f_v: U_{Z_v}(F_v) \to \C$ define the \emph{local Brauer transform} as
	\[
	\hat{\tau}_{U_{Z_v}, ||\cdot||_v}(f_v ; \beta_v) := \int_{U_{Z_v}(F_v)} f(x_v) e^{2 i \pi \mathrm{inv}_v(\beta_v(x_v))} d\tau_{U_{Z_v}, ||\cdot||_v}(x_v).
	\]
	
	For $\beta \in \Br (X ; \mathbf{A})$ and an $L^1$-function $f: (X ; \mathbf{A})(\A_F) \to \C$ we define the \emph{global Brauer transform} as
	\[
	\hat{\tau}_{(X ; \mathbf{A}), \mathbf{H}}(f ; \beta) := \int_{(X ; \mathbf{A})(\A_F)} f(x) e^{2 i \pi \langle \beta, x \rangle_{\text{BM}}} d\tau_{(X ; \mathbf{A}), \mathbf{H}}(x)
	\]
\end{definition}
Note that changing $\beta_v$ by a constant multiplies $\hat{\tau}_{U_{Z_v}, ||\cdot||_v}(f_v ; \beta_v)$ by a constant.
\begin{lemma}\label{lem:estimates for local Brauer transform}
	Let $S$ be a finite set such that $T_{\NS}$ has good reduction outside of $S$ and $X$ has a smooth proper $\mathcal{O}_{F, S}$-model $\mathcal{X}$.
	Let $\mathcal{D}$ be the closure of $D$ and write $\mathcal{U} := \mathcal{X} \setminus \mathcal{D}$. 
	
	 Assume that the metric $||\cdot||_v$ is $\mathcal{U}$-integral for all $v \not \in S$. For $\alpha_v \in \HH^1(\mathcal{O}_v, \Pic U_{\bar{F}_v})$ we have
	\[
	L_v(1, \Pic U_{\overline{F}})^{-1} \hat{\tau}_{U_{Z_v}, ||\cdot||_v}
	(\mathbf{1}_{\mathcal{U}(\mathcal{O}_v)} ; b_v \cup [Y_v]) = 
	\begin{cases}
		1 + &O(q_v^{-\frac{3}{2}}) \emph{ if } \alpha_v \in \HH^1(\mathcal{O}_v, \Pic U_{\bar{F}_v}) \\
		&O(q_v^{-\frac{3}{2}}) \emph{ if } \alpha_v \not \in \HH^1(\mathcal{O}_v, \Pic U_{\bar{F}_v}).
	\end{cases}
	\]
\end{lemma}
\begin{proof}
	We may enlarge $S$ to assume that the $T_{\NS}$-torsor $Y \to U$ has an integral model $\mathcal{Y} \to \mathcal{U}$. By combining Lemmas \ref{lem:measure on universal torsor is product of measure on base and measure on torus}, \ref{lem:Computation of local volume of groups of multiplicative type}, \ref{lem:local measure universal torsor} and the proof of Lemma \ref{lem:Brauer-Manin pairing agrees with Poitou-Tate pairing} we find that
	\[
	\begin{split}
		 \hat{\tau}_{U_{Z_v}, ||\cdot||_v}
		 (\mathbf{1}_{\mathcal{U}(\mathcal{O}_v)} ; b_v \cup [Y_v]) &= \frac{L_v(1, \Pic U_{\overline{F}})}{\# \pi_0(T_{\NS, v}(\mathcal{O}_v))}\sum_{a_v \in T_{\NS})} e^{2 i \pi_v \langle a_v, b_v \rangle_{\mathrm{PT}}} \int_{\mathcal{Y}_{a_v}(\mathcal{O}_v)} d\tau_{Y_{v, a_v}, ||\cdot||_v} \\ &= \frac{L_v(1, \Pic U_{\overline{F}})}{\# \pi_0(T_{\NS, v}(\mathcal{O}_v))} \sum_{a_v \in \HH^1(\mathcal{O}_v, T_{\NS})} e^{2 i \pi_v \langle b_v, a_v \rangle_{\mathrm{PT}}} \left(1 + O(q_v^{-\frac{3}{2}}) \right).
	\end{split}
	\]
	The lemma follows from character orthogonality, Lemma \ref{Number of integral torsors of group of multiplicative type with good reduction} and the fact that $\HH^1(\mathcal{O}_v, T_{\NS})$ is the orthogonal complement of $\HH^1(\mathcal{O}_v, \Pic U_{\bar{F}_v})$.
\end{proof}
\section{Integral Manin conjecture}
Let $F$ be a number field, $X$ a geometrically integral smooth proper variety over $F$, $D \subset X$ a geometrically strict normal crossing divisor and $U := X \setminus D$. Let $H: U(F) \to \R$ be a height function.

Let $\mathcal{C}^{\an}_{\Omega_{F}^{\infty}}(D) = \prod_{v \in \Omega_{F}^{\infty}} \mathcal{C}^{\an}_v(D)$ be the $\Omega_{F}^{\infty}$-analytic Clemens complex of the pair $(X,D)$. Note that for every place $v$ we have an open cover
\begin{equation*}
	X(F_v) = \bigcup_{(A_v, Z_v) \in \mathcal{C}^{\an}_v} U_{Z_v}(F_v) = \bigcup_{(A_v, Z_v) \in \mathcal{C}^{\an, \max}_v} U_{Z_v}(F_v).
\end{equation*}
Choose a compactly supported partition of unity with respect to this cover, i.e.~choose for each $(A_v, Z_v) \in \mathcal{C}^{\an, \max}_v(D)$ a compactly supported function $f_{(A_v, Z_v)}: U_{Z_v}(F_v) \to [0, 1]$ such that $\sum_{(A_v, Z_v) \in \mathcal{C}^{\an, \max}_v} f_{(A_v, Z_v)}(x_v) = 1$ for all $x_v \in X(F_v)$. 
 
Let $\mathcal{U}$ be an integral model of $U$ and for each non-archimedean place $v$ let $f_v: U(F_v) \to \R$ be the indicator function of the set $\mathcal{U}(\mathcal{O}_v)$.

 For a maximal face $\mathbf{A} = ((A_v, Z_v)_{v \in \Omega_{F}}^{\infty}) \in \mathcal{C}^{\an}_{\Omega_{F}}$ we define the product
\[f_{\mathbf{A}} := \prod_{v \in \Omega_{F}^{\infty}} f_{(A_v, Z_v)}: (X ; \mathbf{A})(\A_F) \to [0,1]. \]
This is a compactly supported continuous function. 

By construction we have the following equality for all $T > 0$
\begin{equation}
	\#\{P \in \mathcal{U}(\mathcal{O}_F) : H(P) \leq T \} = \sum_{\mathbf{A} \in \mathcal{C}^{\an, \max}_{\Omega_{F}}} \sum_{P \in U(F), H(P) \leq T} f_{\mathbf{A}}(x).
	\label{Writing the number of integral points as a sum over compactly supported functions}
\end{equation}

It thus suffices to study the sums $\sum_{P \in U(F), H(P) \leq T} f_{\mathbf{A}}(x).$ We will make a conjecture on the asymptotic behavior of such sums. 

If there is an analytic obstruction near $\mathbf{A}$ then the integral points counted by weighting by $f_{\mathbf{A}}$ are not Zariski dense by Theorem \ref{thm: Non-constant regular functions imply failure of Zariski density}. We will thus ignore such faces.
\subsection{The conjecture}
Let $X, D, U$ be as above and let $\mathbf{A} \in \mathcal{C}^{\an}_{\Omega_{F}^{\infty}}(D)$ be a face which has no analytic obstruction, i.e.~$\HH^0((X; \mathbf{A}), \mathcal{O}_{(X ; \mathbf{A})}) = F$. 

Assume that $\Pic X$ is finitely generated. Then $\Pic (X ; \mathbf{A})$ is finitely generated by Lemma \ref{lem:Computation of Pic(X ; A)}. The pseudo-effective cone $\Eff_{(X ; \mathbf{A})} \subset (\Pic (X ; \mathbf{A}))_{\R}$ satisfies $\Eff_{(X ; \mathbf{A})} \cap - \Eff_{(X ; \mathbf{A})} = 0$ by Proposition \ref{prop:pseuo-effective cone has non-zero cone constant}.

Let $K_{(X; \mathbf{A})} \in \Pic (X, \mathbf{A})$ be the image of the log canonical divisor $K_{(X,D)} := K_X + D$ along the map $\Pic X \to \Pic (X, \mathbf{A})$. Assume that $- K_{(X; \mathbf{A})}$ is \emph{big}, i.e.~$- K_{(X; \mathbf{A})} \in \Eff_{(X ; \mathbf{A})}^{\circ}$. This is true for example if $- K_{(X, D)} \in \Pic X$ is big. Let $L \in \Eff_{(X ; \mathbf{A})}^{\circ}$ be a big line bundle. Define the \emph{Fujita invariant} as 
\[
	a := a((X ; \mathbf{A}), L) := \inf\{t \in \R: K_{(X; \mathbf{A})} + t L \in \Eff_{(X ; \mathbf{A})}\}.
\]

Assume moreover that the pseudo-effective cone $\Eff_{(X ; \mathbf{A})}$ is a polyhedral cone. This is true if $\Eff_{X}$ is polyhedral, which holds if $(X, D)$ is log Fano by \cite[Cor.~1.3.5]{Birkar2010Minimal}. Then define $b := b((X ; \mathbf{A}), L)$ as the codimension of the minimal face of the cone $\Eff_{(X ; \mathbf{A})}$ which contains the \emph{adjoint divisor} $K_{(X; \mathbf{A})} + a((X ; \mathbf{A}), L) L$.
\begin{conjecture}\label{conj:Integral Manin conjecture}
	Assume that $(X, D)$ is log Fano and $\overline{F}[U]^{\times} = \overline{F}^{\times}$. Let $L \in \Pic (X ; \mathbf{A})$ be a big (i.e.~$L \in \Eff_{(X ; \mathbf{A})}^{\circ}$) line bundle, $(||\cdot||_v)_{v \in \Omega_F}$ an adelic metric on the $\G_m$-torsor corresponding to $L$ and $H: U(F) \to \R_{>0}^{\times}$ the induced height function.
	
	There exists a thin subset $Z \subset U(F)$ such that the following holds: if $f: (X; \mathbf{A})(\A_F) \to \R_{\geq 0}$ is a compactly supported continuous function which is non-zero when restricted to the closure $\overline{U(F) \setminus Z} \cap U(\A_F)_{\mathbf{A}} \subset (X; \mathbf{A})(\A_F)$ then there exists a positive constant $c((X ; \mathbf{A}), f,H)$ such that as $T \to \infty$ we have

	\begin{equation}
		\sum_{\substack{P \in U(F) \setminus Z \\ H(P) \leq T}} f(P) \sim c((X ; \mathbf{A}), f, H) T^{a((X ; \mathbf{A}), L)}(\log T)^{b((X ; \mathbf{A}), L) - 1}.
		\label{Integral Manin conjecture equation}
	\end{equation}
	\label{Integral Manin conjecture}
\end{conjecture}
\begin{remark} \hfill
	\begin{enumerate}
		\item The assumption that the pair $(X,D)$ is log Fano can likely be weakened, this is already true for the usual Manin conjecture. Since it is unclear what the precise geometric condition one needs for $X$ in the usual Manin conjecture we will make the log Fano assumption.
		
		\item The assumption that the restriction of $f$ to the closure $\overline{U(F) \setminus Z}$ is non-zero is necessary since the left-hand side of \eqref{Integral Manin conjecture equation} is $0$ otherwise.
		
		\item If $L = -a((X; \mathbf{A}), L) K_{(X ; \mathbf{A})}$ then $b((X ; \mathbf{A}), L) = \text{rk} \Pic (X ; \mathbf{A}) = \text{rk} \Pic U + \dim \mathbf{A} + 1$ by Corollary \ref{cor:Rank of Picard group Clemens complex}.
		
		\item We expect that there should also be a conjecture without assuming that $\overline{F}[U]^{\times} = \overline{F}^{\times}$. But in this case one should give a different definition of $\Pic (X ; \mathbf{A})$, see \cite[Def.~2.2.1]{Wilsch2022Integral} for a definition in the split case.
		
		\item There are examples of varieties $U$ for which the integral Manin conjecture holds and whose Picard group has torsion, e.g.~Example \ref{sec:Example with infinite Brauer group}. In this case $U$ has a finite \'etale cover which implies that for every compact set $C \subset (X; \mathbf{A})(\A_F)$ the intersection $U(F) \cap C$ is thin. To get a non-trivial conjecture we thus have to assume that the thin set $Z$ is independent of $f$. 
		
		\item In the case of rational points Lehman, Sengupta and Tanimoto have made a prediction \cite{Lehman2022Geometric} on the shape of $Z$ based on the idea that Manin's conjecture should be geometrically self-consistent and deep theorems in birational geometry. Unfortunately their arguments cannot be easily generalized to the integral case.
	\end{enumerate}
\end{remark}
\subsection{Leading constant}
In this section we will give a conjectural interpretation of the leading constant in \eqref{Integral Manin conjecture equation} assuming that the universal torsor is geometrically simple enough. We keep the notations of Conjecture \ref{conj:Integral Manin conjecture}.

Peyre's constant is easiest to define in the case that $L$ is \emph{adjoint rigid}. In the case that $L$ is not adjoint rigid one has to take into account the \emph{Iitaka fibration}\footnote{The terminology Iitaka fibration is now standard in the Manin conjecture literature, but Batyrev and Tschinkel instead used the term $L$-primitive fibration.} as in \cite{Batyrev1998Polarized}. It is unclear to us what the analogue of the Iitaka fibration is for integral points\footnote{Derenthal and Wilsch \cite{Derenthal2025Iitaka} have given a definition over $\Q$, the general case is still mysterious.} so we will assume that $L$ satisfies the analogue of adjoint rigidity.

\begin{definition}
	By Lemma \ref{lem:Computation of Pic(X ; A)} we can write $K_{(X; \mathbf{A})} + a((X ; \mathbf{A}), L)L$ as a sum 
	\[
	\sum_{i = 1}^n c_i[E_i] + \sum_{v \in \Omega_F^{\infty}} \sum_{j_v = 1}^m d_{j_v} [D_{j_v, v}].
	\]
	Where $E_i \subset X$ are irreducible divisors not contained in $D$, $D_{j_v, v} \subset D_{F_v}$ are irreducible divisors containing $Z_v$ and $c_i, d_{j_v} \in \Q_{> 0}$. We say that $L$ is \emph{adjoint rigid} if this decomposition is unique.
\end{definition}
Let $U_L := U \setminus (\bigcup_{i = 1}^n E_i)$ and $U_{Z_v, L} := U_{Z_v} \setminus \Big(\bigcup_{i = 1}^n E_i \cup (\bigcup_{j_v = 1}^{m_v} D_{j_v})\Big)$ for all places $v$. Consider the adelic variety $(X ; \mathbf{A})_L := (U_L, (U_{Z_v, L})_{v})$. We will need the following properties.
\begin{lemma} \label{lem:Basic properties adjoint rigid}
	If $L$ is adjoint rigid then we have
	\begin{enumerate}
		\item $\overline{F}[U_L]^{\times} = \overline{F}^{\times}$.
		\item $\HH^0((X ; \mathbf{A})_L, \mathcal{O}_{(X ; \mathbf{A})_L}) = F$.
		\item One has $a L = - K_{(X, \mathbf{A})} = -K_{(X, \mathbf{A})_L}$ in $\Pic (X, \mathbf{A})_L$.
	\end{enumerate}
\end{lemma}
\begin{proof}
	Let $\mathcal{E}$ denote the $\Gamma_F$-set of irreducible components of $(\bigcup_{i = 1}^n E_i)_{\overline{F}}$. We have an exact sequence of $\Gamma_F$-modules
	\begin{equation}\label{eq:Exact sequence of adjoint rigid model}
		\overline{F}[U_L]^{\times}/\overline{F}[U]^{\times} \to \Z[\mathcal{E}] \to \Pic U_{\overline{F}} \to \Pic U_{L, \overline{F}} \to 0.
	\end{equation}
	
	The first statement thus follows as long as $\Q[\mathcal{E}] \to (\Pic U_{\overline{F}})_{\Q}$ is injective. If the map is not injective then there exists two non-trivial elements $e_1, e_2 \in \Q_{\geq 0}[\mathcal{E}]$ which become equal in $\Pic U_{\overline{F}}$. Taking norms we get a non-trivial relation in $\Pic U$ between the $[E_i]$ which contradicts adjoint rigidity.
	
	For the second statement let $\mathbf{g} \in \HH^0((X ; \mathbf{A})_L, \mathcal{O}_{(X ; \mathbf{A})_L})$ be non-constant. The assumption that $\HH^0((X ; \mathbf{A}), \mathcal{O}_{(X ; \mathbf{A})}) = F$ implies that the function $\mathbf{g}$ can only have poles at one of the $E_i$ or $D_{j_v, v}$. But this implies that a non-trivial non-negative linear combination of the $[E_i]$ and $[D_{j_v, v}]$ is linearly equivalent to the effective divisor $\mathbf{g}^{-1}(0)$. This contradicts adjoint rigidity.
	
	The last statement is by construction.
\end{proof}

Assume from now on that $L$ is adjoint rigid.

Let $\mathbf{T}_{\NS L}$ be the N\'eron-Severi torus of $(X ; \mathbf{A})_L$. We have a map $\mathbf{T}_{\NS L} \to \mathbf{T}_{\NS}$ coming from the inclusion $(X ; \mathbf{A})_L \subset (X ; \mathbf{A})$. Let $\bm{\pi}: \mathbf{Y}_L \to (X ; \mathbf{A})_L$ be a universal and torsor and 
\[
\mathbf{H}: \coprod_{\mathbf{a} \in H^1(\A_F, \mathbf{T}_{\NS L})} \mathbf{Y}_{L, \mathbf{a}}(\A_F) \to \mathbf{T}_{\NS L}(\A_F)/\mathbf{T}_{\NS L}(\A_F)^1 \cong (\Pic (X ; \mathbf{A})_L)_{\R}^*
\]
a height function extending $H$, which exists by Lemma \ref{lem:Metrics can be extended}.

The Tamagawa measure $\tau_{\mathbf{Y}_{L, \alpha}, H^{-a}}$ is well-defined by Lemma \ref{lem:Basic properties adjoint rigid}(3).

\begin{conjecture} \label{conj:Leading constant}
	Assume that $L$ is adjoint rigid and that for all $v \in \Omega_F$ we have $\pi_1(Y_{L,\overline{F}_v}) = \Br Y_{L, \overline{F}_v} = 0$.
	
	The leading constant in Conjecture \ref{conj:Integral Manin conjecture} is
	\[
	c((X ; \mathbf{A}), f, H) := \frac{\theta((X ; \mathbf{A}), L)}{a (b - 1)!} \tau((X ; \mathbf{A}), f, H)
	\]
	Where $\tau((X ; \mathbf{A}), f, H)$ is defined as
	\[
	\begin{split}
		&\tau((X ; \mathbf{A}), f, H) := \lim_{B \subset \Br_1 (X; \mathbf{A})_L / \Br F} \# B \int_{(X; \mathbf{A})_L(\A_F)^B} f(x) d\tau_{(X; \mathbf{A})_L, H^{-a}}(x) =\\
		&\lim_{B \subset \Br_1 (X; \mathbf{A})_L / \Br F} \sum_{\beta \in B} \hat{\tau}_{(X; \mathbf{A})_L, H^{-a}}(f ; \beta)
		= \sum_{\beta \in \Br_1 (X; \mathbf{A})_L / \Br F} \hat{\tau}_{(X; \mathbf{A})_L, H^{-a}}(f ; \beta).
	\end{split}
	\]
	The limit is over finite subgroups $B$. In general we will only show that the sum converges conditionally in the sense that the limit of finite sums over $B$ converges.
\end{conjecture}
\begin{remark}	\hfil
	\begin{enumerate}
		\item The definitions of $\tau((X ; \mathbf{A}), f, H)$ are equal as $\sum_{\beta \in B} e^{2 i \pi \langle \beta,  x \rangle_{\text{BM}}}$ is equal to $\# B$ if $x \in (X; \mathbf{A})_L(\A_F)^B$ and $0$ otherwise by orthogonality of characters. 
		\item We will show in Theorem \ref{thm:Peyre's constant in terms of Brauer-Manin obstruction} that the limit defining $\tau((X ; \mathbf{A}), f, H)$ convergences.
		\item We assume that $\pi_1(Y_{L, \overline{F}_v}) = \Br Y_{L, \overline{F}_v} = 0$ since otherwise there can be a global obstructions to the distribution of rational points on $Y_{L, \overline{F}_v}$ whose role in the leading constant is not so clear, see \S\ref{sec:Example transcendental Brauer group} for an example which shows that the assumption is necessary. We remark that for rational points $\Br Y_{L, \overline{F}_v}$ can be non-trivial if $X$ has a non-trivial transcendental Brauer group and the role of transcendental elements in Peyre's constant is not so clear since in all known examples the transcendental Brauer group is trivial.
		\item The condition that $\pi_1(Y_{L,\overline{F}_v}) = 0$ for all $v$ is equivalent to the group  $\pi_1(U_{L,\overline{F}})$ being finite abelian. This follows from the fibration exact sequence $\pi_1(Y_{L, \overline{F}_v}) \to \pi_1(U_{Z_v, L, \overline{F}_v}) \to \pi_0(T_{\NS L, \overline{F}_v}) \to 1$, the fact that $\overline{F}^{\times}[Y_{L, \overline{F}_v}] = \overline{F}^{\times}$ and the map $\Pic U_{L, \overline{F}_v} \to \Pic Y_{L, \overline{F}_v}$ being zero. 
	\end{enumerate}
\end{remark}
The following lemma describes how the leading constant depends on the height.
\begin{lemma}\label{lem:Compatibility of heights}
	Let $g: (X ; \mathbf{A})(\A_F) \to \Hom(\Pic (X; \mathbf{A})_L, \R^{\times})$ be a continuous function with $\mathbf{H} := g \cdot \mathbf{H}'$ and $H' := \mathbf{H}'^{L}$. Then $c((X ; \mathbf{A}), f, H) = c((X ; \mathbf{A}), g^{-aL} \cdot f, H')$.
\end{lemma}
\begin{proof}
	The lemma follows immediately from  the $\tau_{(X ; \mathbf{A})_L, H^{-a}}(x) = g(x)^{- aL} \tau_{(X ; \mathbf{A})_L, H'^{-a}}(y)$ in Lemma \ref{lem:Tamagawa measure only depends on heigth function} and the definition of $\hat{\tau}_{(X; \mathbf{A})_L, H^{-a}}(f ; \beta)$.
\end{proof}
\begin{remark}\label{rem:Remarks on changing heights}
	If we replace $H$ by $C H$ for $C \in \R_{> 0}$ then this corresponds to replacing $\mathbf{H}$ by $g \mathbf{H}$ with $g$ constant and $g(L) = C$. In this case $c((X ; \mathbf{A}), f, C H) = C^{- a}c((X ; \mathbf{A}), f, H)$, so Conjecture \ref{conj:Leading constant} is compatible with multiplication of heights by constants.
\end{remark}

\subsubsection{Integrals over universal torsors}
We can relate the leading constant $c((X; \mathbf{A}), f, H)$ to certain integrals over universal torsors. The proof will also show that it well-defined. 

\begin{theorem}\label{thm:Peyre's constant in terms of Brauer-Manin obstruction}
	The limit over $B$ in $c((X ; \mathbf{A}), f, H)$ is absolutely convergent and equal to
	\[
	\begin{split}
		&\frac{1}{a (b - 1)!}\sum_{\alpha \in \HH^1(F, T_{\NS L})} \int_{\mathbf{Y}_{L, \alpha}(\A_F)/T_{\NS L}(F)} \mathbf{1}_{\Eff_{(X; \mathbf{A})_L}^*}(\log \mathbf{H}(y)) H(y)^{-1 - a} f(\bm{\pi}_{\alpha}(y))d\tau_{\mathbf{Y}_{L, \alpha}}^{\emph{gauge}} \\
		& = 
		\frac{1}{a (b - 1)!}\sum_{\alpha \in \HH^1(F, T_{\NS L})} \int_{\mathbf{Y}_{L, \alpha}(\A_F)/T_{\NS L}(F)} \mathbf{1}_{\Eff_{(X; \mathbf{A})_L}^*}(\log \mathbf{H}(y)) H(y)^{-1} f(\bm{\pi}_{\alpha}(y))d\tau_{\mathbf{Y}_{L, \alpha}, H^{-a}}.
	\end{split}
	\]
	Where the sums are absolutely convergent.
\end{theorem}
\begin{remark} \hfill
	\begin{enumerate}
	\item For  points and if $L = -K_{X}$ then this theorem is implicit in Salberger's work \cite[Assertion~5.25]{Salberger1998Tamagawa} and follows explicitly from work of Peyre \cite[Thm.~5.3.1]{Peyre1998Terme}. Note that the measure in Peyre's work is different to Salberger's measure, Peyre uses $d\tau_{\mathbf{Y}_{L, \alpha}}^{\text{gauge}}$ but Salberger uses $d\tau_{\mathbf{Y}_{L, \alpha}, H^{-a}}$.
	\item The limit over $B$ converging absolutely does not immediately imply that the sum over $\Br_1 (X; \mathbf{A})_L/ \Br F$ converges absolutely. It is unclear to the author if this sum converges absolutely in general.
	\end{enumerate}
\end{remark}
\begin{proof}
	The equality of the summands in the two sums over $\HH^1(F, T_{\NS L})$ follows immediately from Lemma \ref{lem:Comparison Gauge and Height measure}.
	
	Note that the $\limsup_B$ or $\liminf_B$ versions of $c((X ; \mathbf{A}), f, H)$ are monotone in $f$ and the integrals in the statement are also monotone. It thus suffices to prove it for a dense collection of $f$. If we prove the theorem for such $f$ then the $\limsup_B$ or $\liminf_B$ versions of $c((X ; \mathbf{A}), f, H)$ are equal for all $f$ so the limit in $B$ converges absolutely.
	
	By linearity of both sides we may assume that $f = \prod_v f_v$, where $f_v: U_{Z_v} \to \R_{\geq 0}$ is compactly supported for all $v$ and $f_v = \mathbf{1}_{\mathcal{U}_v(\mathcal{O}_v)}$ outside a finite set of places $S$ containing the archimedean ones. Here $\mathcal{U}$ is an $\mathcal{O}_{F, S}$-integral model of $U$.
	
	Using Lemma \ref{lem:Br_1 is isomorphic to H^1(Pic) adelically} and Lemma \ref{lem:Brauer-Manin pairing agrees with Poitou-Tate pairing} we have 
	\[
	\tau((X ; \mathbf{A}), f, H) := \sum_{\beta \in \HH^1(F, \Pic((X ; \mathbf{A})_L))} \int_{(X; \mathbf{A})_L(\A_F)} e^{2 i \pi \langle \beta,  \mathbf{Y}_L(x) \rangle_{\text{PT}}} f(x) d\tau_{(X; \mathbf{A})_L, H^{-a}}(x).
	\]
	
	By Proposition \ref{prop:Integral over universal torsor and height function}, Proposition \ref{prop:Relation between the two Tamagawa measures on Universal torsors} and the definition of $\theta((X ; \mathbf{A}), L)$ we can rewrite $\theta((X ; \mathbf{A}), L) \tau((X ; \mathbf{A}), f, H)$ as 
	\begin{equation}\label{eq:poisson summation Brauer side}
		\begin{split}
			&\frac{1}{\tau(\mathbf{T}_{\NS})} \sum_{\beta \in \HH^1(F, \Pic((X ; \mathbf{A})_L))} \int_{\HH^1(\A_F, \mathbf{T}_{\NS L})} e^{2 i \pi \langle \beta, \mathbf{Y}_L(x) \rangle_{\text{PT}}} \\
			&\int_{\mathbf{Y}_{L, \mathbf{a}}(\A_F)/T_{\NS L}(F)} f(\pi_{\mathbf{a}}(y)) 
			\mathbf{1}_{\Eff_{(X; \mathbf{A})_L}^*}(\log \mathbf{H}(y)) \mathbf{H}(y)^{- L}
			d\tau_{\mathbf{Y}_{L, \mathbf{a}}, H^{-a}}(y) 
			d\mu_{\mathbf{T}_{\NS L}}(\mathbf{a}).
		\end{split}
	\end{equation}

	The inner integral is by definition the Fourier transform of the function $F: \HH^1(\A_F, \mathbf{T}_{\NS L}) \to \C$ defined as
	\[
	F(\mathbf{a}) := \int_{\mathbf{Y}_{L, \mathbf{a}}(\A_F)/T_{\NS L}(F)} f(\pi_{\mathbf{a}}(y)) 
	\mathbf{1}_{\Eff_{(X; \mathbf{A})_L}^*}(\log \mathbf{H}(y)) \mathbf{H}(y)^{- L}
	d\tau_{\mathbf{Y}_{L, \mathbf{a}}, H^{-a}}(y).
	\]

	A formal application of the Poisson summation formula in Corollary \ref{cor:Poisson summation for cohomology} shows that \eqref{eq:poisson summation Brauer side} is equal to 
	\[
	\sum_{\alpha \in \HH^1(F, \mathbf{T}_{\NS L})} \int_{\mathbf{Y}_{L, \alpha}(\A_F)/T_{\NS L}(F)} f(\pi_{\alpha}(y))
	\mathbf{1}_{\Eff_{(X; \mathbf{A})_L}^*}(\log \mathbf{H}(y)) \mathbf{H}(y)^{- L}
	d\tau_{\mathbf{Y}_{L, \alpha}, H^{-a}}(y).
	\]
	
	It remain to check the assumptions of Theorem \ref{thm:Poisson summation}.
	
	By enlarging $S$ we may assume that the map 
	\[
	\prod_{v \in S} T_{\NS L}(F_v) \prod_{v \not \in S} T_{\NS L}(\mathcal{O}_v) \to \mathbf{T}_{\NS L, v}(\A_F)/T_{\NS L}(F)
	\]
	is surjective. The map $||\cdot||_v: Y_L(F_v) \to T_{\NS L}(F_v)/T_{\NS L}(\mathcal{O}_v)$ is $T_{\NS L}(F_v)$-invariant by definition. We can thus rewrite the integral $F(\mathbf{a})$ as 
	\begin{equation*}
		\begin{split}
			&\prod_{v \not \in S} \tau_{Y_L, ||\cdot||_v}(\{y_v \in Y_{L, a_v}(F_v): ||y_v||_v = 1, \pi_{a_v}(y_v) \in \mathcal{U}(\mathcal{O}_v)\})\\
			&\int_{(\prod_{v \in S} Y_{L, a_v}(F_v))/T_{\NS L}(\mathcal{O}_{F, S})} \prod_{v \in S} f_v(\pi_{a_v}(y_v)) 
			\mathbf{1}_{\Eff_{(X; \mathbf{A})_L}^*}(\log \mathbf{H}(y)) \mathbf{H}(y)^{- L}
			d\tau_{\mathbf{Y}_{L, \mathbf{a}}, H^{-a}}(y).
		\end{split}
	\end{equation*}
	The second factor can be bounded independently of the $a_v$ since $S$ and the $\HH^1(F_v, T_{\NS L, v})$ are finite.
	
	The set $\{y_v \in Y_{L, a_v}(F_v): ||y_v||_v = 1, \pi_{a_v}(y_v) \in \mathcal{U}(\mathcal{O}_v)\}$ for $v \not \in S$ can be split up as a union
	\[
	\{y_v: ||y_v||_v = 1, \pi_{a_v}(y_v) \in \mathcal{U}_L(\mathcal{O}_v)\} \cup \{y_v : ||y_v||_v = 1, \pi_{a_v}(y_v) \in \mathcal{U}(\mathcal{O}_v) \setminus \mathcal{U}_L(\mathcal{O}_v) \}.
	\]
	The first set is empty unless $a_v \in \HH^1(\mathcal{O}_v, T_{\NS L})$, otherwise its measure is $1 + O(q_v^{-3/2})$ by Lemma \ref{lem:local measure universal torsor}. To bound the measure of the second set we use Lemma \ref{lem:measure on universal torsor is product of measure on base and measure on torus}, Lemma \ref{lem:Computation of local volume of groups of multiplicative type} and \cite[3.3.8]{Batyrev1998Polarized} to deduce that there exists a $0 < \delta < 1/2$ such that $\tau_{Y_L, ||\cdot||_v}(\{y_v : ||y_v||_v = 1, \pi_{a_v}(y_v) \in \mathcal{U}(\mathcal{O}_v) \setminus \mathcal{U}_L(\mathcal{O}_v)\})$ is
	\[
	\tau_{T_{\NS L}, v}(T_{\NS L}(\mathcal{O}_v)) \tau_{U_L, ||\cdot||_v}(\mathcal{U}(\mathcal{O}_v) \setminus \mathcal{U}_L(\mathcal{O}_v)) \ll q_v^{-1 - \delta}.
	\]
	Combining everything we deduce that 
	$F(\mathbf{a}) \ll O(1)^{\omega(\text{cond}_S(a))} \text{cond}_S(a)^{- 1 - \delta}$.
	
	A similar but simpler argument using Lemma \ref{lem:estimates for local Brauer transform} shows that $\hat{F}(\mathbf{b}) \ll \text{cond}_S(\mathbf{b})^{- 1 - \delta}$ for $\mathbf{b} \in \HH^1(\A_F, \bm{\Pic}((X ; \mathbf{A})_L)).$ The conditions for the Poisson summation formula in Theorem \ref{thm:Poisson summation} thus hold by Lemma \ref{lem:Bounds for conductor sums}.
\end{proof}

\subsection{Examples}
If $\Br_1(X ; \mathbf{A})/ \Br F \subset \Br_1 U /\Br F$ is trivial then Conjecture \ref{conj:Integral Manin conjecture} agrees with the prediction in \cite[\S2.5]{Wilsch2022Integral}. As explained in loc. cit. this shows that our prediction agrees with the previously known results in \cite{Chambert2012Integral, Takloo2013Integral, Chow2019Distribution, Wilsch2022Integral,Derenthal2022Integral, Wilsch2023Integral}.
\subsubsection{Wilsch's counterexample}
Let us explain what happens in Wilsch's counterexample \cite{Wilsch2022Integral} to Chambert-Loir and Tschinkel's work \cite{Chambert-Loir2010Integral}. We will use the notation of Wilsch, so we have a toric variety $X$ and a boundary divisor $D \subset X$ with three irreducible components $M, E_1, E_2$. The Clemens complex has $4$ non-trivial elements $M, E_1, E_2$ and $A := E_1 \cap E_2$. 

Wilsch shows in \cite[\S 3.3]{Wilsch2022Integral} that the pair $(X, D)$ has an analytic obstruction near the face $A$ and thus also near the faces $E_1$ and $E_2$. The face $M$ is thus the only non-trivial face with no analytic obstruction. We have $\Pic(X ; M) \cong \Z^3$ by \cite[Prop.~3.3.1]{Wilsch2022Integral} so the expected order of growth is $B (\log B)^2$ which agrees with \cite[Thm.~1.0.1]{Wilsch2022Integral}. 

We have that $(X ; M)$ is $(\Proj^1)^3$ minus the two lines $\{a_1 = b_1 = 0\}$ and $\{a_1 = c_1 = 0\}$ which intersect in a point. Wilsch only counts integral points for which one always has $a_1/a_0 \geq 1$ so is implicitly weighting by a compactly supported function $(X ; M)(\R) \to \R$. He shows in Lemma 3.3.4 that the leading constant can be interpreted as a Tamagawa measure. By Theorem \ref{thm:Peyre's constant in terms of Brauer-Manin obstruction} this agrees with Conjecture \ref{conj:Leading constant}. Actually, the leading constant can be shown to agree with Conjecture \ref{conj:Leading constant} directly from \cite[Lem.~3.2.1]{Wilsch2022Integral} without computing Tamagawa measures on $U$. We leave this to the interested reader.
\subsubsection{Counterexample when there are global invertible sections}\label{sec:Counterexample when there are global invertible sections}
Let $X = \Proj^2$ and $D$ a union of two lines. The pair $(X, D)$ is log Fano but $U \cong \G_m \times \A^1$. If $k = \Q$ then the integral points $\G_m(\Z) \times \A^1(\Z) = \{-1,1\} \times \Z$ are not even Zariski-dense. One can check that this is not due to an analytic obstruction.

The integral points will be Zariski-dense for number fields $F$ with at least two archimedean places. In this case one can count integral points of bounded anticanonical height on $U$. But to do this one has to count integral points on the torus $\G_m$. The number of integral points on $\G_m$ of bounded height has order of magnitude a power of $\log T$ by Dirichlet's unit theorem. Such counting problems are not very similar to those appearing in Manin's conjecture, this is why we make the assumption that $\overline{F}^{\times}[U] = \overline{F}^{\times}$ in Conjecture \ref{conj:Integral Manin conjecture}.
\subsubsection{Effect of different faces}
In the previous examples there was a only a single relevant face. In the following example one can see the effect of counting near different faces.

Let $X$ be the projective plane $\Proj^2_{x,y,z}$ blown up at the point $\{x = z = 0\}$ and the point $\{y = z = 0\}$. Let $E_x, E_y \subset X$ be the respective exceptional divisors and $D$ the proper transform of $\{z = 0\}$. Consider the pair $(X, E_x \cup E_y \cup D)$, it is not log Fano but the anticanonical divisor is big and an analogue of Manin's conjecture still holds. In this case $U = \A^2_{x,y}$ so for the standard integral model $\mathcal{U} = \A^2_{\Z}$ integral points correspond to pairs $x,y \in \Z$.
	
The log anticanonical height is given by $\max(|x y|, 1)$. The Clemens complex has five non-trivial faces, let us discuss the integral points of bounded height close to each of them.
\begin{enumerate}
	\item The maximal face $Z = E_x \cup D$ has dimension $1$. In this case $U_{Z} = X \setminus E_{y}$. Weighting by a compactly supported function $f:U_{Z} \to \R_{\geq 0}$ thus causes one to only count those $x,y$ for which $x = O(y)$. The order of magnitude in this case is $T \log T$. The other maximal face $E_y \cup D$ is analogous.
	\item Another face is $D$ which has dimension $0$. In this case $U_D = X \setminus (E_x \cup E_y) = \Proj^2 \setminus (\{x = z = 0\} \cup \{y = z = 0\})$. Weighting by a compactly supported function $f: U_Z \to \R_{\geq 0}$ thus cause one to only count pairs $x,y$ such that $x = O(y)$ and $y = O(x)$. The order of magnitude is $T$.
	\item For the face $E_x$ we have $U_{E_x} = X \setminus (D \cup E_y) \cong \Proj^1 \times \A^1$. Weighting by a compactly supported function $f: U_Z \to \R_{\geq 0}$ corresponds to only counting those $(x,y)$ with $y = O(1)$. The collection of such points is not Zariski dense. This is due to an analytic obstruction.
\end{enumerate}
	
Note that counting all integral points by only counting those close to faces corresponds to separating the count into two cases, in the first one $x = O(y)$ and in the second one $y = O(x)$. This is the Dirichlet hyperbola method and it is the standard method to count pairs $x,y \in \Z$ such that $|xy| \leq T$.
\subsubsection{Infinite Brauer group}\label{sec:Example with infinite Brauer group}
Let $U := \Spec \Q[x,y,z]/(xy = z^2) \setminus \{x = y = z = 0\}$. This is a smooth toric variety with no invertible global sections, $\Pic U \cong \Z /2 \Z$ and $\Br U \cong \HH^1(\Q, \Z/ 2 \Z) = \Q^{\times}/\Q^{\times 2}$ is infinite. One can count integral points on this variety. Indeed, if $\mathcal{U} := \Spec \Z[x,y,z]/(xy = z^2) \setminus \{x = y = z = 0\}$ is the obvious integral model then integral points on $\mathcal{U}$ correspond to triples $(x,y,z) \in \Z^3$ with $\gcd(x,y,z) = 1$. This implies that integral points are of the form\footnote{The map $\A^2_{u, v} \to U:(u,v) \to (u^2, uv, v^2)$ is a universal torsor of $U$.} $(\pm u^2, \pm v^2, uv)$ for coprime $u, v \in \Z$, well-defined up to multiplication by $-1$. The log-anticanonical height is $\max(u^2, v^2)$.

We see that both $x$ and $y$ have to be $\pm$ a square in $\Z_p$ for all primes $p$. This gives an infinite number of local conditions so $\overline{\mathcal{U}(\Z)} \subset \prod_v \mathcal{U}(\Z_p)$ has empty interior. If $\tau$ is the usual Tamagawa measure on $\prod_v \mathcal{U}(\Z_p)$ then $\tau(\overline{\mathcal{U}(\Z)}) = 0$.
\subsubsection{Transcendental Brauer group}\label{sec:Example transcendental Brauer group}
Let $U := \{wxyz = t^2\} \subset \Proj^4_{w,x,y,z,t} \setminus \{wx = wy = wz = xy = xz = yz = 0\}$. This is a toric variety\footnote{$N = \{(a,b,c) \in (\frac{1}{2} \Z)^3: a + b + c \in \Z \}$ and the rays are $(1,0,0), (0,1,0), (0,0,1), (-1,-1,-1)$.} with $\HH^0(U, \mathcal{O}_U) = 0$ and $\Pic U \cong \Z \oplus (\Z/ 2 \Z)^2$. Consider the quaternion algebra $\beta := (\frac{x}{w}, \frac{y}{w})$. One can show that $\beta \in \Br U$ and that $\beta_{\bar{\Q}} \neq 0 \in \Br U_{\bar{\Q}}$. By Theorems \ref{thm:Main theorem log anticanonical height} and \ref{thm:Strong approximation in introduction} applied to $\mathbf{A} = \emptyset$ one sees that $\beta$ does not provide an additional Brauer-Manin obstruction on top of the algebraic Brauer-Manin obstruction and that it plays no role in the leading constant.
\subsubsection{Universal torsor insufficient}\label{sec:universal torsor insufficient}
Let $A_4 \subset S_4$ be the alternating group acting on $\A^4$ by permuting the coordinates. Let $Z \subset \A^4$ be the closed subvariety on which the action is not free, explicitly this is given by the locus where three coordinates are equal or where two pairs of coordinates are equal. One sees that $Z$ has codimension $2$. Let $U := (\A^4 \setminus Z)/A_4$.

The Picard group $\Pic U_{F^{\sep}}$ is canonically isomorphic to $\Hom(A_4, \G_m(F^{\sep})) = \mu_3(F^{\sep})$ by the Hochschild-Serre spectral sequence 
\[
\HH^p(A_4, \HH^q(\A^4_{F^{\sep}} \setminus Z, \G_m)) \implies \HH^{p + q}(U_{F^{\sep}}, \G_m).
\] Using functoriality of this sequence one can show that $Y := (\A^4 \setminus Z)/(\Z/2\Z)^2 \to U$ is a universal torsor, where $(\Z/2\Z)^2 = [A_4, A_4]$ is the commutator.

The Hochschild-Serre spectral sequence similarly shows that $\Pic Y_{\F^{\sep}} = (\Z/2 \Z)^2$ and $\Br Y_{F^{\sep}} = \HH^2((\Z/2 \Z)^2, F^{\sep, \times}) = \Z/2\Z$. This shows that the assumption on the universal torsor in Conjecture \ref{conj:Leading constant} is not automatically satisfied. Moreover, we see that in this case the natural place to count integral points is not the universal torsor $Y$, it is on the variety $\A^4 \setminus Z$ and its twists.

Note that one can embed $Y$ as a dense open in the $U$ of \S\ref{sec:Example transcendental Brauer group}. 
\subsubsection{The circle method}
The circle method is a very powerful tool to count integral and rational points on varieties of large relative dimension to the degree, as worked out in Birch's seminal work \cite{Birch61Forms}. It is well-known \cite[Prop.~4.2.1]{Peyre1995Hauteurs} that the circle method is compatible with Manin's conjecture for rational points, including Peyre's prediction for the leading constant. Similar arguments show that the circle method is also compatible with Conjectures \ref{conj:Integral Manin conjecture} and \ref{conj:Leading constant}. Let us work this out in detail in the work of Skinner \cite{Skinner1997Forms} on integral points on affine complete intersections.

Let $g_1, \dots, g_m \in \Z[X_1, \cdots, X_s]$ be a collection of polynomials of degree $d$ with $s > m + 3$. Let $G_i \in \Z[X_0, \cdots, X_s]$ the homogenization of $g_i$ and let $h_i \in \Z[X_1, \cdots, X_s]$ be the degree $d$ part of $g_i$. We want to understand integral points of bounded height on $U := \{g_1 = \cdots = g_m = 0\} \subset \A_{\Q}^s$. We will consider the compactification $U \subset X := \{G_1 = \cdots = G_m = 0\} \subset \Proj_{\Q}^s$. Let $D := X \setminus U = \{h_1 = \cdots = h_m = 0\} \subset \{X_0 = 0\} = \Proj_{\Q}^{s - 1}$. We will assume that $X$ and $D$ are smooth complete intersections, $D$ is in particular geometrically irreducible.

The analytic Clemens complex has a single non-trivial face $D$. Let $\mathcal{U}$ be the obvious integral model of $U$. Let $f$ to be the indicator function of $\prod_p \mathcal{U}(\Z_p) \times X(\R)$. 

One has $\Pic X = \Z$ generated by $[D]$ by the Lefschetz hyperplane theorem so $\Pic(X ; D) = \Z$. This implies that $\Br_1 (X ;D) = \Br_1 U = \Br \Q$. The canonical divisor is $K_X = (md - s - 1)[D]$ and the log anticanonical divisor is $-K_{(X, D)} = (s - md)[D]$, which is ample. The adelic universal torsor $\mathbf{Y}$ of $(X ; D)$ is given by $U$ at the primes and by $\{g_1 = \cdots = g_n = 0\} \subset \A^{s + 1}_{\Q} \setminus \{0\}$ at $\infty$. Let $\mathcal{B} \subset \R^{s}$ be a compact neighbourhood of the origin and define an adelic metric on the universal torsor by setting $||\cdot||_p = 1$ for all primes and $||(x_0, \cdots, x_s)||_{\infty} = \inf(t \in \R: (x_1, \cdots, x_{s}) \in t\mathcal{B})^{-1}$ for $(x_0, \cdots, x_s) \in Y_{\infty}(\R) \subset \R^{s + 1}$. Let $H = ||(x_0, \cdots, x_s)||_{\infty}^{-1}$ be the corresponding height.

Skinner \cite[Main Thm.]{Skinner1997Forms} shows\footnote{Skinner actually only works with $\mathcal{B}$ which are box shaped, the case of arbitrary $\mathcal{B}$ follows from \cite[Lem.~4.9]{Loughran2015Rational}.} via the circle method that if $s \geq m(m + 1)(d - 1)2^{d-1}$ then there exists a $c > 0$ such that
\begin{equation}\label{eq:Skinner's main result}
	\# \{ (x_1, \cdots, x_s) \in \mathcal{U}(\Z): H(x_1, \cdots, x_s) \leq T \} \sim \sigma_{\infty} \prod_p \sigma_p T^{s - md}.
\end{equation}
Where $\sigma_p$ is the local factor at $p$ in the singular series \cite[p. 26]{Skinner1997Forms} and $\sigma_{\infty}$ is the singular integral \cite[p.~23]{Skinner1997Forms}.
This agrees with Conjecture \ref{conj:Integral Manin conjecture}.

Let $d u \in \omega_U$ be the differential form on $U$ which is the image of the differential form $g_1^{-1} \cdots g_m^{-1} d X_1 \cdots d X_s$ on $\A^s_{\Q}$ under the adjunction formula for the immersion $U \subset \A_{\Q}^s$. Similarly let $d y$ be the differential form on $E := Y \cap \{X_0 = 0\}$ which is the image of $G_1^{-1} \cdots G_m^{-1} X_0^{-1} dX_0 d X_1 \cdots d X_s$ under the adjunction formula for the immersion $Y \cap \{X_0 = 0\} \subset \A^{s + 1}$. The form $dy$ is the $\omega_{\infty}$ from $\omega = du$ as in Construction \ref{cons:measure gauge form}.

The equality $\sigma_p = \int_{\mathcal{U}(\Z_p)} |du|_p$ is standard, one first uses Fourier inversion to see that $\sigma_p = \lim_{k \to \infty} p^{-k(\dim U)} \# \mathcal{U}(\Z/p^k \Z)$ as in \cite[p. 251]{Birch61Forms} and this is equal to the integral by \cite[Lem.~1.8.1]{Borovoi1995Hardy}. The singular integral is more interesting. The complete intersection is non-singular so this integral is by \cite[Lem.~6.3]{Birch61Forms} equal to
\[
\sigma_{\infty} := \int_{\R^m} \int_{\mathcal{B}} e^{2 i \pi \sum_{i = 1}^m t_i f_i(\mathbf{x})} d\mathbf{x} dt_1 \cdots d t_m = \int_{E(\R)} \mathbf{1}_{\mathcal{B}}(y) |dy|_{\infty}.
\]

The leading constant predicted in Conjecture $\ref{conj:Leading constant}$ is $\frac{1}{s - md} \prod_{p} \sigma_p$ times \[
\int_{E(\R)} \mathbf{1}_{\{H(y) > 1\}}(y) H(y)^{- s + md - 1}|dy|_{\infty}.
\]
To compare this with the singular integral we note that for $t \in \R_{>0}$ we have $| dty|_{\infty} = t^{s - md} |dy|_{\infty}$ so $\text{Vol}(E(\R) \cap t \mathcal{B}) := \int_{E(\R)} \mathbf{1}_{t\mathcal{B}}(y) |dy|_{\infty} = t^{s - md} \sigma_{\infty}$. We then have
\begin{align*}
	&\int_{E(\R)} \mathbf{1}_{\{H(y) > 1\}}(y) H(y)^{- s + md - 1}|dy|_{\infty} = \int_{1}^{\infty} t^{- s + md - 1} \frac{d}{dt} \text{Vol}(E(\R) \cap t \mathcal{B}) dt = \\
	& (s - md) \sigma_{\infty} \int_{1}^{\infty} t^{-2} dt = (s - md) \sigma_{\infty}.
\end{align*}

We conclude that the leading constant $\sigma_{\infty} \prod_p \sigma_p$ in \eqref{eq:Skinner's main result} agrees with Conjecture \ref{conj:Leading constant}.

\begin{remark}
	Skinner in \cite{Skinner1997Forms} proves a counting result which is valid over all number fields $F$ and not only over $\Q$. But the height he uses is not a height in our sense over $((U_{F_v})_{v \in \Omega_{F}^{\text{fin}}}, (X_{F_v})_{v \in \Omega_{F}^{\infty}})$.
	
	We can still interpret his result in our framework. Let $g_1, \dots, g_m \in \mathcal{O}_F[X_1, \cdots, X_s]$ be polynomials, $G_i \in \mathcal{O}_F[X_0, X_1, \cdots X_s]$ the homogenization of $g_i$.
	
	Now consider the variable $X_0$ to only take values in $\Q$ while the other variables take values in $F$. We can study this situation by letting $U$ be the Weil restriction of $\{g_1 = \cdots = g_m = 0\} \subset \A^s_F$ and $X$ as $\{G_1 = \cdots = G_m = 0\} \subset \Proj_{\Q}(\Q \times F^s)$. One can then show that this agrees with the conjectures in the exact same way as before.
\end{remark}
\subsection{Reductions}
It is often convenient to first prove Manin's conjecture for simpler $f$ and particular heights $H$ and then deduce the general conjecture. Since these reductions are rather formal we collect them here.
\begin{proposition}\label{prop:Reduction of general Manin conjecture}
	Assume that Conjecture \ref{conj:Integral Manin conjecture} holds with the constant conjectured in Conjecture \ref{conj:Leading constant} for one particular height $H$ and for all $f = \prod_v f_v$ where for each $v \in \Omega_F$ the function $f_v: U_{Z_v}(F_v) \to \R_{\geq 0}$ is compactly supported and is locally constant, resp. smooth, if $v$ is non-archimedean, resp. archimedean.
	
	Conjecture $\ref{conj:Integral Manin conjecture}$ with the constant conjectured in Conjecture \ref{conj:Leading constant} then holds for all possible heights on $L$ and all compactly supported functions $f$.
\end{proposition}
\begin{proof}
	If Conjecture \ref{conj:Leading constant} holds for $f$ of the form $\prod_v f_v$ then it also holds for linear combinations of such functions. The vector space of such linear combinations is dense in the vector space of compactly supported continuous functions.
	
	Let $f: (X ; \mathbf{A})(\A_F) \to \R_{\geq 0}$ be any compactly supported continuous function. We may then uniformly approximate $f$ both from below and from above by functions for which Conjecture \ref{conj:Leading constant} holds. Taking limits of both these approximations gives the desired statement.
	
	That Conjecture \ref{conj:Leading constant} then holds for all possible heights is essentially due to Peyre \cite[Prop.~3.3(b)]{Peyre1995Hauteurs}. Let us repeat the argument in our context for completeness. Any other height $H'$ is equal to $gH$ for $g: (X ; \mathbf{A})(\A_F) \to \R_{> 0}$ a continuous function. For every $\varepsilon > 0$ we can write $f$ as a sum of continuous compactly supported functions $f_i: (X ; \mathbf{A})(\A_F) \to \R_{\geq 0}$ such that there exists a $c_i > 0$ and such that $c_i - \varepsilon < g(x) < c_i + \varepsilon$ for all $x$ contained in the support of $f_i$ by a partition of unity.
	
	The following inequality follows from Conjecture \ref{conj:Integral Manin conjecture} as $T \to \infty$

	\[
	\sum_{\substack{P \in U(F) \setminus Z \\ (c_i - \varepsilon) T \leq H(P) \leq (c_i + \varepsilon) T}} f_i(P) \ll \varepsilon (c_i T)^{a - 1} (\log c_i T)^b.
	\]
	 Letting $\varepsilon \to 0$ we are reduced to the case when $g$ is constant. This case follows from Lemma \ref{lem:Compatibility of heights}, cf.~Remark \ref{rem:Remarks on changing heights}.
\end{proof}
\section{Counting}
\subsection{Set-up}
Let $F$ be a number field, $X$ a smooth proper toric variety over $F$ with corresponding torus $T$ and $\Gamma_F$-invariant fan $\Sigma$ on the lattice $N := \Hom(\G_m, T)$. Let $Z \subset X$ be a toric divisor corresponding to a set of rays $Z_{\min} \subset \Sigma_{\min}$. Let $U := X \setminus Z$. We will assume that $\overline{F}[U]^{\times} = \overline{F}^{\times}$. The complement $D := \Sigma_{\min} \setminus Z_{\min}$ is the $\Gamma_F$-set of irreducible toric divisors of $U$.

Let $\mathbf{A} = (A_v, Z_v)_{v \in \Omega_F^{\infty}} \in \mathcal{C}^{\text{an}}_{\Omega^{\infty}}(D)$ be a face of the analytic Clemens complex. Note that $A_v \subset D_{\min}$ is the set of divisors containing $Z_v$ for all $v \in \Omega_F^{\infty}$, so $Z_v$ is determined by $A_v$. We will write $U_{A_v} := U_{Z_v}$. We assume that $\HH^0((X ; \mathbf{A}), \mathcal{O}_{(X ; \mathbf{A})}) = F$. This implies by Proposition \ref{prop:pseuo-effective cone has non-zero cone constant} that $\Eff_{(X ; \mathbf{A})} \subset (\Pic (X ; \mathbf{A}))_{\R}$ does not contain a line. 

Note that by \eqref{eq:anticanonical divisor of a toric variety} we have
\begin{equation}\label{eq:log-canonical divisor}
	-K_{(X; \mathbf{A})} = \sum_{D_{\rho} \in D} [D_{\rho}].
\end{equation}
Moreover, by Lemma \ref{lem:log anticanonical divisor is big} we have $-K_{(X; \mathbf{A})} \in \Eff_{(X ; \mathbf{A})}^{\circ}$.
\subsubsection{Universal torsor}
Let $\mathbf{T}_{\NS} = (T_{\NS}, (T_{\NS A_v})_v)$ be the N\'eron-Severi torus of $(X ; \mathbf{A})$.

Let $\pi: Y \to U$, resp. $\pi_v: Y_{A_v} \to U_{A_v}$ for each place $v$, be the universal torsor constructed in Construction \ref{con:Construction of universal torsor}.

There is a natural toric inclusion $Y_{F_v} \subset Y_{A_v}$ so $\mathbf{Y} := (Y, (Y_{A_v})_v)$ is an adelic variety and $\bm{\pi}: \mathbf{Y} \to (X, \mathbf{A})$ is a universal torsor.

Let $\Theta = \G_m^D$ and $\Theta_{A_v} = \G_m^{D \cup A_v}$. Let $\Theta_{\mathbf{A}} := (\Theta, (\Theta_{A_v})_v)$, considered as an adelic torus. It is by construction equal to $\bm{\pi}^{-1}(T) \subset \mathbf{Y}$.

Fix a smooth adelic metric $(||\cdot||_v)_{v \in \Omega_F}$ on the torsor $\bm{\pi}$. By enlarging $S$ we may assume that $(||\cdot||_v)_{v \in \Omega_F}$ has good reduction at all places $v \not \in S$.
Recall that $\mathbf{T}_{\NS}$ acts via the inverse of the natural inclusion $\mathbf{T}_{\NS} \to \Theta_{\mathbf{A}}$. So if $t \in \mathbf{T}_{\NS}(\A_F)$ and $x \in \Theta_{\mathbf{A}}(\A_F)$ then $||t \cdot x|| = t^{-1} ||x|| \in \mathbf{T}_{\NS}(\A_F)/\mathbf{T}_{\NS}(\A_{\mathcal{O}})$. 

Note that $\hat{\Theta}_{\mathbf{A}} = \Z[D / \Gamma_F] \times \prod_{v \in \Omega_F^{\infty}} \Z[A_v / \Gamma_v]$. Given $c,d \in \C$ we denote by $(\mathbf{c}, \mathbf{d}) \in \hat{\Theta}_{\mathbf{A}, \C}$ the element $\sum_{\rho \in D / \Gamma_F} c \cdot [\rho] + \sum_{v \in \Omega_F^{\infty}, \rho_v \in A_v/ \Gamma_v} d \cdot [\rho_v]$ and $\mathbf{c} := (\mathbf{c}, \mathbf{c})$. We will use the same notation for the images of these elements in $\hat{\mathbf{T}}_{\NS, \C}$.

We use the analogous notation for elements $\hat{\Theta}_{A_v} =\Z[(D \cup A_v) / \Gamma_v]$ for all places $v$. 

Let $\mathbf{H}: \coprod_{\mathbf{a} \in H^1(\A_F, \mathbf{T}_{\NS})} \mathbf{Y}(\A_F) \to \Hom(\Pic (X, \mathbf{A}), \R_{>0}^{\times})$ be the corresponding height function. If $t \in \mathbf{T}_{\NS}(\A_F)$ and $x \in \Theta_{\mathbf{A}}(\A_F)$ then $\mathbf{H}(t \cdot x) = |t|_{\mathbf{T}_{\NS}} \mathbf{H}(x)$.

The toric divisors generate the effective cone so by Corollary \ref{Height lies in shift in effective cone} we may assume, after multiplying the metric and thus the height by a constant, that $\log \mathbf{H}(x) \in \Eff_{(X; \mathbf{A})}^*$ for all $x \in T(F)$ by multiplying the metric and thus the height by a constant.
\subsubsection{The function}\label{sec:The function}
Let $S$ be a finite set of places containing all the archimedean places, all ramified places of $K/\Q$ and all the places of bad reduction of $T, \Theta$ and $T_{\NS}$. Let $\mathcal{Y} \to \mathcal{U}$ an $\mathcal{O}_{K, S}$-integral model of the $T_{\NS}$ torsor $Y \to U$ such that $\mathcal{U}$, resp. $\mathcal{Y}$, is equipped with an integral model of the action of $T$, resp. $\Theta$. This exists by spreading out, after potentially enlarging $S$.

For each place $v$ of $F$ let $f_v: U_{A_v}(F_v) \to \R_{\geq 0}$ be a compactly supported function. Assume that $f_v$ is smooth if $v$ is archimedean and locally constant if $v$ is non-archimedean. We will assume that if $v \not \in S$ then $f_v$ is the indicator function of $\mathcal{U}(\mathcal{O}_{v})$. The product $f := \prod_v f_v: (X; \mathbf{A})(\A_F) \to \R_{\geq 0}$ is then a well-defined compactly supported function. 
\subsubsection{Smoothing the height}
To ensure that certain integrals converge absolutely it will be convenient to smooth the height function. 
\begin{construction}\label{cons:Mollifiers}
	For $v \in \Omega_F^{\text{fin}}$ let $\varphi_v: T_{\NS}(F_v) \to \R$ be the indicator function of $T_{\NS}(\mathcal{O}_v)$ and for each place $v \in \Omega_F^{\infty}$ fix a non-zero smooth compactly supported function $\varphi_v: T_{\NS A_v}(F_v) \to \R$. The product function $\varphi := \prod_v \varphi_v: \mathbf{T}_{\NS}(\A_F) \to \R$ is then compactly supported.
	
	By shifting the $\varphi_v$ we may assume that if $y \in \mathbf{T}_{\NS}(\A_F)$ is such that $\varphi(y) \neq 0$ then $\log |y|_{\mathbf{T}_{\NS}} \in \Eff_{(X ; \mathbf{A})}$. 
\end{construction}

We will consider the following smoothed form of the height function $\mathbf{H}$. For $\bm{s} \in (\Pic (X ; \mathbf{A}))_{\C}$ let $\mathcal{H}_{\varphi}(\cdot ; \bm{s}): \Theta_{\mathbf{A}}(\A_F) \to \C$ be defined by the integral
\begin{equation}
	\mathcal{H}_{\varphi}(x; \bm{s}) := \int_{\mathbf{T}_{\NS}(\A_F)} \mathbf{1}_{\Eff_{(X; \mathbf{A})}^*}(\log |y|_{\mathbf{T}_{\NS}}) \varphi(||x|| y) |y|_{\mathbf{T}_{\NS}}^{- \bm{s}} d \tau_{\NS}(y).
	\label{Definition of smoothed height}
\end{equation}
This integral converges absolutely because $\varphi$ is compactly supported. Note also that the integral is independent of the chosen representative of $||x|| \in \mathbf{T}_{\NS}(\A_F)$ because the function $\mathbf{1}_{\Eff_{(X; \mathbf{A})}^*}(\log |y|_{\mathbf{T}_{\NS}}) |y|_{\mathbf{T}_{\NS}}^{- \bm{s}}$ is invariant under multiplication by $\mathbf{T}(\A_F)^1$.

Let $(\mathcal{M} \varphi)(\bm{s}):= \int_{\mathbf{T}_{\NS}(\A_F)} \varphi(z) |z|_{\mathbf{T}_{\NS}}^{\bm{s}} d \tau_{\NS}(z)$ be the Mellin transform of $\varphi$ where $\bm{s} \in (\Pic (X ; \mathbf{A}))_{\C}$.
\begin{remark}\label{rem:Can choose Mellin of phi to be non-zero}
	For any fixed choice of $\bm{s} \in (\Pic (X ; \mathbf{A}))_{\C}$ we can always choose a $\varphi$ such that $(\mathcal{M} \varphi)(\bm{s}) \neq 0$. Since everything that follows will follow for any choice of $\varphi$ satisfying the above assumptions we may always assume that $(\mathcal{M} \varphi)(\bm{s}) \neq 0$.
\end{remark}
\begin{lemma} \hfill 
	\begin{enumerate}
		\item The function $\mathcal{H}_{\varphi}(x; \bm{s})$ is continuous in the variable $x$ and holomorphic in $\bm{s}$.
		\item If $\log \mathbf{H}(x) \in \Eff_{(X; \mathbf{A})}^*$ then $\mathcal{H}_{\varphi}(x; \bm{s}) = (\mathcal{M} \varphi)(-\bm{s}) \mathbf{H}(x)^{ -\bm{s}}$.
		\item There exists some $\ell \in \Eff_{(X; \mathbf{A})}^*$ such that $\mathcal{H}_{\varphi}(x; \bm{s}) = 0$ if $\mathbf{H}(x) - \ell \not \in \Eff_{(X ; \mathbf{A})}$.
	\end{enumerate}
	\label{Properties of smoothed height function}
\end{lemma}
\begin{proof}
	For the first part note that the integrand of \eqref{Definition of smoothed height} is continuous in $x$ and holomorphic in $s$ and that the integral converges absolutely. 
	
	For the second, note that if $\varphi(||x|| y) \neq 0$ then by taking log norms we have that $\log |y|_{\mathbf{T}_{\NS}} - \log \mathbf{H}(x) \in \Eff_{(X ; \mathbf{A})}$. This implies that $\log |y|_{\mathbf{T}_{\NS}} \in \Eff_{(X ; \mathbf{A})}$. Making the change of variables $z = ||x|| y$ we get
	\begin{equation*}
		\mathcal{H}_{\varphi}(x; \bm{s}) = \int_{\mathbf{T}_{\NS}(\A_F)} \varphi(||x|| y) |y|_{\mathbf{T}_{\NS}}^{- \bm{s}} d \tau_{\NS}(y) = \mathbf{H}(x)^{- \bm{s}} \int_{\mathbf{T}_{\NS}(\A_F)} \varphi(z) |z|_{\mathbf{T}_{\NS}}^{- \bm{s}} d \tau_{\NS}(z).
	\end{equation*}
	
	The last part follows from the compact support of $\varphi$.
\end{proof}
\subsubsection{Height zeta function}
The height zeta function is defined as
\begin{equation*}
	Z(\bm{s},f) := \sum_{P \in T(F)} f(P) \mathbf{H}(P)^{- \bm{s}}
\end{equation*}
for $\bm{s} \in (\Pic(X ; \mathbf{A}))_{\C}$ for which this sum convergences absolutely. We will show at the end that this holds if $\Re(\bm{s}) - \mathbf{2} \in \Eff_{(X ; \mathbf{A})}$.

Twisting the torsor $\pi:Y \to U$ by $\mathfrak{a} \in \HH^1(F, T_{\NS})$ results in a $T_{\NS}$ torsor $\pi_{\mathfrak{a}}: Y \to X$ by Lemma \ref{lem:Twists of universal torsors of toric varieties}. It follows from \eqref{Points on the base and on the torsor} that the Zeta function has the following decomposition
\begin{equation}\label{eq:Decomposition of the Height Zeta function}
	Z(\bm{s}, f) = \sum_{a \in H^1(F, \mathbf{T}_{\NS})} Z_{\mathfrak{a}}(\bm{s}, f)
\end{equation}
where 
\begin{equation*}
	Z_{\mathfrak{a}}(\bm{s}, f) := \sum_{P \in \Theta(F)/T_{\NS}(F)} f(\pi_{\mathfrak{a}}(P)) \mathbf{H}(P)^{-\bm{s}}.
\end{equation*}

We note that \ref{lem:Twists of universal torsors of toric varieties} applied to the $\mathcal{O}_{v}$-integral model $\mathcal{Y}_{\mathcal{O}_v} \to \mathcal{U}_{\mathcal{O}_v}$ implies that $Z_{\mathfrak{a}}(\bm{s}, f) = 0$ for $\mathfrak{a} \not \in \prod_{v \not \in S} \HH^1(\mathcal{O}_v, T_{\NS})$. There are thus only finitely many $\mathfrak{a}$ such that $Z_{\mathfrak{a}}(\bm{s}, f) \neq 0$ so the sum \eqref{eq:Decomposition of the Height Zeta function} is finite.

By Lemma \ref{Properties of smoothed height function} and the product formula we get an equality
\begin{equation*}
	(\mathcal{M} \varphi)(-\bm{s}) Z_{\mathfrak{a}}(\bm{s}, f):= \sum_{P \in\Theta(F)/T_{\NS}(F)} \mathcal{H}_{\varphi}(P ; \bm{s}) f(\pi_{\mathfrak{a}}(P))|P|^{\mathbf{2}}_{\Theta_{\mathbf{A}}}.
\end{equation*}

This is thus a sum over the adelic function
\begin{equation*}
	\mathcal{G}(\cdot ; \bm{s}): \Theta_{\mathbf{A}}(\A_F)/T_{\NS}(F) \to \C: x \to \mathcal{H}_{\varphi}(x; \bm{s}) f(\pi_{\mathfrak{a}}(x))|x|^{\mathbf{2}}_{\Theta_{\mathbf{A}}}.
\end{equation*}

Note that this is well-defined because $\mathcal{H}(\cdot ; {\bm{s}})$ is invariant under $\mathbf{T}_{\NS}(\A_{F})^1$.

We will study this sum by applying Poisson summation \eqref{eq:Poisson summation formula} to the locally compact Hausdorff group $\Theta_{\mathbf{A}}(\A_F)/T_{\NS}(F)$. The following lemma will show that Theorem \ref{thm:Poisson summation} holds.
\begin{lemma}\label{lem:Validity of Poisson summation formula}
	For each $\varphi$ as in \ref{cons: Construction of local measures on groups of multiplicative type} and each compact subset $K \subset (\Pic_{(X ; \mathbf{A})})_{\R}$ there exists a $\varphi'$ as in Construction \ref{cons: Construction of local measures on groups of multiplicative type} with the following property.
	
	For each place $v \in S$ there exists a compactly supported continuous function $g_v: U_{A_v}(F_v) \to \R_{\geq 0}$ which is locally constant if $v$ is non-archimedean and which is smooth if $v$ is archimedean and a neighbourhood $V \subset \Theta_{\mathbf{A}}(\A_F)$ of $1$ such that if $g_v := f_v$ for $v \not \in S$ and $g := \prod_v g_v: (X ; \mathbf{A})(\A_F) \to \R_{\geq 0}$ then 
	\[
	|\mathcal{H}_{\varphi}(x ; {\bm{s}}) f(\pi_{\mathfrak{a}}(x))|x|^{\mathbf{2}}_{\Theta_{\mathbf{A}}}| \leq \mathcal{H}_{\varphi'}(y x ; {\Re(\bm{s})}) g(\pi_{\mathfrak{a}}(y x))|y x|^{\mathbf{2}}_{\Theta_{\mathbf{A}}}.
	\]
 	for all $x \in \Theta_{\mathbf{A}}(\A_F)$, $y \in V$ and $\Re(\bm{s}) \in K$. 
\end{lemma} 
\begin{proof}
	Let $\varphi'$ be a non-negative function such that $|\varphi(x)| \leq \varphi'(w x)$ for all $x \in \mathbf{T}_{\NS}(\A_F)$ and $w \in W$ where $W$ is an open neighbourhood of the $\prod_v T_{\NS A_v}(\mathcal{O}_v)$. The existence of such functions reduces to the local case where it is clear.
	
	For each $v \in S$ choose an open neighbourhood of the identity $V_v \subset \Theta_{A_v}(F_v)$ with compact closure. Moreover, take $V_v \subset \Theta(\mathcal{O}_v)$ for $v$ a non-archimedean place. By compactness we can choose $g_v: U_{A_v}(F_v) \to \R_{\geq 0}$ satisfying the assumptions such that for all $y_v \in V_v$ and $x_v \in U_{A_v}(F_v)$ we have $|f_v(x_v)| \leq g_v(x_v\pi_{a_v}(y_v))$.

	For $v \not \in S$ we let $V_v = \Theta(\mathcal{O}_v)$. Then $f_v \circ \pi_{a}$ is $V_v$ invariant by assumption so $|f_v(\pi_{a}(x_v))| \leq g_v(x_v y_v)$ also for these places $v$. Consider the product $V' := \prod_v V_v \subset \Theta_{\mathbf{A}}(\A_F)$.
	
	Consider the map $\Delta:\mathbf{Y}(\A_F) \times V \to \mathbf{T}_{\NS}(\A_F)/\prod_v T_{\NS A_v}(\mathcal{O}_v) : (x,v) \to \Delta(x,y) = ||x||_v^{-1} ||y x||_v$. The map $\Delta$ factors through $(X; \mathbf{A})(\A_F) \times V'$ since it is invariant under $T_{\NS}(\A_F)$. 
	
	The map $\Delta$ is continuous and sends $(x, 1)$ to $1$ so by compact support of $g$ there exists an neighbourhood of the identity $V \subset V'$ such that $\Delta(x, y) \in W$ for all $y \in V$ for which $g( \pi_{\mathfrak{a}}(x)) \neq 0$.
	
	By compactness of $K$ and $W$ there exists an $L$ such that for all $\bm{s}$ with $\Re(\bm{s}) \in K$ and $w \in W$ we have $L^{-1} \leq ||w|_{\mathbf{T}_{\NS}}^{-\bm{s}}|$. By construction we then have for all $y \in W$ that $\mathcal{H}_{\varphi}(x ; {\bm{s}}) \leq L \cdot \mathcal{H}_{\varphi'}(y x ; {\Re(\bm{s})})$.
	
	The desired bound in the lemma then follows by replacing $g_v$ by $L g_v$ for one place $v$.
\end{proof}
\subsection{The Fourier transform of \texorpdfstring{$\mathcal{G}$}{HL}}\label{sec:Fourier transform of G}
In this part we will study the Fourier transform of the function $\mathcal{G}(\cdot; \bm{s})$. Let $\chi: \Theta_{\mathbf{A}}(\A_F)/T_{\NS}(F) \to \C^{\times}$ be a quasi-character and consider the Fourier transform
\begin{equation*}
	\hat{\mathcal{G}}(\chi ; \bm{s}) := \int_{\Theta_{\mathbf{A}}(\A_F)/T_{\NS}(F)} \chi \cdot \mathcal{G}(\cdot ;\bm{s})d\tau_{\Theta_{\mathbf{A}}}.
\end{equation*}

As a first step we substitute the definition of $\mathcal{H}_{\varphi}(x; \bm{s})$ to get that this is equal
\begin{equation*}
	\begin{split}
		&\int_{\Theta_{\mathbf{A}}(\A_F)/T_{\NS}(F)} \chi(x) f(\pi_a(x))|x|_{\Theta_{\mathbf{A}}}^{\mathbf{2}} \\
		&\int_{\mathbf{T}_{\NS}(\A_F)} \mathbf{1}_{\Eff_{(X; \mathbf{A})}^*}(\log |y|_{\mathbf{T}_{\NS}}) \varphi(||x|| y) |y|_{\mathbf{T}_{\NS}}^{-\bm{s}} d \tau_{\NS}(y) d\tau_{\Theta_{\mathbf{A}}}(x).
	\end{split}
\end{equation*}

This integral will only converge absolutely for $\bm{s}$ sufficiently large. In this section we will always assume that $\bm{s}$ is sufficiently large so that at least one of the sides of each equality converges absolutely. At the end we will show that such an $\bm{s}$ does exist.

Make the change of variables $z = y^{-1} x$. Using that $\tau_{\Theta_{\mathbf{A}}}$ is a Haar measure, Fubini, the equality $||z|| = y ||x||$ and the fact that $f$ is $\mathbf{T}_{\NS}(\A_F)$-invariant shows that $\hat{\mathcal{G}}(\chi ; \bm{s})$ is equal to
\begin{equation}\label{eq:Decomposition into two integrals}
	\begin{split}
		&\int_{\Theta_{\mathbf{A}}(\A_F)} \chi(z) f(\pi_a(z)) \varphi(||z||)|z|_{\Theta_{\mathbf{A}}}^{\mathbf{2}} d\tau_{\Theta_{\mathbf{A}}}(z) \\ &\int_{\mathbf{T}_{\NS}(\A_F)/T_{\NS}(F)} \chi(y) \mathbf {1}_{\Eff_{(X; \mathbf{A})}}(\log |y|_{\mathbf{T}_{\NS}}) |y|_{\mathbf{T}_{\NS}}^{\mathbf{2} - \bm{s}} d \tau_{\NS}(y).
	\end{split}
\end{equation}

We will evaluate these two integrals separately. The inner integral can be evaluated explicitly.
\begin{lemma}\label{lem:Integral over Neron-Severi torus}
	The inner integral of \eqref{eq:Decomposition into two integrals} converges absolutely if $\Re(\bm{s})$ is sufficiently large in terms of $\chi$.
	
	If the restriction of $\chi$ to $\mathbf{T}_{\NS}(\A_F)^1/T_{\NS}(F)$ is non-trivial then
	\[
		\int_{{\mathbf{T}_{\NS}(\A_F)/T_{\NS}(F)}}\mathbf \chi(y) \mathbf{1}_{\Eff_{(X; \mathbf{A})}^*}(\log |y|_{\mathbf{T}_{\NS}}) |y|_{\mathbf{T}_{\NS}}^{\mathbf{2} - \bm{s}} d \tau_{\NS}(y) = 0.
	\]
	
	Otherwise the restriction of $\chi$ is equal to $|\cdot|_{\mathbf{T}_{\NS}}^{\bm{z}}$ for some $\bm{z} \in \Pic(X, \mathbf{A})_{\R}$ and 
	\begin{equation*}
		\begin{split}
			\int_{\mathbf{T}_{\NS}(\A_F)/T_{\NS}(F)} \mathbf \chi(y) \mathbf{1}_{\Eff_{(X; \mathbf{A})}^*}(\log |y|_{\mathbf{T}_{\NS}}) |y|_{\mathbf{T}_{\NS}}^{\mathbf{2} - \bm{s}} d \tau_{\NS}(y) = \\ \tau_{\mathbf{T}_{\NS}^1}(\mathbf{T}_{\NS}(\A_F)^1/T_{\NS}(F))\mathcal{X}_{\Eff_{(\mathbf{X}; \mathbf{A})}}(\bm{s} - \bm{2} - \bm{z}).
		\end{split}
	\end{equation*}
\end{lemma} 
\begin{proof}
	For absolute convergence we may replace $\chi$ by $|\chi|$. This quasi-character is trivial when restricted to $\mathbf{T}_{\NS}(\A_F)^1/T_{\NS}(F)$ since this group is compact by Lemma \ref{lem: Quotient is compact}. This thus follows from the second case.
	
	We rewrite the relevant integral as
	\begin{equation*}
		\begin{split}
		&\int_{\mathbf{T}_{\NS}(\A_F)/\mathbf{T}_{\NS}(\A_F)^1} {1}_{\Eff_{(X; \mathbf{A})}^*}(\log |y|_{\mathbf{T}_{\NS}}) |y|_{\mathbf{T}_{\NS}}^{\mathbf{2} - \bm{s}} \\ &\int_{\mathbf{T}_{\NS}(\A_F)^1/ T_{\NS}(F)} \chi(y x) d\tau_{\mathbf{T}_{\NS}^1}(x) d \tau_{\mathbf{T}_{\NS}/\mathbf{T}^1_{\NS}}(y).
				\end{split}
	\end{equation*}
	
	The inner integral is $0$ if the restriction of $\chi$ is to the compact group $\mathbf{T}_{\NS}(\A_F)^1/T_{\NS}(F)$ is non-trivial by orthogonality of characters. If the restriction is trivial then it is equal to the volume $d\tau_{\mathbf{T}_{\NS}^1}(\mathbf{T}_{\NS}(\A_F)^1/T_{\NS}(F))$. 
	
	The log norm $\log | \cdot |_{\mathbf{T}_{\NS}}$ induces an isomorphism $\mathbf{T}_{\NS}(\A_F)/T_{\NS}(F)^1 \cong \Pic (\mathbf{X}, \mathbf{A})^{*}_{\R}$ by Lemma \ref{Norm map is surjective}, and this isomorphism sends the Haar measure $d \tau_{\mathbf{T}_{\NS}/\mathbf{T}^1_{\NS}}$ to the measure $du_{\Pic(X ; \mathbf{A})}$ by construction.
	
	The group $\Pic(\mathbf{X}, \mathbf{A})^{*}_{\R}$ is a real vector space so every quasi-character $\Pic (\mathbf{X}, \mathbf{A})_{\R} \to \C^{\times}$ is given by $e^{\langle\bm{z}, \cdot \rangle}$ for some $\bm{z} \in (\Pic (\mathbf{X}, \mathbf{A}))_{\C}$. The quasi-character $\chi$ is thus equal to $e^{\langle\bm{z}, \log | \cdot |_{\mathbf{T}_{\NS}} \rangle} = | \cdot |_{\mathbf{T}_{\NS}}^{\bm{z}}$.
	
	We then compute that
	\begin{equation*}
		\begin{split}
			&	\int_{\mathbf{T}_{\NS}(\A_F)/\mathbf{T}_{\NS}(\A_F)^1} {1}_{\Eff_{(X; \mathbf{A})}^*}(\log |y|_{\mathbf{T}_{\NS}}) |y|_{\mathbf{T}_{\NS}}^{\mathbf{2} - \bm{s} + \bm{z}} d \tau_{\mathbf{T}_{\NS}/\mathbf{T}^1_{\NS}} = \\
			&\int_{\Pic (\mathbf{X}, \mathbf{A})^{*}_{\R}} {1}_{\Eff_{(X; \mathbf{A})}^*}(u) e^{- \langle \bm{s} - \mathbf{2} - \mathbf{z}, u \rangle} du_{\Pic(X ; \mathbf{A})}.
		\end{split}
	\end{equation*}
	This is exactly the definition of $\mathcal{X}_{\Eff_{(\mathbf{X}; \mathbf{A})}}(\bm{s} - \bm{2} - \bm{z})$.
\end{proof}

The integrand of the other integral is a product of the following local factors. For each quasi-characters $\chi_v: \Theta_{A_v}(F_v) \to \C^{\times}$ consider the local integral
\begin{equation*}
	\hat{G}_v(\chi_v) := \int_{\Theta_{A_v}(F_v)} \chi_v(x_v) f_v(\pi_{a_v}(x_v)) \varphi_v(||x_v||_v) |x_v|^{\mathbf{2}}_{\Theta_{A_v}} d\tau_{\Theta_{A_v}}(x_v).
\end{equation*}
Every quasi-character $\chi$ can be written as a product $\chi = \prod_v \chi_v$. By the definition of $\tau_{\Theta_{\mathbf{A}}}$ we then have
\begin{equation}\label{eq:Product of local integrals is global integral}
	\int_{\Theta_{\mathbf{A}}(\A_F)} \chi(z) f(\pi_a(z)) \varphi(||z||)|z|_{\Theta_{\mathbf{A}}}^{\mathbf{2}} d\tau_{\Theta_{\mathbf{A}}}(z) = \prod_{\rho \in D / \Gamma_F} \zeta^{*}_{F_{\rho}}(1)^{-1} \prod_{v \in \Omega_F} \hat{G}_v(\chi_v).
\end{equation}
\subsection{Local Fourier transforms}
Let us first note the following fact
\begin{lemma}
	The product function $Y_{A_v}(F_v) \to \R: x_v \to f_v(\pi_a(x_v)) \varphi(||x_v||_{v})$ is compactly supported for all places $v \in \Omega_{F}$ and is smooth for $v \in \Omega_F^{\infty}$.
	\label{Product of f and phi is compactly supported}
\end{lemma}
\begin{proof}
	The function $(\pi_{a_v}, ||\cdot||): Y_{A_v}(F_v)/T_{\NS, v}(\mathcal{O}_v) \to X(F_v) \times T_{\NS, v}(F_v)/T_{\NS, v}(\mathcal{O}_v)$ is an open and closed immersion by Lemma \ref{Torsors are principal bundles in the analytic topology}.
	
	The function $X(F_v) \times T_{\NS, v}(F_v)/T_{\NS, v}(\mathcal{O}_v) \to \R: (x_v,t_v) \to f_v(x_v) \varphi_v(x)$ is a product of compactly supported functions which are smooth if $v \in \Omega_F^{\infty}$ and thus has those same properties. The lemma follows. 
\end{proof}
Let $v \in \Omega_F$ be a place. The torus $\Theta_{A_v}$ is by definition the Weil restriction of $\G_m$ for the \'etale algebra corresponding to the $\Gamma_v$-set $D \cup A_v$. For $\rho_v \in (D \cup A_v) / \Gamma_v$ let $F_{\rho_v}$ be the field extension of $F_v$ corresponding to the $\Gamma_v$-orbit $\rho_v$. We then have
\begin{equation*}
	\Theta_{A_v}(F_v) \cong \prod_{\rho_v \in (R \cup A_v) / \Gamma_v} F_{\rho_v}^{\times}. 
\end{equation*}

Under this identification we have an equality of measures 
\begin{equation}
	\tau_{\Theta_{A_v}}(x) = \prod_{\rho_v \in (D \cup A_v) / \Gamma_v} \zeta_{F_{\rho_v}}(1) |x_{\rho_v}|^{-1} dx_{\rho_v}.
	\label{Haar measure on the universal torsor torus}
\end{equation}

We can also make the following identification
\begin{equation*}
	\Hom(\Theta_{A_v}(F_v), \C^{\times}) \cong \prod_{\rho_v \in (D \cup A_v) / \Gamma_v} \Hom(F^{\times}_{\rho_v}, \C^{\times}).
\end{equation*}
This space naturally has the structure of a complex manifold\footnote{With a countable number of components.}.

Every quasi-character $\chi_v: \Theta_{A_v}(F_v) \to \C^{\times}$ can thus be uniquely written as a product $\chi_{v} = \prod_{\rho \in (D \cup A_v) / \Gamma_v} \chi_{\rho_v}$ where $\chi_{\rho_v}: F_{\rho_v}^{\times} \to \C^{\times}$ is a quasi-character.
	
We can also conclude that every quasi-character $\chi_v: \Theta_{A_v}(F_v) \to \C^{\times}$ can be written as a product $\chi_v = \Tilde{\chi}_v |\cdot|_v^{\mathbf{z}}$ for some $\mathbf{z} \in \C[(D \cup A_v) / \Gamma_v]$ and character $ \Tilde{\chi}_v: \Theta_{A_v}(F_v) \to S^1$. Note that $\Re(\mathbf{z}) \in \R[(D \cup A_v) / \Gamma_v]$ is uniquely determined by the quasi-character $\chi$. We can choose $\Tilde{\chi} = 1$ if and only if $\chi$ is trivial when restricted to $\Theta_{A_v}(\mathcal{O}_v)$, in this case we say that $\chi$ is \emph{unramified}.

For $\mathbf{x} \in \R[(D \cup A_v) / \Gamma_v]$ we define the \emph{tube domain}
\begin{equation*}
	\mathbf{T}_{> \mathbf{x}} := \{\chi \in \Hom(\Theta_{A_v}(F_v), \C^{\times}) : \Re(\chi) > \mathbf{x}\}.
\end{equation*}
The partial order on $\R[(D \cup A_v) / \Gamma_v]$ is the components-wise one.

\subsubsection{Finite places}
For non-archimedean places we will use the following lemma
\begin{lemma}\label{lem:Fourier transform at places of bad reduction}
	Let $v \in \Omega_{F}^{\emph{fin}}$. The Fourier transform $\hat{G}_v(\chi_v)$ converges absolutely and is holomorphic for $\chi_v \in \mathbf{T}_{> - \mathbf{2}}$. There exists a compact open subgroup $\mathbf{K}_v \subset \Theta(F_v)$ such that $\hat{G}_v(\chi_v) = 0$ if $\chi_v$ is non-trivial when restricted to $\mathbf{K}_v$.
\end{lemma}
\begin{proof}
	Let $\mathbf{K}_v \subset \Theta(\mathcal{O}_v)$ be the maximal subgroup such that $f_v$ and $||\cdot||_v$ are invariant under $\mathbf{K}_v $. This is an open subgroup because $f_v$ and $||\cdot||_v$ are locally constant.
	
	By the usual argument via orthogonality of characters we conclude that $\hat{G}_v(\chi_v) = 0$ if $\chi_v$ is trivial when restricted to $\mathbf{K}_v$.
	
	Now write the quasi-character $\chi_v = \Tilde{\chi}_v | \cdot |^{\mathbf{z}}$ for some $\mathbf{z} \in \R[E/ \Gamma_v]$. The holomorphicity of $\hat{G}_v(\chi_v)$ will follow from absolute convergence so we may assume that $\Tilde{\chi}_v = 1$. By \eqref{Haar measure on the universal torsor torus} we have
	\begin{equation*}
		\int_{\Theta_{A_v}(F_v)} f_v(\pi_{a_v}(x_v)) \varphi_v(||x_v||_v) |x_v|^{\mathbf{1} + \mathbf{z}}_{\Theta_{A_v}} \prod_{\rho_v \in D / \Gamma_v} \zeta_{F_{\rho_v}}(1) dx_{\rho_v}.
	\end{equation*}

	This integral converges absolutely as long as $\mathbf{z} + \mathbf{2} > 0$ by \cite[Prop.~3.4]{Chambert-Loir2010Igusa} and Lemma \ref{Product of f and phi is compactly supported}.
\end{proof}

We can do much better than the above lemma at the places of good reduction. We note the similarity of the following lemma to \cite[Lem.~3.2.2]{Chambert-Loir2010Integral}. 
\begin{lemma}
	Let $v \not \in S$. If $\chi_v$ is ramified then $\hat{G}_v(\chi_v) = 0$, if $\chi_v$ is unramified and contained in the tube domain $\mathbf{T}_{>- \mathbf{2}}$ then
	\begin{equation*}
		\hat{G}_v(\chi_v) = \prod_{\rho_v \in \Gamma_v \setminus E} L_{F_{\rho_v}}(2 ; \chi_{\rho_v})\left(1 + O\left(q_{v}^{- 4 - 2\min_{\rho_v \in D / \Gamma_v}\Re(\chi_{\rho_v})}\right)\right).
	\end{equation*}
	\label{Fourier transform at places of good reduction}
\end{lemma}
\begin{proof}	
	The assumption $v \not \in S$ implies that $x_v \to f_v(\pi_a(x_v)) \varphi_v(x_v)$ is the indicator function of $\mathcal{Y}(\mathcal{O}_v)$, it is in particular invariant under the action of $\Theta_{A_v}(\mathcal{O}_v)$ so $\hat{G}_v(\chi_v) = 0$ if $\chi_v$ if $\chi_v$ is ramified by orthogonality of characters.
	
	If $\chi_v$ is unramified then it is equal to $|\cdot|_{\Theta_v}^{\mathbf{z}}$ for some $\mathbf{z} \in \C[\Gamma_v \setminus E]$. By \eqref{Haar measure on the universal torsor torus} we have to compute
	\begin{equation*}
		\hat{G}_v(\chi_v) = \prod_{\rho_v \in D / \Gamma_v} \zeta_{F_{\rho_v}}(1) \int_{\Theta(F_v) \cap \mathcal{Y}(\mathcal{O}_v)} |x|^{\mathbf{1} + \mathbf{z}} \prod_{\rho_v \in D / \Gamma_v} d x_{\rho_v}.
	\end{equation*}
	
	For each $\Gamma_v$-invariant subset $B \subset D$ let $\mathcal{Z}_B := \cap_{\beta \in B} \mathcal{Z}_{\beta}$ be the intersection of the corresponding varieties, it is defined over $F$. The closed subscheme $\mathcal{Z}_B$ is an intersection of toric divisors so the open subscheme $\mathcal{Z}_B^{\circ} := \mathcal{Z}_B \setminus \cup_{\varepsilon \in D \setminus B} \mathcal{Z}_{\varepsilon}$ is the torus $\G_m^{D \setminus B}$ if $B$ is the set of rays of a cone in $\Sigma$ and empty otherwise. 
	
	For $\rho_v \in D / \Gamma_v$ let $f_{\rho_v}$ be the degree of the (necessarily unramified) extension $F_{\rho_v}/F$.
	
	The \'etale extension $F_v[D \setminus B]/F_v$ is unramified so $\# \G_m^{D \setminus B}(\F_v) = \prod_{\rho_v \in (D \setminus B) / \Gamma_v} \F_{\rho_v}^{\times} = \prod_{\rho_v \in D / \Gamma_v} (q^{f_{\rho_v}} - 1) = \prod_{\rho_v \in D / \Gamma_v} q^{f_{\rho_v}}\zeta_{F_{\rho_v}}(1)$.
	
	Substituting this into Denef's formula \cite[Prop. 4.5]{Chambert-Loir2010Integral} and simplifying we find the equality
	\begin{equation*}
		\hat{G}_v(\chi_v) = \sum_{\sigma \in \Sigma^{\Gamma_v}} \prod_{\rho_v \in \sigma / \Gamma_v} (q_v^{f_{\rho_v}(2 + z_{\rho_v})} - 1)^{-1}.
	\end{equation*}
	
	The contribution of the trivial cone is $1$ and the contribution of a ray $\rho_v \in D^{\Gamma_v}$ is $q^{-2 - z_{\rho_v}} + O(q^{- 4 - 2 \Re(z_{\rho_v})})$. The contribution of higher dimensional cones $\sigma \in \Sigma^{\Gamma_v}$ is $O(q^{- 4 - 2 \min_{\rho_v \in \sigma / \Gamma_v}(\Re(z_{\rho}))})$. The lemma follows since $L_{\rho_v}(2, \chi_v) = (1 - q^{f_{\rho_v}(-2 - z_{\rho_v})})^{-1}$.
\end{proof}
\subsubsection{Infinite places}
When $v$ is an archimedean then every quasi-character $\chi_v$ of $F_{\rho_v}$ can be written as a product $y \to y^n |y|^{z}$ for unique $z \in \C$ and $n \in \{0,1\}$ or $n \in \Z$ for $F_{\rho_v} \cong \R$ or $F_{\rho_v} \cong \C$ respectively. We define $||\chi_v|| := \max(|z|, n)$.
\begin{lemma}\label{lem:Local Mellin transform at archimedean places}
	Let $v \in \Omega_F^{\infty}$. The Fourier transform $\hat{G}_v(\chi_v)$ converges absolutely for $\mathbf{T}_{> \mathbf{-2}}^{D \cup A_v}$ and there exists a holomorphic function $\varpi_{A_v}: \mathbf{T}_{> \mathbf{-\frac{5}{2}}}^{D \cup A_v} \to \C$ such that
	\begin{equation*}
		\hat{G_v}(\chi_v) = \varpi_{A_v}( \chi_v) \prod_{\substack{\rho_v \in D / \Gamma_v \cup A_v \\ \chi_{\rho_v} = |\cdot|^{z_{\rho_v}}}} \frac{1}{2 + z_{\rho_v}}.
	\end{equation*}

	The function $\varpi_{A_v}$ is rapidly decreasing in vertical strips in the sense that for every compact $C \subset \R_{> -\frac{5}{2}}[(D \cup A_v)/\Gamma_v]$ and for all $\kappa > 0$ we have for $\chi_v$ with $\Re(\chi_v) \in C$ that
	\begin{equation*}
		\varpi_{A_v}(\chi_v) \ll_{\kappa, C} (1 + ||\chi_v||)^{- \kappa}.
	\end{equation*}
\end{lemma}
\begin{proof}
	Write the quasi-character as $\chi_v = \Tilde{\chi}_v | \cdot |^{\mathbf{z}}$ for $\mathbf{z} \in \C[D / \Gamma_v \cup A_v]$ and $\Tilde{\chi}_v$ a character. We may moreover assume that if $\chi_{\rho_v}$ is unramified then $\Tilde{\chi}_{\rho_v} = 1$. We then have by \eqref{Haar measure on the universal torsor torus} that
	\begin{equation*}
	\hat{G}_v(\chi_v) = \int_{\Theta_{A_v}(F_v)} f_v(\pi_a(x_v)) \varphi_v(||x_v||_v) \Tilde{\chi}_v(x_v) |x_v|^{\mathbf{1} + \mathbf{z}}_{\Theta_{A_v}} \prod_{\rho_v \in (D \cup A_v) / \Gamma_v} d x_{\rho_v}
	\end{equation*}

	This is the Mellin transform of a smooth compactly supported function by Lemma \ref{Product of f and phi is compactly supported}. The lemma thus follows from properties of Mellin transforms, c.f. \cite[Lemma~3.3]{Chambert-Loir2010Igusa} and \cite[Lemma~3.4.1]{Chambert-Loir2010Integral}.
\end{proof}

\subsubsection{The global Fourier transform}
For each place $v \in \Omega_F^{\text{fin}}$ let $\mathbf{K}_v \subset \Theta(F_v)$ be a compact open group as in Lemma \ref{lem:Fourier transform at places of bad reduction}. By Lemma \ref{Fourier transform at places of good reduction} we may assume that $\mathbf{K}_v = \Theta(\mathcal{O}_v)$ for $v \not \in S$.

For $v \in \Omega_F^{\infty}$ put $\mathbf{K}_v = 1$. Define $\mathbf{K} := \prod_{v \in \Omega_{F}} \mathbf{K}_v \subset \Theta_{\mathbf{A}}(\A_F)$. 

Any quasi-character $\chi: \Theta_{\mathbf{A}}(\A_F) \to \C^{\times}$ is equal to a product $\prod_v \chi_v$ of quasi-characters $\Theta_{A_v}(F_v) \to \C^{\times}$, all but finitely many of which are trivial when restricted to $\Theta_{A_v}(\mathcal{O}_v)$. We define $||\chi|| = \max_{v \in \Omega_F^{\infty}} ||\chi_v||$.

Also, every quasi-character $\chi: \Theta_{\mathbf{A}}(\A_F) \to \C^{\times}$ can be written as a product 
\[
\prod_{\rho \in D / \Gamma_F} \chi_{\rho} \prod_{v \in \Omega_{F}^{\infty}} \prod_{\rho_v \in A_v /\Gamma_v} \chi_{\rho_v}
\] where $\chi_{\rho}: \G_m(\A_{F_{\rho}}) \to \C^{\times}$ and $\chi_{\rho_v}: \G_m(F_{\rho_v}) \to \C^{\times}$ by the definition of $\Theta_{\mathbf{A}}$.

We define $\Re(\chi) \in \R[D / \Gamma_F] \times \prod_{v \in \Omega_{F}^{\infty}} \R[A_v/\Gamma_v] = \hat{\Theta}_{\mathbf{A}}$ such that $|\chi| = |\cdot|_{\mathbf{A}}^{\Re(\chi)}$ and for $(\mathbf{x}, \mathbf{y}) \in \hat{\Theta}_{\mathbf{A}}$ we define the tube domain
\[
\mathbf{T}_{> \mathbf{x}, > \mathbf{y}} := \{\chi \in \Hom(\Theta_{\mathbf{A}}(\A_F), \C^{\times}) : \Re(\chi) > (\mathbf{x}, \mathbf{y})\}.
\]
The order is the component-wise one.

We can combine the above lemmas to immediately deduce the following.
\begin{lemma}\label{lem:Global Fourier transform}
	Let $\chi$ be a quasi-character $\Theta_{\mathbf{A}}(\A_F)/T_{\NS}(\A_F)^1 \to \C^{\times}$. If $\Re(\chi) \in \mathbf{T}_{> (\mathbf{-\frac{3}{2}}, \mathbf{-\frac{5}{2}})}$ then the integral $\hat{\mathcal{G}}(\chi ; \mathbf{s})$ converges absolutely.
	
	If the restriction of $\chi$ to $\mathbf{K}$ or $\mathbf{T}_{\NS}(\A_F)^1$ is non-trivial then $\hat{\mathcal{G}}(\chi ; \mathbf{s}) = 0$.
	
	In general there exist a holomorphic function $\eta_{\mathbf{A}}$ on the tube domain $\mathbf{T}_{> - (\mathbf{\frac{3}{2}}, \mathbf{\frac{5}{2}})}$ such that
	\begin{equation*}
		\hat{\mathcal{G}}(\chi ; \mathbf{s}) = \prod_{\rho \in D / \Gamma} L_{F_\rho}(2; \chi_{\rho}) \eta_{\mathbf{A}}(\chi) \prod_{\substack{v \in \Omega^{\infty} \\ \rho_v \in A_v / \Gamma_v \\ \chi_{\rho_v} = |\cdot|^{z_{\rho_v}}}} \frac{1}{2 + z_{\rho_v}} \mathcal{X}_{\Eff_{(X; \mathbf{A})}}(\mathbf{s} - \mathbf{2} - \bm{z}).
	\end{equation*}
	
	The function $\eta_{\mathbf{A}}$ is rapidly decreasing in vertical strips in the sense that for any compact $C \subset \R_{> -\frac{3}{2}}[D / \Gamma_F] \times \prod_{v \in \Omega_{F}^{\infty}} \R_{> -\frac{5}{2}}[A_v / \Gamma_v]$ and for all $\kappa > 0$ we have for $\chi$ with $\Re(\chi) \in C$ that
	\begin{equation}\label{eq:Rapid decay of eta}
		\eta_{\mathbf{A}}(\chi) \ll_{\kappa, C} (1 + ||\chi||)^{- \kappa}.
	\end{equation}
\end{lemma}

The function $\hat{\mathcal{G}}(\chi ; \mathbf{s})$ thus has a meromorphic continuation to the tube domain $\mathbf{T}_{> - (\mathbf{\frac{3}{2}}, \mathbf{\frac{5}{2}})}$.
\subsection{Poisson summation}\label{sec:Poisson summation}
We will now apply the Poisson summation formula. Let us first study the Pontryagin dual $(\Theta_{\mathbf{A}}(\A_F)/\Theta(F))^{\vee}$.

By Lemma \ref{lem:Norm map is surjective} we have an exact sequence
\begin{equation}\label{eq:Exact sequence of Theta(A)/Theta(F)}
 	1 \to \Theta_{\mathbf{A}}(\A_F)^1 /\Theta(F) \to \Theta_{\mathbf{A}}(\A_F)/\Theta(F) \xrightarrow{|\cdot|_{\Theta_{\mathbf{A}}}} = \R_{> 0}^{\times}[D/ \Gamma_F] \times \prod_{v \in \Omega_{F}^{\infty}} \R_{> 0}^{\times}[A_v /\Gamma_v] \to 1
\end{equation}

The group $\Theta_{\mathbf{A}}(\A_F)^1 /\Theta(F)$ is compact so its Pontryagin dual is discrete. The measure dual to $\tau_{\Theta_{\mathbf{A}}^1}$ is $\tau_{\Theta_{\mathbf{A}}^1}(\Theta_{\mathbf{A}}(\A_F)^1 /\Theta(F))^{-1}$ times the counting measure.

The volume of this group has been computed in Proposition \ref{prop:Tamagawa volume of adelic groups of multiplicative type}, in the notation of that proposition we have $T' = \Theta$, $E = \emptyset$ and $A_v = A_v$. We have $\tau(\Theta_{\mathbf{A}}) = 1$ by Lemma \ref{lem: Shapiro's lemma} and the definition of $\tau(\Theta_{\mathbf{A}})$ so $\tau_{\Theta_{\mathbf{A}}^1}(\Theta_{\mathbf{A}}(\A_F)^1/\Theta(F)) = c_{\mathbf{A}}$.

The group $\R_{> 0}^{\times}[D/ \Gamma_F] \times \prod_{v \in \Omega_{F}^{\infty}} \R_{> 0}^{\times}[A_v /\Gamma_v]$ is a power of $\R^{\times}_{> 0}$. The Pontryagin dual of $\R^{\times}_{> 0}$ is $\R$, where $x \in \R$ corresponds to the character $|\cdot|^{ix}$. The corresponding dual Haar measure on $\R$ is $\frac{1}{2 \pi}$ times the Lebesgue measure.

A formal application of the Poisson summation \eqref{eq:Poisson summation formula} formula thus leads to the equality
\begin{equation} 
	(\mathcal{M} \varphi)(-\bm{s}) Z_{\mathfrak{a}}(\bm{s}; f, \varphi) = c_{\mathbf{A}}^{-1} \sum_{\chi \in (\Theta_{\mathbf{A}}(\A_F)^1/\Theta(F))^{\vee}} \mathfrak{Z}(\chi ; \bm{s}).
	\label{Application of Poisson summation}
\end{equation}
where
\begin{equation}
	\mathfrak{Z}(\chi ; \bm{s}) := \mathop{\int \cdots \int}_{\substack{\rho \in D / \Gamma_F \\ \Re(z_{\rho}) = 0}} \mathop{\int \cdots \int}_{\substack{v \in \Omega_F^{\infty} \\ \rho_v \in A_v / \Gamma_v \\ \Re(w_{\rho_v}) = 0}} \hat{\mathcal{G}}(\chi \cdot | \cdot|_{\Theta_{\mathbf{B}}}^{(\bm{z}, \bm{w})}; \bm{s}) d\frac{\bm{z}}{2 i \pi} d \frac{\bm{w}}{2 i \pi}
	\label{Definition of Fraktur Z}
\end{equation}
and
\begin{align*}
	&d\frac{\bm{z}}{2 i \pi} := \prod_{\rho \in D / \Gamma_F} d \frac{z_{\rho}}{2 i \pi} & d\frac{\bm{w}}{2 i \pi} := \prod_{v \in \Omega_F^{\infty}} \prod_{\rho_v \in A_v / \Gamma_v} d\frac{w_{\rho_v}}{2 i \pi}. 
\end{align*} 

The subgroup of $(\Theta_{\mathbf{A}}(\A_F)^1/\Theta(F))^{\vee}$ consisting of characters which are trivial when restricted to $\mathbf{K}$ is a finitely generated group.

The integrand decays sufficiently fast in vertical strips to ensure absolute convergence by \eqref{eq:Rapid decay of eta} and since $L$-functions grow at worst polynomially in vertical strips, see e.g.~\cite[Lemma~3.1]{Loughran2018Number}. The sum over $\chi$ also converges absolutely for the same reason, noting that the conductors of the $\chi_{\rho}$ are bounded since $\chi$ is trivial when restricted to $\mathbf{K}$. Theorem \ref{thm:Poisson summation} then holds by noting that this absolute convergence also holds if one replaces $\phi$ and $f$ by the $\phi'$ and $g$ in Lemma \ref{lem:Validity of Poisson summation formula}.

\subsubsection{Moving contours}
Let $L \in \Eff_{(X; \mathbf{A})}^{\circ}$ be a big line bundle, $a := a((X, \mathbf{A}); L)$ and $C \subset \Eff_{(X; \mathbf{A})}$ the minimal face containing the adjoint divisor $K_{(X; \mathbf{A})} + aL$. Let $b := b((X, \mathbf{A}); L)$ be the codimension of $C$.

Let $D_{\adj} \subset E$ be the $\Gamma_F$-set of toric divisors of $U$ whose divisor class lies in the face $C$. Similarly, for each place $v$ let $A_{v, \adj} \subset A_v$ be the $\Gamma_{F_v}$-set of toric divisors contained in the face $C$. The divisors $D_{\adj}$ and $A_{v, \adj}$ generate $C$ since the toric divisors generate the effective cone.

Let $E := D \setminus D_{\adj}$ and $B_v := A_v \setminus A_{v, \adj}$ for all $v$.

The following lemma shows that $\mathfrak{Z}(\chi ; \bm{s})$ can be analytically continued.
\begin{lemma}\label{lem: Moving the contours}
	There exists an $\varepsilon > 0$ independent of $\chi$ such that the function $s \to (\mathcal{M} \varphi)(-sL )Z_{\mathfrak{a}}(s L, f)$ has a meromorphic continuation to $\Re(s) > a - \varepsilon$. It is holomorphic in this region except for possibly a single pole of order at most $b$ at $s = a$.
	
	For $\bm{z} \in \C[E / \Gamma_F]$, $\bm{w} \in \prod_{v \in \Omega_{F}^{\infty}} \C[B_v/\Gamma_v]$ define $\xi(f, s, \bm{z}, \bm{w})$ as the following sum of integrals
	\begin{equation}\label{eq:Definition of xi}
		c_{\mathbf{A}}^{-1}
		\sum_{\substack{\chi \in (\Theta_{\mathbf{A}}(\A_F)^1/\Theta(F))^{\vee} \\ \chi_{\rho} = 1: \rho \in E/\Gamma_F \\ \chi_{\rho_v} = 1: \rho_v \in B/\Gamma_v}}
		\mathop{\int \cdots \int}_{\substack{\rho \in D_{\adj} / \Gamma_F \\ \Re(z_{\rho}) = 0}}\mathop{\int \cdots \int}_{\substack{v \in \Omega_F^{\infty} \\ \rho_v \in A_{v, \adj} /\Gamma_v \\ \Re(w_{\rho_v}) = 0}} 
		\hat{\mathcal{G}}(\chi \cdot | \cdot|_{\Theta_{\mathbf{A}}}^{(\bm{z}, \bm{w}) + (\bm{z}_{\adj}, \bm{w}_{\adj})}; s L) d\frac{\bm{z}_{\adj}}{2 i \pi} d \frac{\bm{w}_{\adj}}{2 i \pi}.
	\end{equation}
	Where 
	\begin{align*}
		&\bm{z}_{\adj} = (z_{\rho})_{\rho \in D_{\adj} / \Gamma_F} \in \C[D_{\adj} / \Gamma_F], &\bm{w}_{\adj} = (w_{\rho_v})_{\rho_v \in A_{v, \adj}/\Gamma_v} \in \prod_{v \in \Omega_{F}^{\infty}} \C[A_{v, \adj}/\Gamma_v].
	\end{align*}
	
	The sum of integrals $\xi(f, s, \bm{z}, \bm{w})$ converges absolutely for $\Re(s)$ sufficiently large and $\Re(z_{\rho}) > -1, \Re(w_{\rho_v}) > -2$ for $\rho \in E/ \Gamma_F$, $\rho_v \in B_v$ respectively. The function 
	\[
	(s - a)^b \prod_{\rho \in E / \Gamma_F} (z_{\rho} + 1) \prod_{\substack{v \in \Omega_F \\ \rho_v \in B_v/\Gamma_v}} (w_{\rho_v} + 2)\xi(f, s, \bm{z}, \bm{w})
	\] 
	can be analytically continued to the region $\Re(s) > a - \varepsilon, \Re(z_{\rho}) > -1 - \varepsilon, \Re(w_{\rho_v}) > -2 - \varepsilon$ and its value at $s = a, z_{\rho} = -1, w_{\rho_v} = -2 $ is equal to 
	\[
	\lim_{s \to a} (s - a)^b (\mathcal{M} \varphi)(-sL )Z_{\mathfrak{a}}(s L, f).
	\]
	\end{lemma}
\begin{proof}
	Choose $0 < \delta < \frac{1}{4}$ such that 
	\begin{equation}
		K_{(X; \mathbf{A})} + aL - \sum_{\rho \in D_{\adj}} 2\delta [Z_{\rho}] - \sum_{\substack{v \in \Omega_{F}^{\infty} \\ \rho_v \in A_{v, \adj}}} 2\delta [Z_{\rho_v}] \in \Eff_{(X; \mathbf{A})}. 
		\label{Choice of delta}
	\end{equation}
	Such a $\delta$ exists by the definitions of $D_{\adj}$ and $A_{v, \adj}$.
	
	For each $\chi \in (\Theta_{\mathbf{A}}(\A_F)^1/\Theta(F))^{\vee}$ move in the definition of $\mathfrak{Z}(\chi; sL)$ the contours coming from $\rho \in E / \Gamma_F$ and $\rho_v \in B_v / \Gamma_v$ to the lines $\Re(z_{\rho}) = -\frac{5}{4}$ and $\Re(w_{\rho_v}) = - \frac{9}{4}$ respectively. Moreover, also move the contours corresponding to $\rho \in D_{\adj} / \Gamma_F$ and $\rho_v \in A_{v, \adj} / \Gamma_v$ to the lines $\Re(z_{\rho}) = -1 + \delta$ and $\Re(w_{\rho_v}) = - 2 + \delta$ respectively.
	
	By Lemma \ref{lem:Global Fourier transform} each contour moved this way only meets at most a single pole at $z_{\rho} = 1$ and $w_{\rho_v} = 0$ respectively. And it only meets such a pole if $\chi_{\rho_v} = 1$ and $\chi_{\rho_v} = 1$ respectively.
	
	By the residue theorem $\mathfrak{Z}(\chi ; s L)$ is thus equal to the sum over all possible subsets $I \subset E/\Gamma_F$ and $J_v \subset B_v / \Gamma_v$ for all $v \in \Omega_{F}^{\infty}$ such that $\chi_{\rho} = 1$ for all $\rho \in I$ and $\chi_{\rho_v} = 1$ for all $\rho_v \in J_v$ of the following terms
	\begin{equation}\label{eq:Terms after moving the contours}
	\begin{split}
		&\prod_{\rho \in I} \zeta^*_{_{F_{\rho}}}(1) \mathop{\int \cdots \int}_{\substack{ \Re(z_{\rho}) = -\frac{5}{4}: \rho \in I^c \\ \Re(z_{\rho}) = - 1 + \delta: \rho \in D_{\adj}/ \Gamma_{F}}} \prod_{\rho \in I^c \cup (D_{\adj}/ \Gamma_{F})} L(2 + z_{\rho} ; \chi_{\rho}) 
		\mathop{\int \cdots \int}_{\substack{v \in \Omega_{F}^{\infty} \\ 
		\Re(w_{\rho_v}) = -\frac{9}{4}: \rho_v \in J_v^c \\ \Re(w_{\rho_v}) = - 2 + \delta: \rho \in A_{v, \adj} / \Gamma_v}} \\
		&\prod_{\substack{v \in \Omega_{F}^{\infty} \\ \rho_v \in J_v \cup (B_v/ \Gamma_v) \\ \chi_{\rho_v} = 1}} \frac{1}{2 + w_{\rho_v}} 
		\mathcal{X}_{\Eff_{(X; \mathbf{A})}}(s L - \mathbf{2} - (\bm{z}, \bm{w})) \eta_{\mathbf{A}}(|\cdot|^{(\bm{z}, \bm{w})} \cdot \chi)
		d\frac{\bm{z}}{2 i \pi} d \frac{\bm{w}}{2 i \pi}.
	\end{split}
	\end{equation}
	Here $z_{\rho} = 1 $ and $w_{\rho_v} = 1$ for $\rho \in I$ and $\rho_v \in J_v$ respectively.
	
	The integrand decays sufficiently fast in vertical strips for the same reason as for $\mathfrak{Z}(\chi ; \bm{s})$. Similarly, the sum over characters $\chi$ of \eqref{eq:Terms after moving the contours} converges absolutely.
	
	The integrand is a meromorphic function in $s$ and the only part of the integrand which depends on $s$ is the factor $\mathcal{X}_{\Eff_{(X; \mathbf{A})}}(s L - \mathbf{2} - (\bm{z}, \bm{w}))$ so let us analyze this.
	
	By the identity \eqref{eq:log-canonical divisor} $K_{(X ; \mathbf{A})} = (-\mathbf{1}, \mathbf{0})$ we have 
	\begin{equation*}
		\begin{split}
			s L - \mathbf{2} - (\bm{z}, \bm{w}) = &(s - a)L + K_{(X; \mathbf{A})} + aL - 
			\sum_{\rho \in D_{\adj}} 2\delta [Z_{\rho}] - 
			\sum_{\substack{v \in \Omega_{F}^{\infty} \\ \rho_v \in A_{v, \adj}}} 2\delta [Z_{\rho_v}] \; + \\
			&\sum_{\rho \in I^c} (- z_{\rho} - 1) [Z_{\rho}] + \sum_{\rho \in D_{\adj}/ \Gamma_{F}} (- z_{\rho} - 1 + 2 \delta) [Z_{\rho}] \; + \\
			&\sum_{\substack{v \in \Omega_{F}^{\infty} \\ \rho_v \in J_v^{c}}} ( - w_{\rho_v} - 2) [Z_{\rho_v}] + \sum_{\substack{v \in \Omega_{F}^{\infty} \\ \rho_v \in A_{v, \adj} / \Gamma_v}} ( - w_{\rho_v} - 2 + 2 \delta) [Z_{\rho_v}].
		\end{split}
	\end{equation*}
	Let $C_{I, (J_v)_v} \subset \Eff_{(X; \mathbf{A})}$ be the face generated by all $[Z_{\rho}]$ for $\rho \not \in I$ and $[Z_{\rho_v}]$ for $\rho_v \not \in J_v$. We clearly have $C \subset C_{I, (J_v)_v}$. Moreover we have $C = C_{I, (J_v)_v}$ if and only if $I = \emptyset, J_v = \emptyset$ since $[Z_{\rho}], [Z_{\rho_v}] \not \in C$ if $\rho \in E/\Gamma_F, \rho_v \in B_v / \Gamma_v$ by the definition of $E$ and $B_v$ respectively.
	
	Let $b_{I, (J_v)_v}$ be the codimension of $C_{I, (J_v)_v}$. By the above we have $b \geq b_{I, (J_v)_v}$ with equality if and only $I = \emptyset, J_v = \emptyset$
	
 	Note that in the integral \eqref{eq:Terms after moving the contours} we have $ \Re(- z_{\rho} - 1), \Re( z_{\rho} - 1 + 2 \delta), \Re( - w_{\rho_v} - 2), \Re( - w_{\rho_v} - 2 + 2 \delta) > 0$ for $\rho \in I^c, \rho \in E / \Gamma_{\adj}, \rho_v \in J_v^c, \rho_v \in A_{v, \adj} / \Gamma_v$ respectively. 
 	
 	In the notation of Lemma \ref{Poles of characteristic functions of cones} the linear function $s L - \mathbf{2} - (\bm{z}, \bm{w})$ is thus equal to a sum of $(s - a)L$ and an element in $C_{I, (J_v)_v, \C}^{\circ}$. The lemma then implies that there exists a $\varepsilon > 0$ such $\mathcal{X}_{\Eff_{(X; \mathbf{A})}}(s L - \mathbf{2} - (\bm{z}, \bm{w}))$ has a meromorphic continuation to the region $\Re(s) > a - \varepsilon$ which is holomorphic except for a single pole at $s = a$ of order at most $b_{I, (J_v)_v}$. It follows that the same is true for the integral \eqref{eq:Terms after moving the contours}. 
 	
 	The integral $\xi(f, s, \bm{z}, \bm{w})$ converges and has a meromorphic continuation for the same reason.
 	
 	 By moving the contours corresponding to $z_{\rho}$ for $\rho \in D_{\adj}/\Gamma_F$ and $w_{\rho_v}$ for $\rho_v \in A_{v, \adj}/ \Gamma_v$ back we see that the contribution of \eqref{eq:Terms after moving the contours} for the case $I = J_v = \emptyset$ is equal to the value at $z_{\rho} = -1$ for $\rho \in E/\Gamma_v$ and $w_{\rho_v} = -2$ for $\rho_v \in B_v/ \Gamma_v$ of the meromorphic continuation of
 	\[
 	\prod_{\rho \in E / \Gamma_F} (z_{\rho} + 1) \prod_{\substack{v \in \Omega_F \\ \rho_v \in B_v/\Gamma_v}} (w_{\rho_v} + 2)\mathop{\int \cdots \int}_{\substack{\rho \in D_{\adj} / \Gamma_F \\ \Re(z_{\rho}) = 0}}\mathop{\int \cdots \int}_{\substack{v \in \Omega_F^{\infty} \\ \rho_v \in A_{v, \adj} /\Gamma_v \\ \Re(w_{\rho_v}) = 0}} 
 	\hat{\mathcal{G}}(\chi \cdot | \cdot|_{\Theta_{\mathbf{A}}}^{(\bm{z}, \bm{w}) + (\bm{z}_{\adj}, \bm{w}_{\adj})}; s L) d\frac{\bm{z}_{\adj}}{2 i \pi} d \frac{\bm{w}_{\adj}}{2 i \pi}.
 	\] 
 	
 	We can switch the limit with the sum over characters by absolute convergence.
\end{proof}
\begin{corollary}\label{cor:Meromorphic continuation of height zeta function}
	The sum $Z_{\mathfrak{a}}(s L, f)$ is absolutely convergent for $s > a$ and has a meromorphic continuation to the half place $\Re(s) > a - \varepsilon$ with at most a single pole at $s = a$ of order $b$.
\end{corollary}
\begin{proof}
	By Remark \ref{rem:Can choose Mellin of phi to be non-zero} we may assume that $(\mathcal{M} \varphi)(-s L) \neq 0$ for any particular $s$ by changing $\varphi$. This then follows from Lemma \ref{lem: Moving the contours}.
\end{proof}
\section{The leading constant}
In this section we will finish our analysis by determining the value of $\lim_{s \to a}(s - a)^bZ_{\mathfrak{a}}(s L, f)$. This is closely related to the leading constant. We keep the notation of the last section.
\subsection{Geometry}
Consider the adelic tori $\Theta_{\mathbf{A}_{\adj}} := (\G_m^{D_{\adj}}, (\G_m^{D_{\adj} \cup A_{v, \adj}})_v)$ and $\Theta_{\mathbf{B}} := (\G_m^{E}), (\G_m^{E \cup B_v})_v)$. There are clearly split exact sequence of locally compact groups
\begin{align}\label{eq:Split exact sequence A, B and A_adj adelic}
	1 \to \Theta_{\mathbf{B}}(\A_F) \to \Theta_{\mathbf{A}}(\A_F) \to \Theta_{\mathbf{A}_{\adj}}(\A_F) \to 1 \\
	\label{eq:Split exact sequence A, B and A_adj global}
	1 \to \G_m^E(F) \to \Theta(F) \to \G_m^{D_{\adj}}(F) \to 1
\end{align} 

Let $\mathbf{T}_{\NS L} := \ker(\mathbf{T}_{\NS} \to \Theta_{\mathbf{A}_{\adj}}) = \mathbf{T}_{\NS} \cap \Theta_{\mathbf{B}}$.

Let $N_L \subset N$, resp. $N_{B_v}$ be the full sublattice of $N$ generated by the rays $\rho$, resp. the rays $\rho_v$, for $\rho \in E$, resp. $\rho_v \in E \cup B_v$. Let $M_L := \Hom(N_L, \Z)$, resp. $M_{B_v} := \Hom(N_{A_v}, \Z)$, be the dual lattice. Let $M_{\adj} := \ker(M \to M_{L})$, resp. $M_{v, \adj} := \ker(M \to M_{B_v})$. 

The following lemma can be compared with the diagram \cite[p.~3225]{Batyrev1996Height}
\begin{lemma}\label{lem:Diagram of Iitaka fibration tori}
	Let $\Pic U_{L, B_v} := \coker(\Z[D_{\adj} \cup A_{v, \adj}] \to \Pic U_{A_v, \overline{F_v}})$. The following commutative diagram is well-defined. It has exact rows, exact columns and the sequence starting at the top left end ending at the bottom right is exact.
	\[\begin{tikzcd}
		& 0 & 0 \\
		0 & {M_{v, \adj}} & {\Z[D_{\adj} \cup A_{v, \adj}]} \\
		0 & M & {\Z[D \cup A_v]} & {\Pic U_{A_v, \overline{F_v}}} & 0 \\
		0 & {M_{B_v}} & {\Z[E \cup B_v]} & {\Pic U_{L, B_v, \overline{F_v}}} & 0 \\
		& 0 & 0 & 0
		\arrow[from=1-2, to=2-2]
		\arrow[from=1-3, to=2-3]
		\arrow[from=2-1, to=2-2]
		\arrow[from=2-2, to=2-3]
		\arrow[from=2-2, to=3-2]
		\arrow[from=2-3, to=3-3]
		\arrow[bend left=25, from=2-3, to=3-4]
		\arrow[from=3-1, to=3-2]
		\arrow[from=3-2, to=3-3]
		\arrow[from=3-2, to=4-2]
		\arrow[from=3-3, to=3-4]
		\arrow[from=3-3, to=4-3]
		\arrow[from=3-4, to=3-5]
		\arrow[from=3-4, to=4-4]
		\arrow[from=4-1, to=4-2]
		\arrow[from=4-2, to=4-3]
		\arrow[from=4-2, to=5-2]
		\arrow[from=4-3, to=4-4]
		\arrow[from=4-3, to=5-3]
		\arrow[from=4-4, to=4-5]
		\arrow[from=4-4, to=5-4]
	\end{tikzcd}\]
	
	Moreover, $M_L = M_{B_v}, M_{\adj} = M_{v, \adj}$ and $N_L = N_{B_v}$ for all places $v$.
\end{lemma}
\begin{proof}
	Exactness of the columns is by definition.
	
	Exactness of the second row is a special case of \eqref{eq:Exact sequence invertible global sections and Picard group} since $\overline{F}[U_{A_v}]^{\times} \subset \overline{F}[U]^{\times} = \overline{F}^{\times}$.
	
	Note that $M_{B_v}$ is the image of $M$ along the projection $\Z[D \cup A_v] \to \Z[E \cup B_v]$ by definition. It follows that $M_{v, \adj} = \Z[D_{\adj} \cup A_{v, \adj}] \cap M$. This implies that the first row is exact. The exactness of the remaining sequence then follows from the snake lemma.
	
	If $M_{\adj} \neq M_{v, \adj}$ then by the above there exists an element $m \in M$ whose image in $\Z[D \cup A_v]$ lies in $\Z[D_{\adj} \cup A_v]$ but not in $\Z[D_{\adj} \cup A_{v, \adj}]$. This implies that the rational function defined by $m$ has a pole or zero at a divisor of $B_v$ and all other poles and divisors on $U_{A_v}$ are contained in $D_{\adj} \cup A_{v, \adj}$.
	
	Taking Galois norms and divisor classes we find that the image in $\Pic(X, \mathbf{A})$ of a non-trivial linear combination of elements in $B_v/ \Gamma_v$ is contained in $C_{\R}$. Since $C$ is a face of $\Eff_{(X ; \mathbf{A})}$ the pullback along $\Z[A_v / \Gamma_v] \to \Pic(X ; \mathbf{A})$ is a face of the standard cone $\R_{\geq 0}[A_v / \Gamma_v]$. But this implies that the divisor class of an element of $B_v$ lies in $C$, which contradicts the definition of $C$. 
	
	We conclude that $M_{\adj} = M_{v, \adj}$. The rest of the lemma follows.
\end{proof}

Let $T_L \subset T$ be the torus dual to $M_L = M_{B_v}$. Let $\Sigma_L$, resp. $\Sigma_{B_v}$ be the subfan of $\Sigma$ of the cones all of whose rays are elements of $E$, resp. $E \cup B_v$. This defines a regular fan on $N_L$ and thus a smooth toric subvariety $T_L \subset U_L \subset U$, resp. $U_L \subset U_{L, B_v} \subset U_{A_v}$ since Galois descent is effective for closed immersions.

Note that the $\Pic U_{L, B_v, \overline{F_v}}$ appearing in Lemma \ref{lem:Diagram of Iitaka fibration tori} is the geometric Picard group of $U_{L, B_v}$ by \eqref{eq:Exact sequence invertible global sections and Picard group}.

Let $(U_L; \mathbf{B}) := (U_L, (U_{L, B_v})_v)$. We have the following properties.
\begin{lemma}\label{lem:Properties of U_L} \hfill
	\begin{enumerate}
		\item $\overline{F}[U_L]^{\times} = \overline{F}^{\times}$.
		\item $\HH^0((U_L; \mathbf{B}), \mathcal{O}_{(U_L; \mathbf{B})}) = F$.
		\item $\mathbf{T}_{\NS L}$ is the N\'eron-Severi torus of $(U_L; \mathbf{B})$.
		\item The image of $a L$ in $(\Pic (U_L; \mathbf{B}))_{\R}$ is equal to $- K_{(U_L; \mathbf{B})}$.
	\end{enumerate}
\end{lemma}
\begin{proof}
	For the first statement use that $E$ generates $N_L$ by definition.
	
	For the second statement it suffices to show that $N_L$ is generated by non-negative linear combinations of elements of $E$ and $B_v$ for all places $v$. If this was not true there would exist a non-zero $m \in M_L$ which maps to a non-zero element $\Z_{\geq 0}[E \cup \bigcup_v B_v] \subset \Z_{\geq 0}[D \cup \bigcup_v A_v]$. Taking Galois norms we get a non-zero element $f \in \Z_{\geq 0}[D / \Gamma_v \cup \bigcup_v A_v/ \Gamma_v]$. 
	
	The element $f$ maps to $0$ in $\Pic U_{L, B_v, \overline{F_v}}$ for all $v$, so the image in $\Pic(X; \mathbf{A})$ must be contained in $C_{\R}$ by Lemma \ref{lem:Diagram of Iitaka fibration tori}. Since $f$ is non-negative we must have $f \in C$. But $f \in \Z_{\geq 0}[E / \Gamma_v \cup \bigcup_v B_v/ \Gamma_v]$ so this contradicts the definition of $E$ and $A_v$. We conclude by contradiction.
	
	Because $\Pic U_{L, B_v, \overline{F_v}} = \coker(\Z[D_{\adj} \cup A_{v, \adj}] \to \Pic U_{A_v, \overline{F_v}})$ we have that its dual group of multiplicative type is $\ker(T_{\NS A_v} \to \G_m^{D_{\adj} \cup A_{v, \adj}})$. This agrees with the definition of $\mathbf{T}_{\NS L}$.
	
	The last statement follows from the definitions of $a, D_{\adj}, A_{v, \adj} $ and Lemma \ref{lem:Diagram of Iitaka fibration tori}.
\end{proof}

Let $\mathbf{Y}_L := \mathbf{Y} \cap (\A^{E}, (\A)^{E \cup B_v})$. It follows directly from the definitions that the restriction of $\pi$ to $\mathbf{Y}_L$ factors through $(U_L, \mathbf{B})$ and that $\mathbf{T}_{\NS L}$ acts on $\mathbf{Y}_{L}$. It is clear from Proposition \ref{prop:Universal torsor of toric variety} that $\pi: \mathbf{Y}_L \to (U_L, \mathbf{B})$ is a universal torsor.

Let $\mathcal{R} := \coker(T_{\NS}(F) \to \G_m^{D_{\adj}}(F))$. For each $\mathfrak{a} \in \HH^1(F, T_{\NS})$ and $\alpha \in \mathcal{R}$ the map $\pi_{\mathfrak{a}, \alpha}: \mathbf{Y}_L \to (U_L, \mathbf{B}) \cdot \pi_{\mathfrak{a}}(\alpha): x \to \pi_{\mathfrak{a}}(x \alpha)$ is a universal torsor since it is just a shift of $\pi$.
\subsection{Integral over universal torsor}
We can rephrase the conditions $\chi_{\rho} = 1$ for $\rho \in E/\Gamma_F$ and $\chi_{b_v} = 1$ for all $b_v \in B_v/\Gamma_v$ and $v \in \Omega_{F}^{\infty}$ in the definition \eqref{eq:Definition of xi} of $\xi(f, s, \bm{z}, \bm{w})$ as the statement that $\chi$ is trivial when restricted to $\Theta_{\mathbf{B}}(\A_F)$. This is equivalent to $\chi \in \Theta_{\mathbf{A}_{\adj}}(\A_F)^{\vee}$ by \eqref{eq:Split exact sequence A, B and A_adj adelic}.

Completely analogously to the analysis in \S\ref{sec:Poisson summation} one works out that
\[
\begin{split}
	&\int_{(\Theta_{\mathbf{A}_{\adj}}(\A_F)/\G_m^{D_{\adj}}(F))^{\vee}} \hat{\mathcal{G}}(\chi \cdot | \cdot|_{\Theta_{\mathbf{A}}}^{(\bm{z}, \bm{w})}; s L) d\tau_{\Theta_{\mathbf{A }_{\adj}}^{\vee}}(\chi) = \\
	&c_{\mathbf{A}_{\adj}}^{-1} \sum_{\substack{\chi \in (\Theta_{\mathbf{A}}(\A_F)^1/\Theta(F))^{\vee} \\ \chi_{\rho} = 1: \rho \in D/\Gamma_F \\ \chi_{\rho_v} = 1: \rho_v \in B/\Gamma_v}}
	\mathop{\int \cdots \int}_{\substack{\rho \in E_{\adj} / \Gamma_F \\ \Re(z_{\rho}) = 0}}\mathop{\int \cdots \int}_{\substack{v \in \Omega_F^{\infty} \\ \rho_v \in A_{v, \adj} /\Gamma_v \\ \Re(w_{\rho_v}) = 0}} 
	\hat{\mathcal{G}}(\chi \cdot | \cdot|_{\Theta_{\mathbf{A}}}^{(\bm{z}, \bm{w}) + (\bm{z}_{\adj, \bm{w}_{\adj}})}; s L) d\frac{\bm{z}_{\adj}}{2 i \pi} d \frac{\bm{w}_{\adj}}{2 i \pi}.
\end{split}
\]
Where $c_{\mathbf{A}_{\adj}} := \prod_{v \in \Omega_{F}^{\infty}} 2^{\# \{a_v \in A_{v, \adj} / \Gamma_v : F_{a_v} = \R\}} (2 \pi)^{\# \{a_v \in A_{v, \adj} / \Gamma_v : F_{a_v} = \C\}}$.

We can substitute this into \eqref{eq:Definition of xi} to deduce that 
\begin{equation}\label{eq:Xi function is integral over characters}
	\begin{split}
		\xi(s, f, \bm{z}, \bm{w}) = c_{\mathbf{B}}^{-1} \int_{(\Theta_{\mathbf{A}_{\adj}}(\A_F)/\G_m^{D_{\adj}}(F))^{\vee}} \hat{\mathcal{G}}(\chi \cdot | \cdot|_{\Theta_{\mathbf{A}}}^{(\bm{z}, \bm{w})}; s L) d\tau_{\Theta_{\mathbf{A}_{\adj}}^{\vee}}(\chi).
	\end{split}
\end{equation}
Where $c_{\mathbf{B}} := \prod_{v \in \Omega_{F}^{\infty}} 2^{\# \{b_v \in B_v / \Gamma_v : F_{b_v} = \R\}} (2 \pi)^{\# \{b_v \in B_v / \Gamma_v : F_{v_v} = \C\}} = c_{\mathbf{A}}/c_{\mathbf{A}_{\adj}}$.

It follows from the exact sequences \eqref{eq:Split exact sequence A, B and A_adj adelic} and \eqref{eq:Split exact sequence A, B and A_adj global} that the following sequence is exact
\[
1 \to \Theta_{\mathbf{B}}(\A_F)\G_m^{D_{\adj}}(F)/T_{\NS}(F) \to \Theta_{\mathbf{A}}(\A_F)/T_{\NS}(F) \to \Theta_{\mathbf{A}_{\adj}}(\A_F)/\G_m^{D_{\adj}}(F) \to 1
\]

The integral in \eqref{eq:Xi function is integral over characters} converges absolutely by Lemma \ref{lem: Moving the contours}, it also still converges after replacing $\varphi$ and $f$ by the $\varphi'$ and $g$ in Lemma \ref{lem:Validity of Poisson summation formula}. By applying the Poisson summation formula, i.e.~Theorem \ref{thm:Poisson summation} to the function $\Theta_{\mathbf{B}}(\A_F)\G_m^{D_{\adj}}(F)/T_{\NS}(F) \to \C: x \to \mathcal{G}(x ; sL) |x|_{\Theta_{\mathbf{A}}}^{(\bm{z}, \bm{w})}$ we deduce that $\xi(s, f, \bm{z}, \bm{w})$ is equal to
\begin{equation}\label{eq:inverse Poisson summation}
	\begin{split}
		&c_{\mathbf{B}}^{-1}\int_{(\Theta_{\mathbf{A}_{\adj}}(\A_F)/\G_m^{D_{\adj}}(F))^{\vee}} \hat{\mathcal{G}}(\chi \cdot | \cdot|_{\Theta_{\mathbf{A}}}^{(\bm{z}, \bm{w})}; s L) d\tau_{\Theta_{\mathbf{A}_{\adj}}^{\vee}}(\chi) = \\
		&c_{\mathbf{B}}^{-1}\int_{\Theta_{\mathbf{B}}(\A_F)\G_m^{D_{\adj}}(F)/T_{\NS}(F)} \mathcal{G}(x ; sL) |x|_{\Theta_{\mathbf{A}}}^{(\bm{z}, \bm{w})}d \tau_{\Theta_{\mathbf{B}}}(x) = \\
		&\sum_{\alpha \in \mathcal{R}} c_{\mathbf{B}}^{-1} \int_{\Theta_{\mathbf{B}}(\A_F)/T_{\NS L}(F)} \mathcal{G}(\alpha x; sL) |x|_{\Theta_{\mathbf{B}}}^{(\bm{z}, \bm{w})}d \tau_{\Theta_{\mathbf{B}}}(x).
	\end{split}
\end{equation}
Where $\mathcal{R} := \coker(T_{\NS L}(F) \to \G_m^{D_{\adj}}(F))$ and the last equality follows since $|x \alpha|_{\Theta_{\mathbf{A}}} = |x|_{\Theta_{\mathbf{B}}} |\alpha|_{\Theta_{\mathbf{A}_{\adj}}} = |x|_{\Theta_{\mathbf{B}}}$ by the product formula. This sum is thus absolutely convergent for $\Re(s)$ sufficiently large and $\Re(z_{\rho}) > -1$, $\Re(w_{\rho_v}) > -2$. 

We will relate these integrals to geometric integrals, the following lemma is inspired by \cite[Prop.~4.10]{Chambert-Loir2010Igusa}
\begin{lemma}\label{lem:Limit is geometric integral}
	Let $g: \mathbf{Y}_L(\A_F) \to \C$ be a compactly supported continuous function. We then have
	\[
	\begin{split}
		&c_{\mathbf{B}} \int_{\mathbf{Y}_L(\A_F)} g(x) dx = \\ 
		&\prod_{\rho \in E / \Gamma_F} \lim_{z_{\rho} \to -1}(z_{\rho} + 1) \prod_{\substack{v \in \Omega_{F}^{\infty} \\ w_{\rho_v} \in B_v / \Gamma_v}} \lim_{w_{\rho_v} \to - 2}(w_{\rho_v} + 2) \int_{\Theta_{\mathbf{B}}(\A_F)/T_{\NS L}(F)} g(x) |x|_{\Theta_{\mathbf{B}}}^{\mathbf{2} + (\bm{z}, \bm{w})} d \tau_{\Theta_{\mathbf{B}}}(x).
	\end{split}
	\]
\end{lemma}
\begin{proof}
	By a standard density argument we may assume that $g = \prod_v g_v$ where $g_v: Y_{L, B_v}(F_v) \to \R_{\geq 0}$ is a compactly supported function which is smooth, resp. locally constant, if $v$ is archimedean, resp. non-archimedean.
	
	For $v \in \Omega_{F}^{\infty}$ the local contribution from the right-hand side is
	\[
	\prod_{\substack{v \in \Omega_F \\ \rho_v \in A_v/\Gamma_v}} \lim_{w_{\rho_v} \to -2}(w_{\rho_v} + 2) \int_{\G_m^{E \cup B_v}(F_v)} g_v(x_v)|x_v|_{\G_m^{E \cup B_v}(F_v)}^{\mathbf{1} + (\bm{z}, \bm{w})} \prod_{\rho_v \in (E \cup B_v)/\Gamma_v} dx_{\rho_v}.
	\]
	This is equal to $c_{B_v} \int_{Y_{L,v}(F_v)_{B_v}} g(x_v) |x_v|_{\G_m^{E \cup B_v}}^{\mathbf{1} + \bm{z}} \prod_{\rho_v \in E/\Gamma_v} dx_{\rho_v}$ by \cite[Prop.~3.4]{Chambert-Loir2010Igusa}, where 
	\[
	c_{B_v} = 2^{\# \{b_v \in B_v / \Gamma_v : F_{b_v} = \R\}} (2 \pi)^{\# \{b_v \in B_v / \Gamma_v : F_{v_v} = \C\}}.
	\]
	
	For $v \in \Omega_{F}^{\text{fin}}$ the local contribution in the right-hand side is 
	\[
	\int_{\G_m^{E}(F_v)} g_v(x_v)|x_v|_{\G_m^{E}}^{\mathbf{1} +\bm{z}} \prod_{\rho_v \in E /\Gamma_v} \zeta_{F_{\rho_v}}(1) dx_{\rho_v}.
	\] 
	Taking the product over all places while remembering the convergence factor $\prod_{\rho \in E/ \Gamma_F} \zeta^*_{F_{\rho}}(1)$ shows that the right-hand side is the limit
	\[
	\begin{split}
		c_{\mathbf{B}} \prod_{\rho \in E / \Gamma_F} \lim_{z_{\rho} \to -1}
		&\prod_{v \in \Omega_{F}^{\infty}} \int_{Y_{L,v}(F_v)_{B_v}} g(x_v) |x_v|_{\G_m^{E \cup B_v}}^{\mathbf{1} + \bm{z}} \prod_{\rho_v \in E/\Gamma_v} dx_{\rho_v} \\
		&\prod_{v \in \Omega_{F}^{\text{fin}}} \int_{Y_L(F_v)} g_v(x_v)|x_v|_{\G_m^{E}}^{\mathbf{1} +\bm{z}} \prod_{\rho_v \in E /\Gamma_v} \frac{\zeta_{F_{\rho_v}}(1)}{\zeta_{F_{\rho_v}}(2 + z_{\rho_v})} dx_{\rho_v}.
	\end{split}
	\]
	The Euler product converges absolutely as long as $\Re(z_{\rho_v}) > - \frac{3}{2}$ by \cite[\S4.3.3]{Chambert-Loir2010Igusa}. We may thus switch the order of the Euler product and the limit over $z_{\rho}$.
	
	We can then conclude the lemma by the dominated convergence theorem and the fact that the $g_v$ are compactly supported.
\end{proof}
This allows us to compute the residue of $\xi(f, s, \bm{z}, \bm{w})$.
\begin{lemma}\label{lem:Taking the limit of the z's and w's}
	We have the equality
	\[
	\begin{split}
		\prod_{\rho \in E / \Gamma_F} \lim_{z_{\rho} \to -1}(z_{\rho} + 1) \prod_{\substack{v \in \Omega_{F}^{\infty} \\ w_{\rho_v} \in B_v / \Gamma_v}} \lim_{w_{\rho_v} \to - 2}(w_{\rho_v} + 2) \xi(f, s, \bm{z}, \bm{w}) = \\
		\sum_{\alpha \in \mathcal{R}} \int_{\mathbf{Y}_L(\A_F)} \mathcal{H}_{\varphi}(\alpha x ; sL) f(\pi_{\mathfrak{a}}(\alpha x)) d x.
	\end{split}
	\]
\end{lemma}
\begin{proof}
	We first switch the order of the limits and the sum over $\mathcal{R}$ by dominated convergence.
	
	We can then apply Lemma \ref{lem:Limit is geometric integral} to each term of \eqref{eq:inverse Poisson summation}. This application is valid since after expanding the definition of $\mathcal{H}_{\varphi}(x ; sL)$ the integral over $\mathbf{T}_{\NS}(\A_F)/T_{\NS}(F)$ is monotone in $(\bm{z}, \bm{w})$ and the remaining integral over $\Theta_{\mathbf{B}}$ has as integrand $f(\bm{\pi}_{\mathfrak{a}}(\alpha x)) \varphi(y ||\alpha x||) |x|_{\Theta_{\mathbf{B}}}^{\mathbf{2} + (\bm{z}, \bm{w})}$ which is compactly supported on $\mathbf{Y}(\A_F)$ and thus also on the closed subspace $\mathbf{Y}_L(\A_F)$.
\end{proof}
\subsection{Anticanonical height}
\begin{lemma}\label{lem:residue of height zeta function:anticanonical}
	Assume that $K_{(X ; \mathbf{A})} = -aL$. In this case $\lim_{s \to a} (s - a)^bZ_{\mathfrak{a}}(s L, f)$ is equal to
	\[
	\int_{\mathbf{Y}(\A_F)} \mathbf{1}_{\Eff_{(X ; \mathbf{A})}}(\log \mathbf{H}(z)) f(\pi_{\mathfrak{a}}(z)) \mathbf{H}(z)^{-(a + 1)L} d \tau_{\mathbf{Y}}^{\emph{gauge}}(z).
	\]
\end{lemma}
\begin{proof}
	We have $D = E$ and $B_v = A_v$ for all $v$ since $K_{(X ; \mathbf{A})} = -aL$ so $\mathbf{Y}_L = \mathbf{Y}$ and $\mathcal{R} = 1$. We will study the integral in the right-hand side of Lemma \ref{lem:Taking the limit of the z's and w's}.
	
	Expanding out the definitions and making a change of variables as in \S\ref{sec:Fourier transform of G} we see that this is equal to 
	\[
		\int_{\mathbf{T}_{\NS}(\A_F)/T_{\NS}(F)} \mathbf{1}_{\Eff_{(X ; \mathbf{A})}}(\log |y|_{\mathbf{T}_{\NS}}) |y|_{\mathbf{T}_{\NS}}^{- sL - K_{(X ; \mathbf{A})}} d\tau_{\mathbf{T}_{\NS}}(y)\int_{\mathbf{Y}(\A_F)_{\mathbf{A}}} f(\pi_{\mathfrak{a}}(x)) \varphi(||x||) d x.
	\]
	The first integral is the only part which depends on $s$. This integral was computed in Lemma \ref{lem:Integral over Neron-Severi torus}, it is a constant times $\mathcal{X}_{\Eff_{(\mathbf{X}; \mathbf{A})}}(sL - K_{(X ; \mathbf{A})})$. The residue of $\mathcal{X}_{\Eff_{(\mathbf{X}; \mathbf{A})}}(sL - K_{(X ; \mathbf{A})})$ at $a$ is $\mathcal{X}_{\Eff_{(\mathbf{X}; \mathbf{A})}}(L)$ by Lemma \ref{Poles of characteristic functions of cones} since $aL = K_{(X ; \mathbf{A})}$. We deduce that $\lim_{s \to a} (s - a)^b (\mathcal{M} \varphi)(-sL) Z_{\mathfrak{a}}(s L, f)$ is equal to
	\[
	\int_{\mathbf{T}_{\NS}(\A_F)/T_{\NS}(F)} \mathbf{1}_{\Eff_{(X ; \mathbf{A})}}(\log |y|_{\mathbf{T}_{\NS}}) |y|_{\mathbf{T}_{\NS}}^{- L} d\tau_{\mathbf{T}_{\NS}}(y)\int_{\mathbf{Y}(\A_F)_{\mathbf{A}}} f(\pi_{\mathfrak{a}}(x)) \varphi(||x||) d x.
	\]
	
	Make the change of variables $z = y \cdot x$ and $t = y \cdot ||z||$. We then have the equalities of measures $dx = |y|_{\Theta}^{-\mathbf{1}} dz$ and $d\tau_{\mathbf{T}_{\NS}}(t) = d\tau_{\mathbf{T}_{\NS}}(y)$. 
	
	We have $|y|_{\Theta}^{-\mathbf{1}} = |y|_{\mathbf{T}_{\NS}}^{K_{(X ; \mathbf{A})}} = |y|_{\mathbf{T}_{\NS}}^{- a L}$ by \eqref{eq:log-canonical divisor}. By definition we also have $||z|| = y^{-1} \cdot ||x||$ and $\left|t \cdot ||z||^{-1}\right|_{\mathbf{T}_{\NS}} = \mathbf{H}(t \cdot z)$. The above integral is thus equal to
	\[
	\int_{\mathbf{T}_{\NS}(\A_F)} \varphi(t)
	\int_{\mathbf{Y}(\A_F)_{\mathbf{A}}/T_{\NS}(F)} \mathbf{1}_{\Eff_{(X ; \mathbf{A})}}(\log \mathbf{H}(t \cdot z)) f(\pi_{\mathfrak{a}}(z)) \mathbf{H}(t \cdot z)^{- (a + 1)L} d z d\tau_{\mathbf{T}_{\NS}}(t).
	\]
	
	Unfolding Construction \ref{cons:measure gauge form} one sees that $|dz| = \tau_{\mathbf{Y}}^{\text{gauge}}(z)$. Note now that $\mathbf{H}(t \cdot z) = |t|_{\mathbf{T}_{\NS}} \mathbf{H}(z)$ and $aL = -K_{\mathbf{X, \mathbf{A}}}$, if one combines this with Lemma \ref{lem:Compatibility of heights} we deduce that 
	\[
	\begin{split}
		&\int_{\mathbf{Y}(\A_F)_{\mathbf{A}}/T_{\NS}(F)} \mathbf{1}_{\Eff_{(X ; \mathbf{A})}}(\log \mathbf{H}(t \cdot z)) f(\pi_{\mathfrak{a}}(z)) \mathbf{H}(t \cdot z)^{- (a + 1)L} d z = \\
		&|t|_{\mathbf{T}_{\NS}}^{- a L}\int_{\mathbf{Y}(\A_F)/T_{\NS}(F)} \mathbf{1}_{\Eff_{(X ; \mathbf{A})}}(\log \mathbf{H}(z)) f(\pi_{\mathfrak{a}}(z)) \mathbf{H}(z)^{-(a + 1)L}\tau_{\mathbf{Y}}^{\text{gauge}}(z).
	\end{split}
	\]
	
	Combining the above we see that $(\mathcal{M} \varphi)(-aL) \lim_{s \to a} (s - a)^b Z_{\mathfrak{a}}(s L, f)$ is equal to
	\[
	\int_{\mathbf{T}_{\NS}(\A_F)} \varphi(t) |t|_{\mathbf{T}_{\NS}}^{- a L} d\tau_{\mathbf{T}_{\NS}}(t) \int_{\mathbf{Y}(\A_F)} \mathbf{1}_{\Eff_{(X ; \mathbf{A})}}(\log \mathbf{H}(z)) f(\pi_{\mathfrak{a}}(z)) \mathbf{H}(z)^{- L} d \tau_{\mathbf{Y}, \mathbf{H}^{-aL}}(z).
	\]
	
	By Remark \ref{rem:Can choose Mellin of phi to be non-zero} we may assume that $(\mathcal{M} \varphi)(-aL) \neq 0$. The lemma follows after dividing both sides by $(\mathcal{M} \varphi)(-aL) =  \int_{\mathbf{T}_{\NS}(\A_F)} \varphi(t) |t|_{\mathbf{T}_{\NS}}^{- a L} d\tau_{\mathbf{T}_{\NS}}(t)$.
\end{proof}
Combining this lemma, Lemma \ref{lem: Moving the contours} and \eqref{eq:Decomposition of the Height Zeta function} we deduce the following theorem.
\begin{theorem}\label{thm:Meromorphic continuation of height zeta function anticanonical}
	The series $Z(s L, f)$ converges absolutely for $\Re(s)$ sufficiently large. There exists an $\varepsilon > 0$ such that $Z(s L, f)$ has a meromorphic continuation to the region $\Re(s) > a - \varepsilon$. Moreover, this meromorphic continuation is holomorphic up to a single pole of order $b$ at $s = a$ and $\lim_{s \to a} (s - a)^bZ(s L, f)$ is equal to
	\[
	 \sum_{\mathbf{a} \in \HH^1(F, T_{\NS})} \int_{\mathbf{Y}(\A_F)} \mathbf{1}_{\Eff_{(X ; \mathbf{A})}}(\log \mathbf{H}(z)) f(\pi_{\mathfrak{a}}(z)) \mathbf{H}(z)^{- (a + 1)L} d \tau_{\mathbf{Y}}^{\emph{gauge}}(z).
	\]
\end{theorem}
This theorem implies that the integral Manin conjecture hold with our prediction for the constant, see Conjecture \ref{conj:Leading constant}.
\begin{theorem}\label{thm:Main theorem anticanonical case}
	Let $L \in \Pic(X ; \mathbf{A})$ be a line bundle such that $aL = - K_{(X ; \mathbf{A})}$ for $a > 0$. Let $b = \emph{rk} \Pic (X ; \mathbf{A})$. Let $H$ be a height function on $L$, and $\mathbf{H}$ an extension to $\mathbf{Y}$. Let $f: (X ; \mathbf{A})(\A_F) \to \C$ be any compactly supported continuous function.
	
	Assume that $c((X ; \mathbf{A}), f , H)$ is non-zero, which we define to be the finite sum
	\[
	\frac{1}{a (b - 1)!}\sum_{\mathbf{a} \in \HH^1(F, T_{\NS})} \int_{\mathbf{Y}(\A_F)} \mathbf{1}_{\Eff_{(X ; \mathbf{A})}}(\log \mathbf{H}(z)) f(\pi_{\mathfrak{a}}(z)) \mathbf{H}(z)^{- (a + 1)L} d\tau_{\mathbf{Y}}^{\emph{gauge}}(z).
	\]
	We then have as $T \to \infty$ that 
	\[\sum_{\substack{P \in T(F) \\ H(P) \leq T}} f(P) \sim c((X ; \mathbf{A}), f , H) T^a (\log T)^{b - 1}.\]
\end{theorem}
\begin{proof}
	By Proposition \ref{prop:Reduction of general Manin conjecture} it suffices to prove this for $\mathbf{H}$ and $f$ as in \S7.1 This case follows from Theorem \ref{thm:Meromorphic continuation of height zeta function anticanonical} and a Tauberian theorem, e.g.~\cite[Thm.~II.7.13]{Tenenbaum2015Analytic}.
\end{proof}

Theorem \ref{thm:Main theorem log anticanonical height} in the introduction follows from this theorem. Indeed, in that case $a = 1$. Also $f$ is non-negative and the measure $d\tau_{\mathbf{Y}}^{\text{gauge}}$ is supported on $\mathbf{Y}(\A_F)_{\mathbf{A}}$ by construction. So $c((X ; \mathbf{A}), f , H)$ will be non-zero as long as $f$ is non-zero on any element of the image of $\mathbf{Y}(\A_F)_{\mathbf{A}}$ along one of the $\pi_{\mathfrak{a}}$. This image is equal to $U(\A_F)_{\mathbf{A}}^{\Br_1}$ by Theorem \ref{thm:Descent theory for adelic varieties}.
	
We can deduce the following corollary, implying Theorem \ref{thm:Strong approximation in introduction}.
\begin{corollary}
	The algebraic Brauer-Manin obstruction is the only obstruction to strong approximation for $(X; \mathbf{A})$.
\end{corollary}
\begin{proof}
	Let $V \subset (X ; \mathbf{A})(\A_F)$ be an open such that $V \cap U(\A_F)_{\mathbf{A}}^{\Br_1} \neq \emptyset$. Let $f: (X ; \mathbf{A})(\A_F) \to \R_{\geq 0}$ be continuous function with compact support $\text{supp}(f) \subset V$ and such that $\text{supp}(f) \cap U(\A_F)_{\mathbf{A}}^{\Br_1} \neq \emptyset$. We have $c((X ; \mathbf{A}), f , H) \neq 0$ by the same reasoning as above so Theorem \ref{thm:Main theorem anticanonical case} implies that $T(F) \cap V \neq \emptyset$.
	
	 We deduce that the closure $\overline{T(F)}$ in $(X ; \mathbf{A})(\A_F)$ contains $U(\A_F)_{\mathbf{A}}^{\Br_1}$.
\end{proof}
\subsection{Other heights}
Let $L \in \Eff_{(X ; \mathbf{A})}^{\circ}$ now be any big line bundle and $H$ a height on it. For each $\alpha \in \mathcal{R}$ let $H_{L, \alpha}$ be the restriction of $H$ to $(U_L ; \mathbf{B})$ and $\mathbf{H}_{L, \alpha}$ the height corresponding to an extension of the adelic metric defining $H$ to the universal torsor $\mathbf{Y}_L \to (U_L ; \mathbf{B})$. This exists because of Lemma \ref{lem:Metrics can be extended}.

Computing the limit $\lim_{s \to a} (s - a)^b Z_{\mathfrak{a}}(s L, f)$ directly as in the proof of Lemma \ref{lem:residue of height zeta function:anticanonical} is not so simple, we will instead use a sieving argument which uses Lemma \ref{lem:residue of height zeta function:anticanonical} for $(U_L, \mathbf{B})$.
\begin{lemma}
	The limit $\lim_{s \to a} (s - a)^b Z_{\mathfrak{a}}(s L, f)$ is equal to
	\[
	\sum_{\alpha \in \mathcal{R}} \int_{\mathbf{Y}_L(\A_F)} \mathbf{1}_{\Eff_{(U_L ; \mathbf{B})}}(\log \mathbf{H}_{L, \alpha}(z)) f(\pi_{\mathfrak{a}}(z \alpha)) H_{L, \alpha}^{-(a + 1)}(z) d\tau_{\mathbf{Y}_L}^{\emph{gauge}}(z).
	\]
\end{lemma}
\begin{proof}
	Let $W := U \setminus \left( \cup_{Z_i \in D_{\adj}} Z_i\right)$ and for each place $v$ let $W_v := U_{A_v} \setminus \left( \cup_{Z_i \in D_{\adj} \cup B_c} Z_i\right)$ and let $\mathcal{W} \subset \mathcal{U}$ be the complement of the closures of the $Z_i$ for $Z_i \in D_{\adj}$. Let $S \subset S_1 \subset S_2 \subset \cdots \subset \Omega_F$ be a sequence of subsets of places such that $\cup_{n} S_n = \Omega_F$. For each $n \in \N$ and $v \in \Omega_F$ let $g_{v, n}: W_v(F_v) \to \R_{\geq 0}$ be a compactly supported function which is smooth, resp.~locally constant, if $v$ is archimedean, resp.~non-archimedean, satisfying the following properties: $g_{v, n-1} \leq g_{v,n} \leq f_v$, $g_{v, n}$ is the indicator function of $\mathcal{W}(\mathcal{O}_v)$ if $v \not \in S_n$ and $\lim_{n \to \infty} g_{v,n} = f_v$.
	
	Let $g_n = \prod_v g_{n,v}$ and $h_n = \left(\prod_{v \in S_n} f_v - \prod_{v \in S_n} g_{n, v}\right) \prod_{v \not \in S_n} f_v$. The function $g_n$, resp.~$h_n$, is compactly supported and continuous on $\prod_v W_v(F_v)$, resp.~$(X ; \mathbf{A})(\A_F)$. Note that $g_n \leq f \leq g_n + h_n$, $\lim_{n \to \infty} g_n = f$ and $\lim_{n \to \infty} h_n = 0$.
	
	Consider the zeta functions $Z_{\mathfrak{a}}(s L, g_n) = \sum_{P \in T(F)} g_n(P) H(P)^{-s L}$ and $Z_{\mathfrak{a}}(s L, h_n) = \sum_{P \in T(F)} h_n(P) H(P)^{-s L}$. Both integrals converge absolutely for $\Re(s) > a$ by comparison with $Z_{\mathfrak{a}}(s L, f)$ and Corollary \ref{cor:Meromorphic continuation of height zeta function}. 
	
	 It follows from Lemma \ref{lem:Taking the limit of the z's and w's} and monotone convergence that
	 \[
	 \lim_{n \to \infty} \prod_{\rho \in E / \Gamma_F} \lim_{z_{\rho} \to -1}(z_{\rho} + 1) \prod_{\substack{v \in \Omega_{F}^{\infty} \\ w_{\rho_v} \in B_v / \Gamma_v}} \lim_{w_{\rho_v} \to - 2}(w_{\rho_v} + 2) \xi(h_n, s, \bm{z}, \bm{w}) = 0
	 \]
	 for all $\Re(s) \geq a$. It follows from Lemma \ref{lem: Moving the contours} that $\lim_{n \to \infty} \lim_{s \to a} (s - a)^b Z_{\mathfrak{a}}(s L, h_n) = 0$. As $g_n \leq f \leq g_n + h_n$ we find that \[
	 \lim_{s \to a} (s - a)^b Z_{\mathfrak{a}}(s L, f) = \lim_{n \to \infty} \lim_{s \to a} (s - a)^bZ_{\mathfrak{a}}(s L, g_n).
	 \]

	Note that $U_{L, B_v} \subset W_v$ is a closed immersion. The function $g_n$ is compactly supported on $\prod_v W_v(F_v)$ so for each $\alpha \in \mathcal{R}$ its restriction to the closed subset $(U_L ; \mathbf{B})(\A_F) \pi_{\mathfrak{a}}(\alpha)$ is also compactly supported. Moreover, since it is compactly supported its restriction to $(U_L ; \mathbf{B})(\A_F) \pi_{\mathfrak{a}}(\alpha)$ is zero for all but finitely many $\alpha \in \mathcal{R}$. 
	
	By Lemma \ref{lem:Basic properties adjoint rigid} we can apply Lemma \ref{lem:residue of height zeta function:anticanonical} to $(U_L, \mathbf{B})$ and deduce that for all $\alpha \in \mathcal{R}$ the limit $\lim_{s \to a}(s - a)^b\sum_{P \in T_L(F)\pi_{\mathfrak{a}}(\alpha)} g_n(P) H(P)^{-s L}$ is equal to 
	\[
	\int_{\mathbf{Y}_L(\A_F)} \mathbf{1}_{\Eff_{(U_L ; \mathbf{B})}}(\log \mathbf{H}_{L, \alpha}(z)) g_n(\pi_{\mathfrak{a}}(z \alpha)) H_{L, \alpha}^{-(a + 1)}(z) d\tau_{\mathbf{Y}_L}^{\text{gauge}}(z).
	\] 
	Summing over $\alpha \in \mathcal{R}$, noting that this is a finite sum for each $n$, we find a similar formula for $\lim_{s \to a} (s - a)^bZ_{\mathbf{a}}(s, g_n)$. We conclude the lemma by taking the increasing limit as $n \to \infty$.
\end{proof}

Combining this lemma, Lemma \ref{lem: Moving the contours} and \eqref{eq:Decomposition of the Height Zeta function} we deduce the following theorem.
\begin{theorem}
	The series $Z(s L, f)$ converges absolutely for $\Re(s)$ sufficiently large. There exists an $\varepsilon > 0$ such that $Z(s L, f)$ has a meromorphic continuation to the region $\Re(s) > a - \varepsilon$. Moreover, this meromorphic continuation is holomorphic up to a single pole of order $b$ at $s = a$ and $\lim_{s \to a} (s - a)^bZ(s L, f)$ is equal to
	\[
	\sum_{\mathbf{a} \in \HH^1(F, T_{\NS})} \sum_{\alpha \in \mathcal{R}} \int_{\mathbf{Y}_L(\A_F)} \mathbf{1}_{\Eff_{(U_L ; \mathbf{B})}}(\log \mathbf{H}_{L, \alpha}(z)) f(\pi_{\mathfrak{a}}(z \alpha)) H_{L, \alpha}^{-(a + 1)}(z) d\tau_{\mathbf{Y}_L}^{\emph{gauge}}(z).
	\]
\end{theorem}
The same reasoning as for Theorem \ref{thm:Main theorem anticanonical case} shows that this implies the following.

\begin{theorem}\label{thm: Main theorem for general heights}
	Let $L \in \Eff_{(X ; \mathbf{A})}^{\circ}$ be big. Let $H$ be a height function on $L$, and $\mathbf{H}_{L, \alpha}$ an extension of $H$ to the universal torsor $\pi_{\mathfrak{a}, \alpha}: \mathbf{Y}_L \to (U_L, \mathbf{B}) \cdot \pi_{\mathfrak{a}}(\alpha)$. Let $f: (X ; \mathbf{A})(\A_F) \to \C$ be any compactly supported continuous function.
	
	Assume that $c((X ; \mathbf{A}), f , H)$ is non-zero, which we define to be the sum
	\[
	\frac{1}{a (b - 1)!} \sum_{\mathbf{a} \in \HH^1(F, T_{\NS})} \sum_{\alpha \in \mathcal{R}} \int_{\mathbf{Y}_L(\A_F)} \mathbf{1}_{\Eff_{(U_L ; \mathbf{B})}}(\log \mathbf{H}_{L, \alpha}(z)) f(\pi_{\mathfrak{a}}(z \alpha)) H_{L, \alpha}^{-(a +1)}(z) d\tau_{\mathbf{Y}_L}^{\emph{gauge}}(z).
	\]
	We then have as $T \to \infty$ that 
	\[\sum_{\substack{P \in T(F) \\ H(P) \leq T}} f(P) \sim c((X ; \mathbf{A}), f , H) T^a (\log T)^{b - 1}.\]
\end{theorem}
We deduce Theorem \ref{thm:Main theorem all integral points} by using the partition of unity arguments at the beginning of \S6.
\begin{remark}
	If $L$ is adjoint rigid then $\pi_{\mathfrak{a}}(\alpha) \in T(F) = T_L(F) \subset (U_L, \mathbf{B})$. The maps $\pi_{\mathfrak{a}, \alpha}: \mathbf{Y}_L \to (U_L, \mathbf{B})$ as $\mathbf{a}$ ranges over all elements of $\HH^1(F, T_{\NS})$ and $\alpha$ ranges over all elements of $\mathcal{R}$ then correspond to all universal torsors of $(U_L, \mathbf{B})$ by Lemma \ref{lem:Universal torsor of open subvariety}. Theorem \ref{thm: Main theorem for general heights} thus agrees with Conjecture \ref{conj:Leading constant} in this case.
	
	 We note that the assumptions of Conjecture \ref{conj:Leading constant} are satisfied since $Y_{L, v}$ is the complement of a codimension $2$ subspace in the affine space $\A^{E \cup B_v}$. We thus have $\pi_1(Y_{L, v, \overline{F}_v}) = \pi_1(\A_{\overline{F}_v}^{E \cup B_v}) = 0$ and $\Br Y_{L, v, \overline{F}_v} = \Br \A_{\overline{F}_v}^{E \cup B_v} = 0$ by \cite[Thm.~3.7.1, Thm.~6.1.1]{Colliot2021Brauer}.
\end{remark}

\bibliographystyle{alpha} 
\bibliography{references_IPBHTV}
\end{document}